\documentclass[11pt]{amsart}

\usepackage[normalem]{ulem}
\usepackage{enumerate}
\usepackage{amssymb}
\usepackage{amsmath}
\usepackage{mathrsfs}
\usepackage{amsthm}
\usepackage{bm}
\usepackage{verbatim}
\usepackage{pgf}
\usepackage{float}
\usepackage[square,numbers]{natbib}
\usepackage{xxcolor}
\usepackage[colorlinks=true, linkcolor=blue, urlcolor=blue, citecolor = blue]{hyperref}
\usepackage{stmaryrd}
\usepackage[squaren]{SIunits}
\usepackage{graphicx}
\usepackage{tabularx}
\usepackage{chngcntr}
\usepackage{subcaption}
\usepackage{afterpage}
\usepackage{multirow}
\usepackage[normalem]{ulem}
\usepackage[papersize={8.5in,11in}]{geometry}
\newtheorem{theorem}{Theorem}[section]
\newtheorem{lemma}[theorem]{Lemma}

\theoremstyle{definition}

\theoremstyle{remark}

\newtheorem{remark}[theorem]{Remark}

\numberwithin{equation}{section}

\def\dd{\, \mathrm{d}}

\newcommand{\wsc}{\overset{\ast}{\rightharpoonup}}%

\newtheorem{innercustomgeneric}{\customgenericname}
\providecommand{\customgenericname}{}
\newcommand{\newcustomproblem}[2]{%
	\newenvironment{#1}[1]
	{%
		\renewcommand\customgenericname{#2}%
		\renewcommand\theinnercustomgeneric{##1}%
		\innercustomgeneric
	}
	{\endinnercustomgeneric}
}

\newcustomproblem{customprob}{\textbf{{\em Problem}}}
\newcustomproblem{customalgo}{Algorithm}
\newcommand*{\bqed}{\hfill\ensuremath{\blacksquare}}%

\usepackage{lineno}

\begin{document}
	
	\today
	
	
	\title[Dynamical shallow ice sheets]{On the dynamics of grounded shallow ice sheets: Modelling and analysis.}
	
	
	\author{Paolo Piersanti}
	\address{Department of Mathematics and Institute for Scientific Computing and Applied Mathematics, Indiana University Bloomington, 729 East Third Street, Bloomington, Indiana, USA}
	\email{ppiersan@iu.edu}
	
	\author{Roger Temam}
	\address{Department of Mathematics and Institute for Scientific Computing and Applied Mathematics, Indiana University Bloomington, 729 East Third Street, Bloomington, Indiana, USA}
	\email{temam@indiana.edu}
	
\begin{abstract}
In this article we formulate a model describing the evolution of thickness of a grounded shallow ice sheet. The thickness of the ice sheet is constrained to be nonnegative. This renders the problem under consideration an obstacle problem.
A rigorous analysis shows that the model is thus governed by a set of variational inequalities that involve nonlinearities in the time derivative and in the elliptic term, and that it admits solutions, whose existence is established by means of a semi-discrete scheme and the penalty method.
\end{abstract}

\maketitle


\tableofcontents

\section{Introduction}
\label{Intro}

The study of ice sheets melting and its correlation with the global warming problem has been attracting the interest of experts from all over the branches of science.

In this article we propose and study a mathematical model describing the evolution of the thickness of a grounded shallow ice sheet (from now on, simply \emph{shallow ice sheet}). The ice thickness of a shallow ice sheet evolves as a consequence of many factors like, for instance, the rate at which snow deposits, the rate at which melting occurs, as well as the velocity at which the glacier slides along the lithosphere. Since the ice thickness level is constrained to remain on or above the lithosphere at all times, the problem under consideration can be regarded as a time-dependent obstacle problem.

Grounded ice sheets and sea ice are two related but physically different problems. The mathematical literature related to ice sheets and glaciers and which deals with sea ice and glaciers is vast and abundant; in this direction we mention, for instance, the articles~\cite{Jouvet2015,Jouvet2015-2,Jouvet2014,Jouvet2013,Michel2014,Schoof2006,Schoof2006-2}, and the celebrated article by W.D. Hibler~\cite{Hibler1979} to which we refer later on, and which deals with sea ice. The local-in-time well-posedness of Hibler's model has recently been established in the article by Liu, Thomas and Titi~\cite{Titi2021}. The authors designed a regularizing scheme based on physical observations for deriving the sought local-in-time well-posedness.

Brandt et al. discussed in~\cite{Hieber2021} the local-in-time well-posedness of Hibler's model by means of a different regularization approach. The authors constructed the local strong solutions of the corresponding regularized system by investigating the analytic semi-group property of the parabolic operator of the regularized system in a bounded domain, and they also established the global well-posedness in time for initial data close to constant equilibria.

In this direction, we also mention the recently appeared pre-prints~\cite{Brandt2022-2,Brandt2022-3,Brandt2022-1}.

In the recent pre-print~\cite{Figalli2021}, Figalli, Ros-Oton and Serra studied the phase transition of ice melting to water as a Stefan problem. The results established in this article provide a refined understanding of the Stefan problem's singularities and answer some long-standing open questions in the field of free-boundary problems.

In 2002 Calvo, D\'iaz, Durany, Schiavi and V\'azquez~\cite{Diaz2002} published a paper discussing the evolution of the thickness of a shallow ice sheet as an obstacle problem. In their article, the authors only considered one spatial direction, they assumed the basal velocity to be smooth, and assumed the lithosphere to be flat. This seminal paper is, to our best knowledge, the very first record in the literature where a sound model describing the evolution of ice sheets thickness in the context of obstacle problems is discussed. This formulation was also exploited to justify the generation of fast ice streams in ice sheets flowing along soft and deformable beds~\cite{Diaz2007}. In this direction, we also cite the related previous paper~\cite{Diaz1999}.

Ten years later, in 2012, Jouvet and Bueler~\cite{JouvBuel2012} studied the steady (i.e., time-independent) version of the problem considered in~\cite{Diaz2002} where, this time, two spatial directions and a more general lithosphere topography were taken into account.

It appears that the ice thickness evolution as a time-dependent obstacle problem over a two-dimensional spatial domain has not been addressed in the literature yet.
The purpose of this article is exactly to address this problem under suitable assumptions.

This article is divided into four Sections (including this one). In Section~\ref{Sec:1} we present the main notations we shall be using throughout the manuscript, and we formally derive the governing equations for our model.

In Section~\ref{Sec:2} we formulate the ``penalized'' version corresponding to the obstacle problem introduced in Section~\ref{Sec:1}, and we establish the existence of solutions to this model by resorting to a series of new preparatory results as well as the Dubinskii's compactness theorem.

Finally, in Section~\ref{Sec:3}, we pass to the limit in the penalty parameter and we recover the actual model corresponding to the model we formally derived in Section~\ref{Sec:1}. This model, for which we also define the rigorous concept of solution, will take the form of a set of variational inequalities. The presence of the constraint, which adds a further nonlinearity to the two already considered (the first nonlinearity appears in the evolutionary term, while the second nonlinearity is the $p$-Laplacian), requires new strategies to be adopted in the limit passage in order to overcome the arising mathematical difficulties. In particular, we will see that vector-valued measures will play a critical role in the analysis.

\section{Notations and formal derivation of the model}
\label{Sec:1}

We denote by $(\mathbb{R}^n,\cdot)$ the $n$-dimensional Euclidean space equipped with its standard inner product.
Given an open subset $\Omega$ of $\mathbb{R}^n$ notations such as $L^2(\Omega)$, $H^m(\Omega)$, or $H^m_0 (\Omega)$, $m \ge 1$, designate the usual Lebesgue and Sobolev spaces, and the notation $\mathcal{D} (\Omega)$ designates the space of all functions that are infinitely differentiable over $\Omega$ and have compact support in $\Omega$. The notation $\left\| \cdot \right\|_X$ designates the norm in a normed vector space $X$. The dual space of a vector space $X$ is denoted by $X^\ast$ and the duality pair between $X^\ast$ and $X$ is denoted by $\langle \cdot, \cdot\rangle_{X^\ast,X}$.
Spaces of vector-valued functions are denoted with boldface letters. Lebesgue spaces defined over a bounded open interval $I$ (cf.~\cite{Leoni2017}), are denoted $L^p(I;X)$, where $X$ is a Banach space and $1 \le p \le \infty$. The notation $\left\|\cdot\right\|_{L^p(I;X)}$ designates the norm of the Lebesgue space $L^p(I;H)$.
Sobolev spaces defined over a bounded open interval $I$ (cf.~\cite{Leoni2017}), are denoted $W^{m,p}(I;X)$, where $X$ is a Banach space, $m\ge 1$ and $1 \le p \le \infty$.
The notation $\left\|\cdot\right\|_{W^{m,p}(I;X)}$ designates the norm of the Sobolev space $W^{m,p}(I;X)$.

A \emph{domain in} $\mathbb{R}^n$ is a bounded and connected open subset $\Omega$ of $\mathbb{R}^n$, whose boundary $\partial \Omega$ is Lipschitz-continuous, the set $\Omega$ being locally on a single side of $\partial \Omega$, viz.~\cite{PGCLNFAA}.

Let $\Omega$ be a domain in $\mathbb{R}^2$ and let $x=(x_1,x_2)$ be a generic point in $\overline{\Omega}$. Let $\bm{\nu}$ denote the outer unit vector field along the boundary of $\Omega$. Let $\nabla$ denote the gradient operator with respect to the coordinates $x_1$ and $x_2$, namely,
$$
\nabla=\left(\dfrac{\partial}{\partial x_1},\dfrac{\partial}{\partial x_2}\right),
$$
and let the symbol $(\nabla \cdot)$ denote the divergence operator, namely, for any vector field $\bm{v}=(v_1,v_2):\Omega \to \mathbb{R}^2$
$$
\nabla \cdot \bm{v}:= \dfrac{\partial v_1}{\partial x_1} + \dfrac{\partial v_2}{\partial x_2}.
$$
Let $T>0$ be given and let us consider the time interval $(0,T)$, any time instant in which is denoted by the letter $t$.

The lithosphere elevation is described by the function $b:\overline{\Omega} \to \mathbb{R}$. Positive values of $b$ are associated with altitudes above the sea level, whereas negative values of $b$ are associated with altitudes below the sea level.
We also assume, without loss of generality, that the lithosphere topography does not change throughout the observation time.

The elevation of the upper ice surface is described by the function $h:[0,T]\times \overline{\Omega} \to \mathbb{R}$. It is immediate to observe that 
\begin{equation}
	\label{constraint-h}
	h \ge b \quad\textup{ in }[0,T]\times \overline{\Omega}.
\end{equation}

The ice thickness $H:=h-b$ is thus nonnegative in $[0,T] \times \overline{\Omega}$. This consideration implies that the problem of studying the evolution of the sea ice thickness can be regarded as an obstacle problem, where the \emph{obstacle} is represented by the lithosphere.

This constraint implies the existence of a \emph{free boundary}~\cite{Brezis1972,SalsaCaff2005,Fichera77}. For all, or possibly almost every (a.e. in what follows) $t \in (0,T)$, we define the set $\Omega_{t}^{+}$ by:
\begin{equation}
	\label{ice-region}
	\Omega_{t}^{+}:=\{x \in \Omega; h(t,x) > b(x)\}=\{x \in \Omega; H(t,x)>0\}.
\end{equation}

The set $\Omega_{t}^{+}$ denotes the region of $\Omega$ which is covered with ice at the time instant $t$. The corresponding free boundary is the set
\begin{equation}
	\label{free-boundary}
	\Gamma_{f,t}:=\Omega \cap \partial\Omega_{t}^{+}.
\end{equation}

The variation of the ice thickness $H$ is influenced by two source terms: the surface-mass balance $a_s$, which is associated with ice accumulation and ablation rate, and the basal melting rate $a_b$. The function $a_b$ is equal to zero when the basal temperature is smaller than the ice melting point; otherwise, it is greater than zero.
The functions $a_s$ and $a_b$, in general, solely depend on the horizontal location $x$ and the surface elevation $h=h(t,x)$ at time $t \in [0,T]$. The terms $a_s$ and $a_b$ can thus be regarded as functions
\begin{align*}
	a_s&:[0,T] \times \overline{\Omega} \times \mathbb{R} \to \mathbb{R},\\
	a_b&:[0,T] \times \overline{\Omega} \times \mathbb{R} \to \mathbb{R}^+_{0}.
\end{align*}

In what follows we assume that $a_b \equiv 0$.
We define the function $a:[0,T] \times \overline{\Omega} \times \mathbb{R} \to \mathbb{R}$ by:
$$
a:=a_s,
$$
and we recall that this function is \emph{assumed} to be continuous in $[0,T] \times \overline{\Omega}$, strictly positive in a subset of $\Omega_{t}^{+}$ (where accumulation occurs), strictly negative in $\Omega_{t}^{-}:=\Omega \setminus (\Omega_{t}^{+} \cup \Gamma_{f,t})$ (where ablation occurs), and nonnegative in the complementary region of $\Omega_{t}^{+}$ characterised by ablation~\cite{GreveBlatter2009,JouvBuel2012,SchoofHewitt2013}. Ice sheets are incompressible, non-Newtonian, gravity driven flows~\cite{Fowler1997, GreveBlatter2009}. Ice flows from areas of $\Omega_{t}^{+}$ characterised by accumulation (i.e., regions where $a_s>0$) to areas of $\Omega_{t}^{+}$ characterised by ablation (i.e., regions where $a_s \le 0$).

Ice thickness is also influenced by the basal sliding velocity $\bm{U}_b$, which can be regarded as a given vector field in $\mathbb{R}^2$ solely depending on the horizontal position, i.e., 
$$
\bm{U}_b:\overline{\Omega} \to \mathbb{R}^2.
$$
and we recall that (cf., e.g., \cite{GreveBlatter2009}), when the ice base is frozen, we have $\bm{U}_b=\bm{0}$. 

The vector field $\bm{U}:[0,T] \times \overline{\Omega} \times \mathbb{R} \to \mathbb{R}^2$ denotes the horizontal ice flow velocity; the vector field $\bm{Q}$ denotes the \emph{volume flux}, defined as the integral of the horizontal ice flow velocity $\bm{U}$ with respect to the vertical direction (cf., e.g., equation~(5.47) of~\cite{GreveBlatter2009}), namely:
\begin{equation}
	\label{flux}
	\bm{Q}:[0,T] \times \overline{\Omega}\to\mathbb{R}^2 \qquad \textup{ and } \quad \bm{Q}:=\int_{b}^{h} \bm{U} \dd z.
\end{equation}

In the same spirit as Jouvet \& Bueler~\cite{JouvBuel2012}, we \emph{assume} that for each $t \in (0,T)$ there is no volume ice flow towards $\Omega_{t}^{-}$, i.e., 
\begin{equation}
	\label{normal-flux}
	\bm{Q} \cdot \bm{\nu} =0 \quad \textup{ on }\Gamma_{f,t}.
\end{equation}

In view of~\eqref{normal-flux}, we can naturally extend the volume flux $\bm{Q}$ by zero outside $\Omega_{t}^{+}$, i.e.,
\begin{equation}
	\label{outer-flux}
	\bm{Q}=\bm{0} \quad \textup{ in } \Omega_{t}^{-}.
\end{equation}

Besides, once again in the spirit of~\cite{JouvBuel2012}, we have that
\begin{equation}
	\label{BC}
	H=0 \textup{ on }\Gamma_{f,t} \qquad\textup{ and } \quad H=0 \textup{ on } \partial\Omega.
\end{equation}

The evolution of the ice thickness is governed by the \emph{ice thickness equation} (cf., e.g., equation~(5.55) in~\cite{GreveBlatter2009}), that we recall here below:
\begin{equation}
	\label{sie}
	\dfrac{\partial H}{\partial t}= -\nabla \cdot \bm{Q} + a.
\end{equation}

The free boundary $\Gamma_{f,t}$ and the region $\Omega_{t}^{+}$ are correctly described \emph{only} through a weak formulation of the problem under consideration. Prior to rigorously stating the weak formulation of the problem under consideration, we have to \emph{formally} recover the boundary value problem associated with~\eqref{sie}.
To this aim we \emph{first} fix a \emph{smooth enough} test function $v$ such that $v \ge b$ in $[0,T] \times \overline{\Omega}$ and, \emph{second}, we multiply~\eqref{sie} by $(v-h)$ and integrate over $\Omega$. As a result of this manipulation of~\eqref{sie}, the following identity holds for all $t \in (0,T)$: 
\begin{equation}
	\label{sie-2}
	\int_{\Omega} \dfrac{\partial H}{\partial t} (v-h) \dd x + \int_{\Omega} \nabla \cdot \bm{Q} (v-h) \dd x = \int_{\Omega} a (v-h) \dd x.
\end{equation}

By virtue of the fact that $h \ge b$ in $[0,T]\times \overline{\Omega}$, we can specialise $v=h$. The definition of $\Omega_{t}^{-}$ in turn implies that $\Omega = \Omega_{t}^{+} \sqcup \Omega_{t}^{-} \cup \Gamma_{f,t}$, where the symbol $\sqcup$ denotes the union of two disjoint sets.
An application of~\eqref{outer-flux} and of the Gauss-Green theorem to~\eqref{sie-2} gives
\begin{equation}
	\label{Gauss}
	\begin{aligned}
		&\int_{\Omega_{t}^{+}} \dfrac{\partial H}{\partial t} (v-h) \dd x +\int_{\Omega_{t}^{-}} \dfrac{\partial H}{\partial t} (v-h) \dd x\\
		&\quad+ \int_{\Omega_t^{+}} \nabla \cdot \bm{Q} (v-h) \dd x + \int_{\Omega_t^{-}} \nabla \cdot \bm{Q} (v-h) \dd x\\
		&\quad- \int_{\Omega_t^{+}} a (v-h) \dd x- \int_{\Omega_t^{-}} a (v-h) \dd x\\
		&=\int_{\Omega_{t}^{+}} \left(\dfrac{\partial H}{\partial t} + \nabla\cdot\bm{Q}-a\right) (v-h) \dd x + \int_{\Omega_{t}^{-}} \dfrac{\partial H}{\partial t} (v-h) \dd x\\
		&\quad+\underbrace{\int_{\partial\Omega_{t}^{-}} \bm{Q} \cdot \bm{\nu} (v-h) \dd \Gamma-\int_{\Omega_{t}^{-}} \bm{Q} \cdot \nabla(v-h) \dd x}_{=0 \textup{ by \eqref{outer-flux} and the fact that }\partial\Omega_{t}^{+}=\partial\Omega_{t}^{-}} - \int_{\Omega_{t}^{-}} a(v-h) \dd x.
	\end{aligned}
\end{equation}

In order to work out the following step of the boundary value problem recovery, let us recall that $\Omega_{t}^{-}$ denotes the region in $\Omega$ where the lithosphere is \emph{not} covered with ice, i.e., we have $H=0$ in $\Omega_{t}^{-}$. Moreover, the ice thickness $H$ in $\Omega_{t}^{-}$ \emph{either} does not change ($\partial H/\partial t = 0$) \emph{or} increases ($\partial H/\partial t > 0$); equivalently, the ice thickness $H$ \emph{cannot} diminish in $\Omega_{t}^{-}$.

On the one hand, in the region $\Omega_{t}^{+}$ the ice thickness evolution is governed by~\eqref{sie}, so that we have
\begin{equation}
	\label{Gauss-2}
	\int_{\Omega_{t}^{+}} \left(\dfrac{\partial H}{\partial t} + \nabla\cdot\bm{Q}-a\right) (v-h) \dd x =0, \quad \textup{ for all }t\in (0,T).
\end{equation}

On the other hand, specialising $v=h+\varphi$ in~\eqref{Gauss}, where $\varphi \in \mathcal{D}((0,T) \times \Omega)$, $\varphi \ge 0$ in $[0,T] \times \Omega$, recalling that $a \le 0$ in $\Omega_{t}^{-}$, and recalling the remark made above about the nonnegativeness of $\partial H/\partial t$ in $\Omega_{t}^{-}$, we obtain
\begin{equation}
	\label{Gauss-3}
	\int_{\Omega_{t}^{-}} \left(\dfrac{\partial H}{\partial t} -a\right) \varphi \dd x \ge 0,\quad\textup{ for all } t \in (0,T).
\end{equation}

Putting together~\eqref{sie-2}--\eqref{Gauss-3} gives
\begin{equation}
	\label{Gauss-4}
	\int_{\Omega} \left(\dfrac{\partial H}{\partial t}+\nabla \cdot \bm{Q}-a\right) \varphi \dd x \ge 0, \quad \textup{ for all }\varphi \in \mathcal{D}((0,T)\times \Omega), \varphi \ge 0 \textup{ in }[0,T]\times \overline{\Omega}.
\end{equation}

The latter can be straightforwardly changed into:
\begin{equation}
	\label{ineq-1}
	\begin{aligned}
		&\dfrac{\partial H}{\partial t}+\nabla \cdot \bm{Q} -a = 0, \quad\textup{ in } \Omega_t^{+}=\{x \in \Omega; H(t,x)>0\} \textup{ for all } t\in (0,T),\\
		\\
		&\dfrac{\partial H}{\partial t}+\nabla \cdot \bm{Q} -a \ge 0, \quad\textup{ in } \Omega_t^{-}=\Omega \setminus (\Omega_{t}^{+} \cup \Gamma_{f,t}) \textup{ for all } t\in (0,T).
	\end{aligned}
\end{equation}

Define $H_0:=h(0)-b$, and observe that $H_0 \ge 0$ by virtue of the observation made at the beginning of this Section (Section~\ref{Sec:1}). Putting together~\eqref{BC} and~\eqref{ineq-1} gives the sought free boundary value problem: \emph{Find} $H\ge 0$ \emph{satisfying:}
\begin{equation}
	\label{BVP}
	\begin{cases}
		&\dfrac{\partial H}{\partial t}+\nabla \cdot \bm{Q} -a = 0, \quad\textup{ in } \Omega_t^{+}=\{x \in \Omega; H(t,x)>0\} \textup{ for all } t\in (0,T),\\
		\\
		&\dfrac{\partial H}{\partial t}+\nabla \cdot \bm{Q} -a \ge 0, \quad\textup{ in } \Omega_t^{-}=\Omega\setminus(\Omega_{t}^{+} \cup (\Omega \cap \partial\Omega_{t}^{+})) \textup{ for all } t\in (0,T),\\
		\\
		&H(t,\cdot)=0, \quad\textup{ on }\partial\Omega \textup{ for all }t \in (0,T),\\
		\\
		&H(0,\cdot)=H_0 \ge 0.
	\end{cases}
\end{equation}

Multiplying the equation in the boundary value problem~\eqref{BVP} by $(h-b)$, integrating over $(0,T) \times \Omega$, and observing that $h(t,\cdot)=b$ in $\Omega_{t}^{-}$ gives:
\begin{equation}
	\label{weak-1}
	\int_{0}^{T} \int_{\Omega} \left(\dfrac{\partial H}{\partial t} +\nabla \cdot \bm{Q}-a\right)(h-b) \dd x \dd t =0.
\end{equation}

Let $v=v(t,x)$ be any test function such that $v(t,x) \ge b(x)$ for a.e. $(t,x) \in (0,T) \times \Omega$. Multiplying the equations in the boundary value problem~\eqref{BVP} by $(v-b)$, and integrating over $(0,T) \times \Omega$ gives:
\begin{equation}
	\label{weak-2}
	\int_{0}^{T} \int_{\Omega} \left(\dfrac{\partial H}{\partial t} +\nabla \cdot \bm{Q}-a\right)(v-b) \dd x \dd t \ge 0.
\end{equation}

Combining~\eqref{weak-1} and~\eqref{weak-2} gives:
\begin{equation}
	\label{weak-3}
	\int_{0}^{T} \int_{\Omega} \left(\dfrac{\partial H}{\partial t} +\nabla \cdot \bm{Q}-a\right)(v-h) \dd x \dd t \ge 0.
\end{equation}

The arbitrariness of the test function $v$ taken as above finally allows us to write down the weak variational formulation associated with the boundary value problem~\eqref{BVP} (recall that $H:=h-b$):\newline
\emph{Find} $H \ge 0$ \emph{satisfying the variational inequalities:}
\begin{equation}
	\label{weak-4}
	\int_{0}^{T} \int_{\Omega} \dfrac{\partial H}{\partial t} (v-h) \dd x\dd t -\int_{0}^{T} \int_{\Omega} \bm{Q} \cdot \nabla (v-h) \dd x \dd t \ge \int_{0}^{T}\int_{\Omega} a (v-h) \dd x\dd t,
\end{equation}
\emph{for all test functions} $v$ \emph{such that} $v \ge b$ \emph{for a.e.} $(t,x) \in (0,T) \times \Omega$, a\emph{nd satisfying the following boundary conditions along} $\partial\Omega$
$$
H=0\quad\mbox{ on }\partial\Omega,
$$
\emph{and satisfying the following initial condition}
$$
H(0,x) = H_0(x), \quad \mbox{ for a.e. }x\in \Omega,
$$
\emph{where} $H_0=h(0)-b$ \emph{is given, nonnegative and nonzero} (see~\eqref{constraint-h}).

Note that, at least for the time being, the rigorous concept of solution has not been defined yet as we did not specify the regularity of $H$ and of the data. This task will be postponed to forthcoming Sections.

Ice is viscous too; its viscosity is described in terms of the Glen power law (cf., e.g., equation~(4.16) in~\cite{GreveBlatter2009}) with ice softness coefficient $A(x,z)$ and exponent $2.8 \le p \le 5$. The attainable values for $p$ are suggested by laboratory experiments~\cite{GoldKohl2001}. Regarding the ice softness as a function is motivated by the idea of coupling the ice thickness equation with a thermodynamical model~\cite{GreveBlatter2009,Hooke2005}.

By equation~(5.84) in~\cite{GreveBlatter2009}, the horizontal ice flow velocity $\bm{U}$ can be expressed in terms of the elevation of the upper ice surface $h$ via the following formula:
\begin{equation}
	\label{velocity}
	\bm{U}(\cdot,\cdot,z)=-2 (\rho g)^{p-1} \left(\int_{b}^{z} A(s) (h-s)^{p-1} \dd s\right) |\nabla h|^{p-2} \nabla h +\bm{U}_b.
\end{equation}

Observe that, not only does the formula~\eqref{velocity} use Glen's power law, but also that it is only valid in a certain shallow limit derived from the Stokes model (cf., e.g., \cite{SchoofHewitt2013}).

Plugging formula~\eqref{velocity} into~\eqref{flux} gives:
\begin{equation}
	\label{flux-2}
	\begin{aligned}
		\bm{Q}&=-2 (\rho g)^{p-1}\left(\int_{b}^{h} \int_{b}^{z} A(s) (h-s)^{p-1} \dd s \dd z\right) |\nabla h|^{p-2} \nabla h +(h-b) \bm{U}_b\\
		&=-2 (\rho g)^{p-1}\left(\int_{b}^{h} A(s) (h-s)^{p} \dd s \right) |\nabla h|^{p-2} \nabla h +(h-b) \bm{U}_b,
	\end{aligned}
\end{equation}
where the latter equality is obtained by an integration by parts with respect to the variable $z$.

The weak form~\eqref{weak-4} is a degenerate extension of the $p$-Laplacian obstacle problem since, at each time $t \in [0,T]$, the thickness $H = h-b$ tends to zero at the free boundary $\Gamma_{f,t}$. 
This is the reason why the solution of~\eqref{weak-4} exhibits infinite gradients at the margin $\Gamma_{f,t}$, viz. \cite{Bueler2005} and~\cite{GreveBlatter2009}. 
A further reformulation that recovers nondegenerate $p$-Laplacian form is proposed next, cf.~\eqref{weak-u}. 
While this formulation involves only a flat obstacle, on the one hand, it acquires a ``tilt'' which does not preserve monotonicity of the variational form on the other hand.

For the sake of clarity, the magnitudes involved in the model are listed in the following table.

\begin{table}[h]
	\begin{tabular}{|c|p{5cm}|}
		\hline
		\textbf{Variable} & \textbf{Description} \\ \hline
		$A=A(s)$ & Ice softness in the Glen power law \\ \hline
		$b$ & Lithosphere elevation \\ \hline
		$g$ & Gravitational acceleration \\ \hline
		$h$ & Upper ice surface elevation \\ \hline
		$H:=h-b$ & Ice thickness \\ \hline
		$a_s$ & Accumulation rate function \\ \hline
		$a_b$ & Ablation rate function \\ \hline
		$a$ & $a=a_s - a_b$ Balance function \\ \hline
		$\bm{U}_b$ & Basal sliding velocity \\ \hline
		$\bm{U}$ & Horizontal ice flow velocity \\ \hline
		$\bm{Q}$ & Ice volume flux \\ \hline
		$\rho$ & Ice density \\ \hline
		$p$ & Index in the Glen power law \\
		\hline
	\end{tabular}
	\caption{Quantities entering the model}
	\label{magnitudes}
\end{table}

As already mentioned by Jouvet \& Bueler~\cite{JouvBuel2012}, the expression of $\bm{Q}$ in~\eqref{flux-2} exhibits a degenerate behaviour at the free boundary, in the sense that the gradient norm power blows up as $h$ gets close to the free boundary of $\Omega$. In this direction, we refer the reader to the article~\cite{SayagWorster2013} for a laboratory verification of the boundary degeneracy.

We also observe that the model we are considering follows the formulation originally proposed by J.W. Glen~\cite{Glen1970}, according to which ice was regarded as a viscous fluid.
The celebrated model proposed by Hibler, itself inspired by the article of M.D. Coon~\cite{Coon1974}, is on the one hand \emph{less precise} than Glen's model as the ice velocity $\bm{U}$ has a simplified expression. On the other hand, it achieves the goal of describing in the same instance the behaviours of ice as a solid and as a fluid.

The operation of averaging considerably simplifies the expression of the ice volume flux $\bm{Q}$, while the coupling of the shallow ice equation for the averaged ice thickness with the mechanical equation spares the effort of expressing the ice flow velocity $\bm{U}$ in terms of the ice thickness.

In order to overcome the difficulty arising as a result of the gradient degeneracy in the vicinity of the free boundary, we introduce the following transformation, originally suggested in~\cite{Raviart1967} (see also~\cite{Diaz2002}):
\begin{equation}
	\label{transfo}
	H:=u^{(p-1)/2p}.
\end{equation}

Observe that $H > 0$ in $[0,T] \times \overline{\Omega}$ if and only if $u > 0$ in $[0,T] \times \overline{\Omega}$, and that $H(t,x)=0$ if and only if $u(t,x)=0$, $x\in\overline{\Omega}$ and $t\in[0,T]$. As a result of the transformation~\eqref{transfo}, we \emph{first} obtain, formally,
\begin{equation}
	\label{dHdt}
	\dfrac{\partial H}{\partial t}=\dfrac{\partial}{\partial t}(|u|^{\frac{3p-1}{2p}-2} \, u),
\end{equation}
\emph{secondly}, we obtain
\begin{equation}
	\label{ice-soft}
	\int_{b}^{h} A(s) (h-s)^p \dd s=u^{\frac{(p+1)(p-1)}{2p}} \int_{0}^{1} A(x,b+u^{\frac{p-1}{2p}} \, s') (1-s')^p \dd s',
\end{equation}
and, \emph{third}, we obtain:
\begin{equation}
	\label{grad}
	|\nabla h|^{p-2} \nabla h=\left(\dfrac{p-1}{2p}\right)^{p-1} u^{\frac{(-p-1)(p-2)}{2p}} \left|\nabla u +u^{(p+1)/2p} \nabla b\right|^{p-2} u^{\frac{-p-1}{2p}} \left(\nabla u +u^{(p+1)/2p} \nabla b\right).
\end{equation}

For all $x \in \Omega$ and all $u=u(t,x) \in \mathbb{R}$, we define the vector field $\bm{\Phi}:\overline{\Omega} \times \mathbb{R} \to \mathbb{R}^2$ by:
\begin{equation}
	\label{Phi}
	\bm{\Phi}=\bm{\Phi}(x,u):=-\dfrac{2p}{p-1} u^{(p+1)/2p} \nabla b.
\end{equation}

Plugging~\eqref{Phi} into~\eqref{grad} gives:
\begin{equation}
	\label{grad-2}
	|\nabla h|^{p-2} \nabla h=\left(\dfrac{p-1}{2p}\right)^{p-1} u^{\frac{(-p-1)(p-2)}{2p}} \left|\nabla u -\bm{\Phi}\right|^{p-2} u^{\frac{-p-1}{2p}} \left(\nabla u -\bm{\Phi}\right).
\end{equation}

Finally, for all $x \in \Omega$ and all $u \in \mathbb{R}$, we define the vector field $\bm{\Psi}:\overline{\Omega} \times \mathbb{R} \to \mathbb{R}^2$ by,
\begin{equation}
	\label{psi}
	\bm{\Psi}=\bm{\Psi}(x,u):= (h-b) \bm{U}_b,
\end{equation}
and the function $\tilde{a}:[0,T] \times \overline{\Omega} \times \mathbb{R} \to \mathbb{R}$ by:
\begin{equation}
	\label{aNEW}
	\tilde{a}(t,x,u):=a(t,x,b+u^{(p-1)/2p}).
\end{equation}

Inserting~\eqref{ice-soft}--\eqref{aNEW} into~\eqref{flux-2} gives:
\begin{equation}
	\label{flux-3}
	-\bm{Q}=2\left(\rho g \dfrac{p-1}{2p}\right)^{p-1} \left[\int_{0}^{1} A(x,b+u^{(p-1)/2p} s') (1-s')^p \dd s'\right] |\nabla u -\bm{\Phi}|^{p-2} (\nabla u -\bm{\Phi}) -\bm{\Psi}.
\end{equation}

For the sake of brevity, we define the function $\mu:\overline{\Omega} \times \mathbb{R} \to \mathbb{R}$ as a function of $u$ by:
\begin{equation}
	\label{mu-def}
	\mu(x,u):=2\left(\rho g \dfrac{p-1}{2p}\right)^{p-1} \left[\int_{0}^{1} A(x,b+u^{(p-1)/2p} s') (1-s')^p \dd s'\right].
\end{equation}

Observe that plugging~\eqref{mu-def} into~\eqref{flux-3} gives:
\begin{equation}
	\label{flux-4}
	-\bm{Q}=\mu(x,u) |\nabla u -\bm{\Phi}|^{p-2} (\nabla u -\bm{\Phi}) -\bm{\Psi}.
\end{equation}

Plugging~\eqref{dHdt}--\eqref{aNEW} and~\eqref{flux-4} into the boundary value problem~\eqref{BVP}, gives that the \emph{new} unknown $u$ satisfies the following unilateral boundary value problem: \newline
\emph{Find} $u \ge 0$ \emph{defined on} $[0,T]\times \overline{\Omega}$ \emph{satisfying:}
\begin{equation}
	\label{BVPu}
	\begin{cases}
		&\dfrac{\partial}{\partial t}(|u|^{\frac{3p-1}{2p}-2} \, u)-\nabla \cdot \left(\mu(x,u) |\nabla u -\bm{\Phi}|^{p-2} (\nabla u -\bm{\Phi}) -\bm{\Psi}\right) =\tilde{a}, \textup{ in } \Omega_t^{+} \textup{ for all } t\in (0,T),\\
		\\
		&\dfrac{\partial}{\partial t}(|u|^{\frac{3p-1}{2p}-2} \, u)-\nabla \cdot \left(\mu(x,u) |\nabla u -\bm{\Phi}|^{p-2} (\nabla u -\bm{\Phi}) -\bm{\Psi}\right) \ge \tilde{a}, \textup{ in } \Omega_t^{-} \textup{ for all } t\in (0,T),\\
		\\
		&u(t,\cdot)=0, \quad\textup{ on }\partial\Omega \textup{ for all }t \in (0,T),\\
		\\
		&u(0,\cdot)=u_0:=H_0^{(2p/(p-1))} \ge 0.
	\end{cases}
\end{equation}

Similarly to~\eqref{weak-4}, the weak formulation associated with the boundary value problem~\eqref{BVPu} is expected to take the following form: \emph{Find} $u \ge 0$ \emph{satisfying the variational inequality:}
\begin{equation}
	\label{weak-u}
	\begin{aligned}
		&\int_{0}^{T} \int_{\Omega} \left(\dfrac{\partial}{\partial t}(|u|^{\frac{3p-1}{2p}-2} \, u)\right)(v-u) \dd x \dd t\\
		&\quad-\int_{0}^{T} \int_{\Omega} \mu(x,u) |\nabla u -\bm{\Phi}|^{p-2} (\nabla u -\bm{\Phi}) \cdot \nabla(v-u) \dd x\dd t 
		\\&\ge \int_{0}^{T} \int_{\Omega} \tilde{a} (v-u) \dd x \dd t +\int_{0}^{T}\int_{\Omega} \bm{\Psi} \cdot \nabla(v-u) \dd x \dd t,
	\end{aligned}
\end{equation}
\emph{for all test functions} $v$ \emph{such that} $v \ge b$ \emph{for a.e.} $(t,x) \in (0,T) \times \Omega$, a\emph{nd satisfying the following boundary conditions along} $\partial\Omega$
$$
u=0\quad\mbox{ on }\partial\Omega,
$$
and satisfying the following initial condition
$$
u(0,x) = u_0(x)=[H_0(x)]^{(2p/(p-1))}, \quad \mbox{ for a.e. }x\in \Omega.
$$

Once again, note that, at least for the time being, the rigorous concept of solution has not been defined yet. This task will be carried out in Section~\ref{Sec:2}, in which the function spaces where the solutions are going to be sought will be defined as well as the requirements a function has to meet in order to be regarded as a solution.

\section{Weak formulation of the unilateral boundary value problem}
\label{Sec:2}

Let $\Omega \subset \mathbb{R}^2$ be a domain. Let $2.8 \le p \le 5$ as suggested by the experimental results in~\cite{GoldKohl2001}. For the sake of brevity, let
\begin{equation}
	\label{alpha}
	\alpha:=\dfrac{3p-1}{2p},
\end{equation}
and observe that $1 < \alpha <2 <p+1$. By the Rellich-Kondrachov theorem (cf., e.g., Lions \& Magenes~\cite{LioMagI} and the references therein) the following chain of embeddings holds:
\begin{equation}
	\label{emb}
	W_0^{1,p}(\Omega) \hookrightarrow\hookrightarrow \mathcal{C}^0(\overline{\Omega}) \hookrightarrow L^\alpha(\Omega).
\end{equation}

Define the set 
\begin{equation}
	\label{K}
	K:=\{v \in W_0^{1,p}(\Omega); v \ge 0 \textup{ in }\overline{\Omega}\}.
\end{equation}

The weak formulation of the unilateral boundary value problem and the definition of the concept of solution will be recovered as a result of a constructive proof, based on the penalty method. The idea of using the penalty method to derive the existence of solutions of time-dependent contact problems was first used by Bock and Jarusek in~\cite{BockJar2009,BockJar2013} in the context of dynamic contact problems and was further developed in the recent article~\cite{BockJarSil2016}. A time-dependent obstacle problem in the case where the target space is finite-dimensional was also addressed in the context of the study of biological systems in~\cite{PWDT3D}. The finiteness of the target space dimension played a fundamental role to improve the results obtained by Bock, Jarusek and Silhavy in~\cite{BockJarSil2016}.

For the sake of simplicity, we make the following assumptions on the geometry of the problem:
\begin{enumerate}[$(H1)$]
	\item $b = 0$ in $\overline{\Omega}$, i.e., the lithosphere is flat;\label{H1}
	\item The ice softness $A$ is assumed to be independent of the ice sheet height $h$ and time. As a result, the function $\mu$ defined in~\eqref{mu-def} is independent of the ice sheet height $h$ as well. Moreover, there exist two positive constants $\mu_1$ and $\mu_2$ such that $\mu_1 \le \mu(x) \le \mu_2$ for all $x \in \overline{\Omega}$;\label{H2}
	\item The basal velocity $\bm{U}_b = \bm{0}$;\label{H3}
	\item The function $\tilde{a}$ defined in~\eqref{aNEW} is independent of the ice sheet height and is of class $W^{1,p}(0,T;\mathcal{C}^0(\overline{\Omega}))$.\label{H4}
\end{enumerate}

Let us now establish some preparatory results. The first result collects some properties of the negative part operator.

\begin{lemma}
	\label{lem:1}
	Let $\mathscr{O} \subset \mathbb{R}^m$, with $m\ge 1$ an integer, be an open set. The operator $-\{\cdot\}^{-}:L^2(\mathscr{O}) \to L^2(\mathscr{O})$ defined by
	$$
	f \in L^2(\mathscr{O}) \mapsto -\{f\}^{-}:=\min\{f,0\} \in L^2(\mathscr{O}),
	$$
	is monotone, bounded and Lipschitz continuous with Lipschitz constant equal to $1$.
\end{lemma}
\begin{proof}
	Let $f$ and $g$ be arbitrarily given in $L^2(\mathscr{O})$. In what follows, sets of the form $\{f \ge 0\}$ read
	$$
	\{x \in \mathscr{O}; f(x) \ge 0\}.
	$$
	
	Recall that the negative part of a function is also given by
	$$
	f^{-}=\dfrac{|f|-f}{2}.
	$$
	
	We have that
	\begin{align*}
		&\int_{\mathscr{O}}\left\{(-\{f\}^{-})-(-\{g\}^{-})\right\}(f-g) \dd x\\
		&\ge \int_{\mathscr{O}} |-\{f\}^{-}|^2 \dd x + \int_{\mathscr{O}} |-\{g\}^{-}|^2 \dd x -2\int_{\{f \le 0\} \cap \{g \le 0\}} (-\{f\}^{-})(-\{g\}^{-}) \dd x\\
		& \ge \int_{\{f \le 0\} \cap \{g \le 0\}} \left|(-\{f\}^{-})-(-\{g\}^{-})\right|^2 \dd x \ge 0,
	\end{align*}
	and the monotonicity is thus established.
	
	Let us show that the operator under consideration is bounded, in the sense that it maps bounded sets onto bounded sets. Let $\mathscr{F} \subset L^2(\mathscr{O})$ be bounded. By H\"older's inequality, we have that for each $f \in \mathscr{F}$,
	$$
	\sup_{\substack{v \in L^2(\mathscr{O})\\ v \neq 0}} \dfrac{\left|\int_{\mathscr{O}} (-\{f\}^{-}) v \dd x\right|}{\|v\|_{L^2(\mathscr{O})}}
	\le \|f\|_{L^2(\mathscr{O})},
	$$
	and the boundedness of $\mathscr{F}$ implies the boundedness of the supremum. The boundedness property is thus established.
	
	Finally, to establish the Lipschitz continuity property, evaluate
	\begin{align*}
		&\left(\int_{\mathscr{O}} |(-\{f\}^{-})-(-\{g\}^{-})|^2 \dd x\right)^{1/2}=\left(\int_{\mathscr{O}} \left|\dfrac{f-|f|}{2}-\dfrac{g-|g|}{2}\right|^2 \dd x\right)^{1/2}\\
		&=\dfrac{1}{2}\left(\int_{\mathscr{O}} |(f-g)-(|f|-|g|)|^2 \dd x\right)^{1/2} \le \dfrac{1}{2}\{\|f-g\|_{L^2(\mathscr{O})}+\||f|-|g|\|_{L^2(\mathscr{O})}\}\\
		&=\dfrac{1}{2}\|f-g\|_{L^2(\mathscr{O})}+\dfrac{1}{2}\left(\int_{\mathscr{O}}\left||f|-|g|\right|^2 \dd x\right)^{1/2}
		\le \dfrac{1}{2}\|f-g\|_{L^2(\mathscr{O})}+\dfrac{1}{2}\left(\int_{\mathscr{O}}|f-g|^2 \dd x\right)^{1/2}\\
		&=\|f-g\|_{L^2(\mathscr{O})},
	\end{align*}
	where the inequality holds thanks to the Minkowski inequality. In conclusion, we have shown that
	$$
	\|(-\{f\}^{-})-(-\{g\}^{-})\|_{L^2(\mathscr{O})} \le \|f-g\|_{L^2(\mathscr{O})},
	$$
	and the arbitrariness of $f$ and $g$ gives the desired Lipschitz continuity property. This completes the proof.
\end{proof}

The following lemma presents an inequality that was originally proved in 1973 by Pedersen in the context of $C^\ast$-algebras (cf., e.g., Proposition~1.3.8 of~\cite{Pedersen2018}) to show that the mapping $t \mapsto t^\beta$, with $0<\beta<1$, is monotone, where $t$ here denotes a generic operator. This inequality plays a crucial role in the forthcoming analysis. We hereby provide an alternate elementary proof proof which solely makes use of convex analysis in the special case where $t \in \mathbb{R}$.

\begin{lemma}
	\label{lem:2-0}
	Let $0<\beta \le 1$. Then
	$$
	\left||x|^\beta - |y|^\beta\right| \le |x-y|^\beta,\quad\textup{ for all } x, y \in \mathbb{R}.
	$$
\end{lemma}
\begin{proof}
	If $\beta=1$ or at least one between $x$ and $y$ is equal to zero then the conclusion is trivially true.
	So, let us consider the case where $0<\beta<1$, and assume, without loss of generality, that $|x| \ge |y| > 0$, and let
	$$
	t:=\dfrac{|y|}{|x|} \in [0,1].
	$$
	
	Let us first show that
	$$
	1-t^\beta \le (1-t)^\beta,\quad\textup{ for all } t\in [0,1].
	$$
	
	Since $0<\beta \le1$ and since $t \in [0,1]$, we have that $t^\beta \ge t$. Therefore, we have that two consecutive applications of this estimate lead to 
	$$
	(1-t)^\beta \ge 1-t \ge 1-t^\beta,
	$$
	thus proving the inequality:
	$$
	\left||x|^\beta - |y|^\beta\right|\le \left||x|-|y|\right|^\beta,\quad\textup{ for all }x,y\in\mathbb{R}.
	$$
	
	An application of the triangle inequality to the right-hand side of the estimate above and the monotonicity of the mapping $\xi \ge 0 \mapsto \xi^\beta \ge 0$ lead to the sought inequality and completes the proof.
\end{proof}

We hereby provide a complementary result to Lemma~\ref{lem:2-0}.

\begin{lemma}
	\label{lem:2}
	Let $\sigma>1$. Then:
	$$
	|x-y|^\sigma \le \left|x^\sigma - y^\sigma\right|, \quad\textup{ for all } x,y \ge 0.
	$$
\end{lemma}
\begin{proof}
	Assume, without loss of generality, that $x \ge y > 0$, and let
	$$
	t:=\dfrac{x}{y} \ge 1.
	$$
	
	Therefore, the sought inequality is equivalent to proving that
	$$
	(t-1)^\sigma \le t^\sigma -1,\quad\textup{ for all }t \ge 1.
	$$
	
	Consider the function $f:[1,\infty) \to \mathbb{R}$ defined by
	$$
	f(t):=(t-1)^\sigma-t^\sigma+1.
	$$
	
	Observe that
	$$
	f'(t)=\sigma (t-1)^{\sigma-1}-\sigma t^{\sigma-1}=\sigma((t-1)^{\sigma-1}-t^{\sigma-1}).
	$$
	
	Since $\sigma-1>0$ and since $0\le t-1 <t$, the monotonicity of the power operator gives that $(t-1)^{\sigma-1} \le t^{\sigma-1}$, so that $f'(t)\le 0$ for all $t>1$. Since $f(1)=0$, we infer that $f(t)\le 0$ for all $t \ge 1$ and the proof is complete.
\end{proof}

Thanks to Lemma~\ref{lem:2}, we can establish the following result, thanks to which we will be able to define a sound initial condition for the variational formulation we will be considering.

\begin{lemma}
	\label{lem:3}
	Let $\Omega$ be a domain in $\mathbb{R}^2$, let $T>0$ be given, let $1< \sigma <2$, and let $u \in L^\infty(0,T;L^\sigma(\Omega))$ be such that:
	$$
	\dfrac{\dd}{\dd t}\left(|u|^\frac{\sigma-2}{2}u\right) \in L^2(0,T;L^2(\Omega)).
	$$
	
	Then, we have that $\left(|u|^\frac{\sigma-2}{2} u\right) \in \mathcal{C}^0([0,T];L^2(\Omega))$ and that $u \in\mathcal{C}^0([0,T];L^\sigma(\Omega))$.
\end{lemma}
\begin{proof}
	Observe that
	$$
	\int_{0}^{T} \int_{\Omega} \left||u|^\frac{\sigma-2}{2}u\right|^2 \dd x \dd t= \int_{0}^{T} \int_{\Omega} |u|^\sigma \dd x \dd t.
	$$
	
	Since it was assumed that $u \in L^\infty(0,T;L^\sigma(\Omega))$, then the latter integral is finite. Hence, combining the latter with the assumption on the distributional derivative in time gives
	$$
	\left(|u|^\frac{\sigma-2}{2}u\right) \in H^1(0,T;L^2(\Omega)) \hookrightarrow \mathcal{C}^0([0,T];L^2(\Omega)).
	$$
	
	To show that $u \in \mathcal{C}^0([0,T];L^\sigma(\Omega))$, consider the function $h:\mathbb{R} \to \mathbb{R}$ defined by
	\begin{equation*}
		h(x)=|x|^\frac{2-\sigma}{\sigma} x,\quad\textup{ for all } x\in\mathbb{R}.
	\end{equation*}
	
	Clearly, the fact that $1<\sigma<2$ implies that $h(0)=0$, or equivalently, that $x=0$ is a removable discontinuity for $h$.
	The Nemytskii's operator associated with the function $h$ clearly satisfies the assumptions of Theorem~1.27 in~\cite{Roubicek2005} and we thus have that $u \in \mathcal{C}^0([0,T];L^\sigma(\Omega))$ as it was to be proved.
\end{proof}

It is actually possible to establish that $|u| \in \mathcal{C}^0([0,T];L^\sigma(\Omega))$ by resorting to a technique independent of the critical continuity theorem for Nemytskii's operators.
Such a technique will be exploited to establish the uniform convergence in Lemma~\ref{lem:5}.
We have to show that for each $t_0 \in [0,T]$,
$$
\lim_{t\to t_0} \int_{\Omega} ||u(t)|-|u(t_0)||^\sigma \dd x =0.
$$

By Lemma~\ref{lem:2}, we have that the Cauchy-Schwarz inequality and the triangle inequality give
\begin{align*}
	&\int_{\Omega} ||u(t)|-|u(t_0)||^\sigma \dd x \le \int_{\Omega} \left||u(t)|^\sigma-|u(t_0)|^\sigma\right| \dd x\\
	&=\int_{\Omega}\left|\left|\left||u(t)|^\frac{\sigma-2}{2}u(t)\right|^{2/\sigma}\right|^\sigma-\left|\left||u(t_0)|^\frac{\sigma-2}{2}u(t_0)\right|^{2/\sigma}\right|^\sigma\right| \dd x\\
	&=\int_{\Omega}\left|\left||u(t)|^\frac{\sigma-2}{2}u(t)\right|^2-\left||u(t_0)|^\frac{\sigma-2}{2}u(t_0)\right|^2\right| \dd x\\
	&=\int_{\Omega} \left|\left(\left||u(t)|^\frac{\sigma-2}{2}u(t)\right|-\left||u(t_0)|^\frac{\sigma-2}{2}u(t_0)\right|\right) \cdot \left(\left||u(t)|^\frac{\sigma-2}{2}u(t)\right|+\left||u(t_0)|^\frac{\sigma-2}{2}u(t_0)\right|\right)\right|\dd x\\
	&\le \left\|\left||u(t)|^\frac{\sigma-2}{2}u(t)\right|-\left||u(t_0)|^\frac{\sigma-2}{2}u(t_0)\right|\right\|_{L^2(\Omega)}
	\left\|\left||u(t)|^\frac{\sigma-2}{2}u(t)\right|+\left||u(t_0)|^\frac{\sigma-2}{2}u(t_0)\right|\right\|_{L^2(\Omega)}\\
	&\le \left\||u(t)|^\frac{\sigma-2}{2}u(t)-|u(t_0)|^\frac{\sigma-2}{2}u(t_0)\right\|_{L^2(\Omega)}
	\left\|\left||u(t)|^\frac{\sigma-2}{2}u(t)\right|+\left||u(t_0)|^\frac{\sigma-2}{2}u(t_0)\right|\right\|_{L^2(\Omega)}.
\end{align*}

Observe that the continuity in $[0,T]$ of $\left(|u|^\frac{\sigma-2}{2}u\right)$ implies the uniform boundedness of the second factor, as well as that the first factor tends to zero as $t \to t_0$.

The next lemma establishes that the supremum can be interchanged with a monotonically increasing continuous extended-real-valued function.

\begin{lemma}
	\label{lem:4} 
	Let $f :\mathbb{R} \to \overline{\mathbb{R}}$ be a monotonically increasing and continuous function. Then, given any $S \subset \mathbb{R}$, 
	$$
	\sup_{x \in S} f(x) = f(\sup S).
	$$
\end{lemma}
\begin{proof}
	By the definition of supremum, we can find a maximizing sequence $\{s_k\}_{k=1}^\infty \subset S$ for which
	$$
	s_k \to \sup S,\quad\textup{ as }k\to\infty.
	$$
	
	By the assumed continuity of $f$, we have that
	$$
	\lim_{k \to\infty} f(s_k) =f(\sup S).
	$$
	
	Since $f(s_k) \le \sup_{x \in S}f(x)$ then the inequality $f(\sup S) \le \sup_{x \in S}f(x)$ is obviously true, on the one hand. On the other hand, since $f$ is monotone, we have that
	$$
	f(x) \le f(\sup S), \quad \textup{ for all }x \in S.
	$$
	
	Therefore, passing to the supremum on the left-hand side, we obtain that
	$$
	\sup_{x \in S} f(x) \le \sup_{x \in S} f(\sup S)=f(\sup S),
	$$
	and the proof is complete.
\end{proof}

The next lemma establishes a convergence property for sequences of functions enjoying the regularities announced in Lemma~\ref{lem:3}.

\begin{lemma}
	\label{lem:5}
	Let $\Omega$ be a domain in $\mathbb{R}^2$, let $T>0$ be given, let $1< \sigma <2$.
	Let $\{u_k\}_{k=1}^\infty \subset L^\infty(0,T;W_0^{1,p}(\Omega))$ be such that:
	$$
	\left\{|u_k|^\frac{\sigma-2}{2} u_k\right\}_{k=1}^\infty \textup{ strongly converges in }\mathcal{C}^0([0,T];L^2(\Omega)).
	$$
	
	Then $\{|u_k|\}_{k=1}^\infty$ strongly converges in $\mathcal{C}^0([0,T];L^\sigma(\Omega))$.
\end{lemma}
\begin{proof}
	Since the space $\mathcal{C}^0([0,T];L^\sigma(\Omega))$ is complete, it suffices to show that  the sequence $\{u_k\}_{k=1}^\infty$ is a Cauchy sequence in $\mathcal{C}^0([0,T];L^\sigma(\Omega))$, i.e., we have to show that:
	
	\begin{equation}
		\label{Cauchyseq}
		\lim_{k,s \to\infty} \sup_{t\in[0,T]} \left(\int_{\Omega} ||u_k(t)|-|u_s(t)||^\sigma\dd x\right)^{1/\sigma} =0.
	\end{equation}
	
	Using Lemma~\ref{lem:2}, H\"older's inequality and Lemma~\ref{lem:4}, we have that
	\begin{align*}
		&\sup_{t\in[0,T]} \left(\int_{\Omega} ||u_k(t)|-|u_s(t)||^\sigma\dd x\right)^{1/\sigma} 
		\le \sup_{t\in[0,T]} \left(\int_{\Omega} \left||u_k(t)|^\sigma-|u_s(t)|^\sigma\right|\dd x\right)^{1/\sigma}\\
		&= \sup_{t\in[0,T]} \left(\int_{\Omega} \left|\left||u_k(t)|^\frac{\sigma-2}{2}u_k(t)\right|^2-\left||u_s(t)|^\frac{\sigma-2}{2}u_s(t)\right|^2\right|\dd x\right)^{1/\sigma}\\
		&\le \sup_{t\in[0,T]} \bigg\{\left(\int_{\Omega} \left||u_k(t)|^\frac{\sigma-2}{2}u_k(t)-|u_s(t)|^\frac{\sigma-2}{2}u_s(t)\right|^2\dd x\right)^{1/2}\\
		&\qquad\cdot \left(\int_{\Omega} \left||u_k(t)|^\frac{\sigma-2}{2}u_k(t)+|u_s(t)|^\frac{\sigma-2}{2}u_s(t)\right|^2\dd x\right)^{1/2}\bigg\}^{1/\sigma}\\
		&=\left\{\sup_{t\in[0,T]}\left[\left\||u_k(t)|^\frac{\sigma-2}{2}u_k(t)-|u_s(t)|^\frac{\sigma-2}{2}u_s(t)\right\|_{L^2(\Omega)}
		\left\||u_k(t)|^\frac{\sigma-2}{2}u_k(t)+|u_s(t)|^\frac{\sigma-2}{2}u_s(t)\right\|_{L^2(\Omega)}\right]\right\}^{1/\sigma}.
	\end{align*}
	
	The assumed strong convergence implies that the second factor is uniformly bounded with respect to $t \in [0,T]$. Hence, the latter term is less or equal than:
	\begin{align*}
		&\bigg\{\left[\sup_{t\in[0,T]}\left\||u_k(t)|^\frac{\sigma-2}{2}u_k(t)-|u_s(t)|^\frac{\sigma-2}{2}u_s(t)\right\|_{L^2(\Omega)}\right]\\
		&\qquad \cdot \left[\sup_{t\in[0,T]}
		\left\||u_k(t)|^\frac{\sigma-2}{2}u_k(t)+|u_s(t)|^\frac{\sigma-2}{2}u_s(t)\right\|_{L^2(\Omega)}\right]\bigg\}^{1/\sigma}.
	\end{align*}
	
	Since the second factor is uniformly bounded with respect to $t \in [0,T]$ and since the first factor tends to zero as $k,s \to \infty$ by the assumed strong convergence, we finally obtain the sought convergence~\eqref{Cauchyseq}. This completes the proof.
\end{proof}

The next lemma establishes an immersion that will be used in the proof of the existence of the solution.

\begin{lemma}
	\label{lem:6}
	Let $T>0$ be given and let $X$ be a normed vector space with the Radon-Nikodym property (cf., e.g., \cite{Raymond2002}). 
	Let $\langle\langle \cdot, \cdot\rangle\rangle$ denote the duality between $\left(\mathcal{C}^0([0,T];X)\right)^\ast$ and $\mathcal{C}^0([0,T];X)$, and let $\langle \cdot , \cdot \rangle$ denote the duality between $X^\ast$ and $X$.
	
	The mapping
	$$
	\Upsilon:L^1(0,T;X^\ast) \to \left(\mathcal{C}^0([0,T];X)\right)^\ast
	$$
	defined by 
	$$
	\langle\langle \Upsilon u,v\rangle\rangle:=\int_{0}^{T} \langle u, v\rangle \dd t,\quad \textup{ where } u \in L^1(0,T;X^\ast) \textup{ and } v \in \mathcal{C}^0([0,T];X),
	$$
	is linear, continuous and injective. Besides, if $\{u_k\}_{k=1}^\infty$ is bounded in $L^1(0,T;X^\ast)$ then $\{\Upsilon u_k\}_{k=1}^\infty$ is bounded in $\left(\mathcal{C}^0([0,T];X)\right)^\ast$.
\end{lemma}
\begin{proof}
	Let us check that $\Upsilon$ is linear. For each $u_1,u_2 \in L^1(0,T;X^\ast)$ we have that
	\begin{align*}
		\langle\langle\Upsilon(u_1+u_2),v\rangle\rangle &= \int_{0}^{T} \langle u_1+u_2,v \rangle \dd t = \int_{0}^{T} \langle u_1,v \rangle \dd t+\int_{0}^{T} \langle u_2,v \rangle \dd t\\
		&=\langle\langle\Upsilon u_1,v\rangle\rangle+\langle\langle\Upsilon u_2,v\rangle\rangle =\langle\langle \Upsilon u_1 + \Upsilon u_2,v\rangle\rangle,
	\end{align*}
	for all $v \in \mathcal{C}^0([0,T];X)$.
	
	Similarly, for each $\alpha \neq 0$ (for $\alpha =0$ the conclusion is immediate) and each $u \in L^1(0,T;X^\ast)$, we have that
	$$
	\langle\langle\Upsilon(\alpha u),v\rangle\rangle=\int_{0}^{T} \alpha \langle u,v\rangle\dd t=\langle\langle\Upsilon u,\alpha v\rangle\rangle
	=\langle\langle\alpha \Upsilon u,v\rangle\rangle,
	$$
	for all $v\in \mathcal{C}^0([0,T];X)$, and the sought linearity property is thus proved.
	
	To show the continuity of the mapping $\Upsilon$, let $\{u_k\}_{k=1}^\infty$ be such that $u_k \to u$ in $L^1(0,T;X^\ast)$. By H\"older's inequality in Lebesgue-Bochner spaces (cf., e.g., \cite{Yosida1980}), and the fact that the norm $\|\cdot\|_{L^\infty(0,T;X)}$ coincides with $\|\cdot\|_{\mathcal{C}^0([0,T];X)}$ along the elements of the space $\mathcal{C}^0([0,T];X)$, we have that:
	$$
	\sup_{\substack{v \in \mathcal{C}^0([0,T];X)\\ v \neq 0}}\dfrac{\left|\langle\langle \Upsilon u_k -\Upsilon u , v \rangle\rangle\right|}{\|v\|_{\mathcal{C}^0([0,T];X)}} \le \dfrac{\int_{0}^{T} |\langle u_k - u,v\rangle| \dd t}{\|v\|_{\mathcal{C}^0([0,T];X)}} \le \|u_k-u\|_{L^1(0,T;X^\ast)} \to 0 \textup{ as }k \to\infty.
	$$
	
	The boundedness of the mapping $\Upsilon$ is thus a direct consequence of the linearity and the continuity.
	
	To prove that $\Upsilon$ is injective, assume that $\Upsilon u_1 = \Upsilon u_2$ in $\left(\mathcal{C}^0([0,T];X)\right)^\ast$ then, for each $v \in \mathcal{C}^0([0,T];X)$, we have that
	$$
	\int_{0}^{T} \langle u_1 - u_2,v\rangle \dd t =0.
	$$
	
	In particular, consider functions of the form $v(t):=\varphi(t) w$, with $\varphi \in \mathcal{D}(0,T)$ and $w \in X$. Fix $\varphi\in\mathcal{D}(0,T)$ and let $w \in X$ vary arbitrarily. Since $X$ has the Radon-Nikodym property, we have that (cf., e.g., Theorem~8.13 of~\cite{Leoni2017}) the latter becomes
	$$
	0=\int_{0}^{T} \langle u_1 - u_2,v\rangle \dd t=\left\langle \int_{0}^{T} (u_1(t)-u_2(t)) \varphi(t) \dd t, w \right\rangle, \quad\textup{ for all }w\in X.
	$$
	
	This in turn implies that
	$$
	\int_{0}^{T} (u_1(t)-u_2(t)) \varphi(t) \dd t = 0 \textup{ in }X^\ast.
	$$
	
	Since this conclusion is independent of the choice of $\varphi \in \mathcal{D}(0,T)$, an application of the Fundamental Lemma of the Calculus of Variations (cf., e.g., \cite{Yosida1980}) gives that $u_1=u_2$ in $L^1(0,T;X^\ast)$ and the injectivity is thus proved.

\end{proof}

The conclusion of Lemma~\ref{lem:6} is that the space $L^1(0,T;X^\ast)$ can be identified with a subspace of $(\mathcal{C}^0([0,T];X))^\ast$. 

The proof of existence of solutions hinges on a compactness result proved by Dubinskii~\cite{Dubinski1965} (see also~\cite{BarretSuli2011} for some improvements and corrections), as well as other results proved by Raviart in the article~\cite{Raviart1970} that we recall here below.

\begin{lemma}[Lemma~1.1 of~\cite{Raviart1970}]
	\label{lem:7}
	Let $1< r <\infty$ and let $r'$ denote the H\"older conjugate exponent of $r$. The following inequalities hold
	\begin{align*}
		(|\xi|^{r-2}\xi-|\eta|^{r-2}\eta)\xi &\ge \dfrac{1}{r'}(|\xi|^r-|\eta|^r),\\
		(|\xi|^{r-2}\xi-|\eta|^{r-2}\eta)(\xi-\eta)&\ge C \left(|\xi|^{(r-2)/2}\xi-|\eta|^{(r-2)/2}\eta\right)^2, \textup{ for some }C=C(r)>0,
	\end{align*}
	for all $\xi, \eta \in \mathbb{R}$.
	\qed
\end{lemma}

The first inequality in this lemma is a direct consequence of Young's inequality (cf., e.g., \cite{Young1912}), while the second inequality was proved by Simon in the article~\cite{Simon1978}.

The next preliminary result we recall, is a generalized integration-by-parts formula, whose proof hinges on the Lebesgue theorem (cf., e.g., Theorem~2.11-3 of~\cite{PGCLCSHA}).

\begin{lemma}[Lemma~1.2 of~\cite{Raviart1970}]
	\label{lem:8}
	Let $\Omega$ be a domain in $\mathbb{R}^2$, and let $T>0$.
	Let $\alpha$ and $p$ be two real numbers greater than $1$, and let $v$ be a function such that
	\begin{align*}
		v \in L^p(0,T;W_0^{1,p}(\Omega)) \cap \mathcal{C}^0([0,T];L^\alpha(\Omega)),\\
		\dfrac{\dd}{\dd t}(|v|^{\alpha-2} v) \in L^{p'}(0,T;W^{-1,p'}(\Omega)).
	\end{align*}
	
	Then, the following formula holds
	$$
	\int_{0}^{T}\left\langle \dfrac{\dd}{\dd t}(|v|^{\alpha-2} v), v \right\rangle \dd t = \dfrac{\|v(T)\|_{L^\alpha(\Omega)}^\alpha}{\alpha'} - \dfrac{\|v(0)\|_{L^\alpha(\Omega)}^\alpha}{\alpha'},
	$$
	where $\langle \cdot, \cdot\rangle$ denotes the duality between $W^{-1,p'}(\Omega)$ and $W_0^{1,p}(\Omega)$.
	\qed
\end{lemma}

Finally, we recall Dubinskii's compactness lemma. The proof of Dubinskii's compactness lemma hinges on a preliminary lemma of independent importance, that we here recall for the sake of clarity. This lemma was originally proved by Dubinskii in the article~\cite{Dubinski1965} (cf. Lemma~2 of~\cite{Dubinski1965}) and was later improved by Barrett and S\"uli in the article~\cite{BarretSuli2011} (cf. Lemma~2.2 of~\cite{BarretSuli2011}).

\begin{lemma}
	\label{Dub:lemma}
	Let $A_1$ be a normed linear space. Let $S \subset A_1$ be such that $\lambda S \subset S$, for all $\lambda \in \mathbb{R}$. 
	
	Assume that the set $S$ is endowed with the semi-norm $M:S\to \mathbb{R}$, having the following properties:
	\begin{itemize}
		\item[$(1)$] $M(v) \ge 0$, for all $v \in S$,
		\item[$(2)$] $M(\lambda v)=|\lambda| M(v)$, for all $v \in S$ and all $\lambda \in \mathbb{R}$.
	\end{itemize}
	
	Assume that the set $\mathscr{M}:=\{v \in S;M(v)\le 1\}$ is relatively compact in $A_1$. Consider the semi-normed set:
	$$
	Y:=\left\{u \in L^1_{\textup{loc}}(0,T;A_1); \int_{0}^{T} [M(u(t))]^{q_0} \dd t < \infty \textup{ and } \int_{0}^{T} \left\|\dfrac{\dd u}{\dd t}\right\|_{A_1}^{q_1}\dd t < \infty\right\}.
	$$
	
	Then, we have that $Y \hookrightarrow\hookrightarrow \mathcal{C}^0([0,T];A_1) \subset L^{q_0}(0,T;A_1)$ and, moreover,
	\begin{itemize}
		\item[$(a)$] if $1 \le  q_0 \le \infty$ and $1<q_1 \le \infty$, then $Y \hookrightarrow \hookrightarrow \mathcal{C}^0([0,T];A_1)$;
		\item[$(b)$] if $1 \le q_0 <\infty$ and $q_1=1$, then $Y \hookrightarrow \hookrightarrow L^{q_0}(0,T;A_1)$.
	\end{itemize}
	\qed
\end{lemma}

We are now ready to recall the well-known Dubinskii compactness lemma (cf., e.g., Theorem~1 of~\cite{Dubinski1965} and Theorem~2.1 of~\cite{BarretSuli2011}).

\begin{lemma}
	\label{Dub}
	Let $A_0$ and $A_1$ be normed linear spaces such that $A_0 \hookrightarrow A_1$. Let $S \subset A_0$ be such that $\lambda S \subset S$, for all $\lambda \in \mathbb{R}$. 
	
	Assume that the set $S$ is endowed with the semi-norm $M:S\to \mathbb{R}$, having the following properties:
	\begin{itemize}
		\item[$(1)$] $M(v) \ge 0$, for all $v \in S$,
		\item[$(2)$] $M(\lambda v)=|\lambda| M(v)$, for all $v \in S$ and all $\lambda \in \mathbb{R}$.
	\end{itemize}
	
	Assume that the set $\mathscr{M}:=\{v \in S;M(v)\le 1\}$ is relatively compact in $A_0$. Consider the semi-normed set
	$$
	Y:=\left\{u \in L^1_{\textup{loc}}(0,T;A_1); \int_{0}^{T} [M(u(t))]^{q_0} \dd t < \infty \textup{ and } \int_{0}^{T} \left\|\dfrac{\dd u}{\dd t}\right\|_{A_1}^{q_1}\dd t < \infty\right\},
	$$
	with $1 \le q_0 \le \infty$ and $1 \le q_1 \le \infty$, and the pairs $(q_0,q_1)=(1,\infty)$ and $(q_0,q_1)=(\infty,1)$ cannot be attained.
	
	Then $Y$ is relatively compact in $L^{q_0}(0,T;A_0)$.
	\qed
\end{lemma}

A discrete version of this compactness result has been established by Raviart in the article~\cite{Raviart1967}.

\begin{lemma}[Lemma~1.4 of~\cite{Raviart1970}]
	\label{Dub:dis}
	Let $A_0$ and $A_1$ be normed linear spaces such that $A_0 \hookrightarrow A_1$. Let $S \subset A_0$ be such that $\lambda S \subset S$, for all $\lambda \in \mathbb{R}$. 
	Assume that the set $S$ is endowed with the semi-norm $M:S\to \mathbb{R}$, having the following properties:
	\begin{itemize}
		\item[$(1)$] $M(v) \ge 0$, for all $v \in S$,
		\item[$(2)$] $M(\lambda v)=|\lambda| M(v)$, for all $v \in S$ and all $\lambda \in \mathbb{R}$.
	\end{itemize}
	
	Assume that the set $\mathscr{M}:=\{v \in S;M(v)\le 1\}$ is relatively compact in $A_0$. 
	
	For each $k>0$, consider the vector $\bm{v}_k:=(v_k^n)_{n=0}^N$ of elements of $S$ such that
	$$
	k\sum_{n=0}^N [M(v_k^n)]^{q_0} \le c_0,\quad\textup{ for some constants }c_0\ge 0,
	$$
	and
	$$
	k\sum_{n=0}^{N-1}\left\|\dfrac{v_k^{n+1}-v_k^n}{k}\right\|_{A_1}^{q_1} \le c_1,\quad\textup{ for some constants }c_1\ge 0,
	$$
	where $1 \le q_0 \le \infty$ and $1 \le q_1 \le \infty$, and the pairs $(q_0,q_1)=(1,\infty)$ and $(q_0,q_1)=(\infty,1)$ cannot be attained.
	
	Then, it is possible to extract a subsequence of the sequence $(\Pi_k \bm{v}_k)$ that strongly converges in $L^{q_0}(0,T;A_0)$, where
	$$
	\Pi_k \bm{u}_k :(0,T) \to A_0 \qquad \textup{ with } \qquad (\Pi_k \bm{v}_k)(t):=v_k^{n+1}, \textup{ for a.e. } nk<t\le (n+1)k.
	$$
	\qed
\end{lemma}

Let us also recall a result on vector-valued measures proved by Zinger in the article~\cite{Singer1959} and later improved by Dinculeanu (see also, e.g., page~182 of~\cite{DieUhl1977}, and page~380 of~\cite{Dinculeanu1965}).

\begin{lemma}[Dinculeanu-Zinger theorem]
	\label{ds}
	Let $\omega$ be a compact Hausdorff space and let $X$ be a Banach space satisfying the Radon-Nikodym property. Let $\mathcal{F}$ be the collection of Borel sets of $\omega$.
	
	There exists an isomorphism between $(\mathcal{C}^0(\omega;X))^\ast$ and the space of the regular Borel measures with finite variation taking values in $X^\ast$. In particular, for each $F \in (\mathcal{C}^0(\omega;X))^\ast$, there exists a unique regular Borel measure $\mu:\mathcal{F} \to X^\ast$ in $\mathcal{M}(\omega;X^\ast)$ with finite variation such that
	$$
	\langle\langle f,F\rangle\rangle_{X} =\int_\omega \langle \dd \mu,f\rangle_{X^\ast,X},
	$$
	for all $f \in \mathcal{C}^0(\omega;X)$.
	\qed
\end{lemma}

We now state the penalized problem, which is suggested by the formal model~\eqref{BVPu}. Note that the initial condition makes sense thanks to Lemma~\ref{lem:3}.

\begin{customprob}{$\mathcal{P}_\ell$}
	\label{Pkappa}
	Find a function $u_\ell$ that satisfies
	\begin{align*}
		u_\ell &\in L^\infty(0,T;W_0^{1,p}(\Omega)),\\
		\dfrac{\dd}{\dd t}\left(|u_\ell|^{\frac{\alpha-2}{2}}u_\ell\right) &\in L^2(0,T;L^2(\Omega)),\\
		\dfrac{\dd}{\dd t}\left(|u_\ell|^{\alpha-2}u_\ell\right) &\in L^\infty(0,T;W^{-1,p'}(\Omega)),
	\end{align*}
	and satisfying the following variational equations:
	\begin{equation}
		\label{penalty}
		\int_{\Omega} \dfrac{\dd}{\dd t}(|u_\ell|^{\alpha-2} u) v \dd x
		+\int_{\Omega} \mu(x) |\nabla u_\ell|^{p-2}\nabla u_\ell \cdot \nabla v \dd x -\dfrac{1}{\ell}\int_{\Omega} \{u_\ell\}^{-} v\dd x
		= \int_{\Omega} \tilde{a}(t,x) v \dd x,
	\end{equation}
	in the sense of distributions on $(0,T)$ for all $v \in W_0^{1,p}(\Omega)$,	as well as the following initial condition
	$$
	u_\ell(0)=u_0,
	$$
	for some prescribed nonzero $u_0 \in K$.
	\bqed
\end{customprob}

\section{Existence of weak solutions and variational formulation for Problem~\eqref{BVPu}}
\label{Sec:3}

The recovery of the variational formulation for Problem~\eqref{BVPu} and the existence of solutions for this problem are broken into three parts.
We first establish the existence of solutions for Problem~\ref{Pkappa} by means of a semi-discrete scheme. For each $0 \le n \le N-1$, consider the semi-discrete problem:
\begin{customprob}{$\mathcal{P}_\ell^{n+1}$}
	\label{Pkappa:seq}
	Given $u_{\ell,k}^n \in W_0^{1,p}(\Omega)$, find a function $u_{\ell,k}^{n+1} \in W_0^{1,p}(\Omega)$ that satisfies the following variational equations:
	\begin{equation}
		\label{penalty:seq}
		\begin{aligned}
			&\dfrac{1}{k}\{|u_{\ell,k}^{n+1}|^{\alpha-2} u_{\ell,k}^{n+1} - |u_{\ell,k}^{n}|^{\alpha-2}u_{\ell,k}^{n}\}-\nabla \cdot \left(\mu |\nabla u_{\ell,k}^{n+1}|^{p-2} \nabla u_{\ell,k}^{n+1}\right)
			-\dfrac{\{u_{\ell,k}^{n+1}\}^{-}}{\ell}\\
			&=\dfrac{1}{k} \int_{nk}^{(n+1)k} \tilde{a}(t) \dd t \textup{ in }W^{-1,p'}(\Omega),
		\end{aligned}
	\end{equation}
	where $u_{\ell,k}^{0}:=u_0 \in K$, and $u_0$ is the prescribed element appearing in Problem~\ref{Pkappa}.
	\bqed
\end{customprob}

The following existence-and-uniqueness result can be established.

\begin{theorem}
	\label{thm:1}
	Let $T>0$, $\Omega \subset \mathbb{R}^2$ and $p$ be as in Section~\ref{Sec:2} and let $\alpha$ be as in~\eqref{alpha}. Let $\ell>0$ be given, let $N \ge 1$ be an integer, and define $k:=T/N$.
	Assume that~($H$\ref{H1})--($H$\ref{H4}) hold.
	
	For each $0 \le n \le N-1$, Problem~\ref{Pkappa:seq} admits a unique solution $u_{\ell,k}^{n+1} \in W_0^{1,p}(\Omega)$.
\end{theorem}
\begin{proof}
	For the sake of brevity, define
	
	\begin{equation}
		\label{atilde}
		\tilde{a}_k^n := \dfrac{1}{k} \int_{nk}^{(n+1)k} \tilde{a}(t) \dd t.
	\end{equation}
	
	Consider the operator $A_\ell :W_0^{1,p}(\Omega) \to W^{-1,p'}(\Omega)$ defined by
	\begin{equation}
		\label{Akappa}
		A_\ell(v):=|v|^{\alpha-2} v -\nabla \cdot \left(\mu |\nabla v|^{p-2} \nabla v\right)-\dfrac{\{v\}^{-}}{\ell},\quad\textup{ for all }v\in W_0^{1,p}(\Omega).
	\end{equation}
	
	Consider the operator $B_\ell :W_0^{1,p}(\Omega) \to W^{-1,p'}(\Omega)$ defined by
	\begin{equation}
		\label{Bkappa}
		B_\ell(v):= -\nabla \cdot \left(\mu |\nabla v|^{p-2} \nabla v\right)-\dfrac{\{v\}^{-}}{\ell},\quad\textup{ for all }v\in W_0^{1,p}(\Omega).
	\end{equation}
	
	The operators $A_\ell$ and $B_\ell$ are hemi-continuous, as each of their terms is hemi-continuous.
	By Lemma~\ref{lem:1} and Lemma~\ref{lem:7}, the operators $A_\ell$ and $B_\ell$ are strictly monotone.
	Finally, an application of the Poincar\'e-Friedrichs inequality and the monotonicity of the negative part operator (Lemma~\ref{lem:1}) give:
	$$
	\dfrac{\langle A_\ell v, v \rangle_{W^{-1,p'}(\Omega), W_0^{1,p}(\Omega)}}{\|v\|_{W_0^{1,p}(\Omega)}}
	\ge \dfrac{\|v\|_{L^\alpha(\Omega)}^\alpha + \mu_1 \|\nabla v\|_{L^p(\Omega)}^p}{\|v\|_{W_0^{1,p}(\Omega)}} \ge \mu_1 c_0 \|v\|_{W_0^{1,p}(\Omega)}^{p-1},
	$$
	and the term on the right diverges as $\|v\|_{W_0^{1,p}(\Omega)} \to \infty$. This means that the operator $A_\ell$ is coercive. With the same reasoning, it can be proved that the operator $B_\ell$ is coercive too.
	An application of the Minty-Browder theorem (cf., e.g., Theorem~9.14-1 of~\cite{PGCLNFAA}) 
	ensures that the numerical scheme in~\eqref{penalty:seq} admits a unique 
	solution $u_{\ell,k}^{n+1} \in W_0^{1,p}(\Omega)$.
	The proof is complete.
\end{proof}

Define the mapping $\Pi_k \bm{u}_\ell:(0,T) \to W_0^{1,p}(\Omega)$ by
\begin{equation}
	\label{Pil}
	\Pi_k \bm{u}_\ell(t):=u_{\ell,k}^{n+1}, \quad\textup{ if } n k < t \le (n+1) k.
\end{equation}

Next, we discuss the convergence of the sequence $\{\Pi_k \bm{u}_\ell\}_{k>0}$ as $k \to 0^+$ or, equivalently, as $N \to\infty$. To this aim, we need to establish some \emph{a priori} estimates.

\begin{theorem}
	\label{thm:2}
	Let $T>0$, $\Omega \subset \mathbb{R}^2$ and $p$ be as in Section~\ref{Sec:2} and let $\alpha$ be as in~\eqref{alpha}. Let $\ell>0$ be given and intended to go to zero, let $N \ge 1$ be an integer, and define $k:=T/N$.
	Assume that~($H$\ref{H1})--($H$\ref{H4}) hold.
	
	Then, the following \emph{a priori} estimates hold:
	\begin{equation}
		\label{est:1}
		\max_{0\le n\le N} \|u_{\ell,k}^{n}\|_{L^\alpha(\Omega)} \le C,
	\end{equation}
	
	\begin{equation}
		\label{est:2}
		k \sum_{n=0}^N \|u_{\ell,k}^{n}\|_{W_0^{1,p}(\Omega)}^p \le C,
	\end{equation}
	
	\begin{equation}
		\label{est:3}
		k \sum_{n=0}^{N-1} \left\|\dfrac{|u_{\ell,k}^{n+1}|^{\frac{\alpha-2}{2}}u_{\ell,k}^{n+1} -|u_{\ell,k}^{n}|^{\frac{\alpha-2}{2}}u_{\ell,k}^{n}}{k}\right\|_{L^2(\Omega)}^2 \le C,
	\end{equation}
	
	\begin{equation}
		\label{est:3:1}
		\max_{0\le n\le N} \left\||u_{\ell,k}^{n}|^\frac{\alpha-2}{2} u_{\ell,k}^{n}\right\|_{L^2(\Omega)} \le C,
	\end{equation}
	
	\begin{equation}
		\label{est:4}
		\max_{0\le n\le N} \|u_{\ell,k}^{n}\|_{W_0^{1,p}(\Omega)} \le C,
	\end{equation}
	
	\begin{equation}
		\label{est:5}
		\max_{0\le n\le N-1} \left\|\dfrac{|u_{\ell,k}^{n+1}|^{\alpha-2}u_{\ell,k}^{n+1} - |u_{\ell,k}^{n}|^{\alpha-2}u_{\ell,k}^{n}}{k}\right\|_{W^{-1,p'}(\Omega)} \le C\left(1+\dfrac{1}{\ell}\right),
	\end{equation}
	
	\begin{equation}
		\label{est:5:1}
		\max_{0\le n\le N} \left\||u_{\ell,k}^{n}|^{\alpha-2} u_{\ell,k}^{n}\right\|_{L^{\alpha'}(\Omega)} \le C,
	\end{equation}
	
	for some $C>0$ independent of $\ell$ and $N$.
\end{theorem}
\begin{proof}
	Multiplying~\eqref{penalty:seq} by $u_{\ell,k}^{n+1}$ in the sense of the duality between $W^{-1,p'}(\Omega)$ and $W_0^{1,p}(\Omega)$, we obtain
	\begin{equation*}
		\begin{aligned}
			&\dfrac{1}{k} \langle |u_{\ell,k}^{n+1}|^{\alpha-2} u_{\ell,k}^{n+1} -|u_{\ell,k}^{n}|^{\alpha-2} u_{\ell,k}^{n}, u_{\ell,k}^{n+1} \rangle_{W^{-1,p'}(\Omega), W_0^{1,p}(\Omega)}
			+\int_{\Omega} \mu(x) |\nabla u_{\ell,k}^{n+1}|^{p-2} \nabla u_{\ell,k}^{n+1} \cdot \nabla u_{\ell,k}^{n+1}\dd x\\
			&\quad-\dfrac{1}{\ell} \int_{\Omega} \{u_{\ell,k}^{n+1}\}^{-} u_{\ell,k}^{n+1} \dd x =\dfrac{1}{k}\int_{\Omega} \left(\int_{nk}^{(n+1)k} \tilde{a}(t) \dd t\right) u_{\ell,k}^{n+1} \dd x.
		\end{aligned}
	\end{equation*}
	
	Using the first estimate in Lemma~\ref{lem:7} with $r=\alpha$, $\xi=u_{\ell,k}^{n+1}$ and $\eta=u_{\ell,k}^{n}$, the Poincar\'e-Friedrichs inequality, Bochner's theorem (Theorem~8.9 of~\cite{Leoni2017}), H\"older's inequality and Young's inequality (cf., e.g., \cite{Young1912}) we obtain:
	\begin{equation*}
		\begin{aligned}
			&\dfrac{1}{\alpha' k}(\|u_{\ell,k}^{n+1}\|_{L^\alpha(\Omega)}^\alpha - \|u_{\ell,k}^{n}\|_{L^\alpha(\Omega)}^\alpha) +c_0 \mu_1 \|u_{\ell,k}^{n+1}\|_{W_0^{1,p}(\Omega)}^p
			+\dfrac{1}{\ell}\|\{u_{\ell,k}^{n+1}\}^{-}\|_{L^2(\Omega)}^2\\
			&\le \dfrac{1}{k \varepsilon p'} \left(\int_{nk}^{(n+1)k} \|\tilde{a}(t)\|_{W^{-1,p'}(\Omega)}^{p'} \dd t\right) +\dfrac{\varepsilon}{p}\|u_{\ell,k}^{n+1}\|_{W_0^{1,p}(\Omega)}^p,
		\end{aligned}
	\end{equation*}
	where $\varepsilon$ is any positive number.
	
	Multiplying by $(k \alpha')$ both sides and summing over $0 \le n \le s-1$, where $1 \le s \le N$, we find that
	\begin{equation*}
		\begin{aligned}
			&\sum_{n=0}^{s-1} \left\{\|u_{\ell,k}^{n+1}\|_{L^\alpha(\Omega)}^\alpha+\left(c_0 \mu_1-\dfrac{\varepsilon}{p}\right)k\alpha' \|u_{\ell,k}^{n+1}\|_{W_0^{1,p}(\Omega)}^p+\dfrac{k \alpha'}{\ell}\|\{u_{\ell,k}^{n+1}\}^{-}\|_{L^2(\Omega)}^2\right\}\\
			&\le \dfrac{\alpha'}{\varepsilon p'} \sum_{n=0}^{s-1} \left(\int_{nk}^{(n+1)k} \|\tilde{a}(t)\|_{W^{-1,p'}(\Omega)}^{p'} \dd t\right) +\sum_{n=0}^{s-1} \|u_{\ell,k}^{n}\|_{L^\alpha(\Omega)}^\alpha,
		\end{aligned}
	\end{equation*}
	and we finally obtain:
	\begin{equation*}
		\begin{aligned}
			&\|u_{\ell,k}^{s}\|_{L^\alpha(\Omega)}^\alpha +\left(c_0 \mu_1-\dfrac{\varepsilon}{p}\right)k\alpha' \sum_{n=0}^{s-1} \|u_{\ell,k}^{n+1}\|_{W_0^{1,p}(\Omega)}^p
			+\dfrac{k \alpha'}{\ell} \sum_{n=0}^{s-1} \|\{u_{\ell,k}^{n+1}\}^{-}\|_{L^2(\Omega)}^2\\
			&\le \dfrac{\alpha'}{\varepsilon p'} \sum_{n=0}^{s-1} \left(\int_{nk}^{(n+1)k} \|\tilde{a}(t)\|_{W^{-1,p'}(\Omega)}^{p'} \dd t\right)+\|u_0\|_{L^\alpha(\Omega)}^\alpha.
		\end{aligned}
	\end{equation*}
	
	Therefore, there exists $C=C(u_0,\alpha,\varepsilon,p,\Omega)>0$ such that:
	\begin{equation}
		\label{est:penalty}
		\begin{aligned}
			\max_{0\le n\le N} \|u_{\ell,k}^{n}\|_{L^\alpha(\Omega)} &\le C,\\
			k \sum_{n=0}^N \|u_{\ell,k}^{n}\|_{W_0^{1,p}(\Omega)}^p &\le C,\\
			\dfrac{k}{\ell} \sum_{n=1}^{N}\|\{u_{\ell,k}^{n}\}^{-}\|_{L^2(\Omega)}^2 &\le C,
		\end{aligned}
	\end{equation}
	so that the estimates~\eqref{est:1} and~\eqref{est:2} are proved. 
	
	By~\eqref{est:1}, for each $0 \le n \le N$, we have that
	\begin{equation*}
		\|u_{\ell,k}^{n}\|_{L^\alpha(\Omega)}^\alpha =\int_{\Omega} \left||u_{\ell,k}^{n}|^\frac{\alpha-2}{2} u_{\ell,k}^{n}\right|^2 \dd x =
		\left\||u_{\ell,k}^{n}|^\frac{\alpha-2}{2} u_{\ell,k}^{n}\right\|_{L^2(\Omega)}^2.
	\end{equation*}
	
	Since the left-hand side of the previous equation is bounded independently of $N$ and $\ell$, we immediately infer that
	$$
	\max_{0\le n\le N} \left\||u_{\ell,k}^{n}|^\frac{\alpha-2}{2} u_{\ell,k}^{n}\right\|_{L^2(\Omega)} \le C,
	$$
	for some $C>0$ independent of $N$ and $\ell$, thus establishing the estimate~\eqref{est:3:1}.
	
	In order to recover the estimates~\eqref{est:3} and~\eqref{est:4}, let us multiply~\eqref{penalty:seq} by $(u_{\ell,k}^{n+1}-u_{\ell,k}^{n})$ in the sense of the duality between $W^{-1,p'}(\Omega)$ and $W_0^{1,p}(\Omega)$. We obtain
	\begin{equation}
		\label{prelim:1}
		\begin{aligned}
			&\dfrac{1}{k} \langle |u_{\ell,k}^{n+1}|^{\alpha-2} u_{\ell,k}^{n+1} -|u_{\ell,k}^{n}|^{\alpha-2} u_{\ell,k}^{n}, u_{\ell,k}^{n+1} - u_{\ell,k}^{n}\rangle_{W^{-1,p'}(\Omega), W_0^{1,p}(\Omega)}\\
			&\quad+\int_{\Omega} \mu(x) |\nabla u_{\ell,k}^{n+1}|^{p-2} \nabla u_{\ell,k}^{n+1} \cdot \nabla (u_{\ell,k}^{n+1}-u_{\ell,k}^{n})\dd x\\
			&\quad-\dfrac{1}{\ell} \int_{\Omega} \{u_{\ell,k}^{n+1}\}^{-} (u_{\ell,k}^{n+1} - u_{\ell,k}^{n})\dd x =\dfrac{1}{k}\int_{\Omega} \left(\int_{nk}^{(n+1)k} \tilde{a}(t) \dd t\right) (u_{\ell,k}^{n+1}-u_{\ell,k}^{n}) \dd x.
		\end{aligned}
	\end{equation}
	
	For each $1 \le s \le N$, the following identity holds:
	\begin{equation}
		\label{prelim:2}
		\begin{aligned}
			&\sum_{n=0}^{s-1} \int_{\Omega} \tilde{a}_k^n (u_{\ell,k}^{n+1} - u_{\ell,k}^{n}) \dd x
			=-\sum_{n=0}^{s-2}\int_{\Omega} (\tilde{a}_k^{n+1}-\tilde{a}_k^n) u_{\ell,k}^{n+1} \dd x
			+\int_{\Omega}\tilde{a}_k^{s-1} u_{\ell,k}^{s} \dd x -\int_{\Omega} \tilde{a}_k^0 u_0 \dd x\\
			&=\sum_{n=0}^{s-2} \int_{nk}^{(n+1)k} \int_{\Omega} \left(\dfrac{\tilde{a}(t+k)-\tilde{a}(t)}{k}\right) u_{\ell,k}^{n+1} \dd x \dd t +\int_{\Omega}\tilde{a}_k^{s-1} u_{\ell,k}^{s} \dd x -\int_{\Omega} \tilde{a}_k^0 u_0 \dd x.
		\end{aligned}
	\end{equation}
	
	An application of Lebesgue's inequality, the triangle inequality and Young's inequality~\cite{Young1912} gives
	\begin{equation}
		\label{prelim:3}
		\begin{aligned}
			&\left|\sum_{n=0}^{s-2}\int_{\Omega} (\tilde{a}_k^{n+1}-\tilde{a}_k^n) u_{\ell,k}^{n+1} \dd x\right|
			=\left|\sum_{n=0}^{s-2} \int_{nk}^{(n+1)k} \int_{\Omega} \left(\dfrac{\tilde{a}(t+k)-\tilde{a}(t)}{k}\right)u_{\ell,k}^{n+1} \dd x \dd t\right|\\
			&\le \sum_{n=0}^{s-2} \int_{nk}^{(n+1)k} \left|\int_{\Omega} \dfrac{\tilde{a}(t+k)-\tilde{a}(t)}{k} u_{\ell,k}^{n+1}\dd x\right| \dd t\\
			&\le \sum_{n=0}^{s-2}\int_{nk}^{(n+1)k} \left\|\dfrac{\tilde{a}(t+k)-\tilde{a}(t)}{k}\right\|_{W^{-1,p'}(\Omega)} \|u_{\ell,k}^{n+1}\|_{W_0^{1,p}(\Omega)}\dd t\\
			& \le \sum_{n=0}^{s-2}\left[\dfrac{1}{p' \varepsilon}\int_{nk}^{(n+1)k} \left\|\dfrac{\tilde{a}(t+k)-\tilde{a}(t)}{k}\right\|_{W^{-1,p'}(\Omega)}^{p'}\dd t+\dfrac{\varepsilon k}{p}\|u_{\ell,k}^{n+1}\|_{W_0^{1,p}(\Omega)}^p\right].
		\end{aligned}
	\end{equation}
	
	Since $\tilde{a} \in W^{1,p}(0,T;\mathcal{C}^0(\overline{\Omega}))$ (see assumption~($H$\ref{H4})), an application of the finite difference quotients theory (cf., e.g., Chapter~5 in~\cite{Evans2010}), we have that the latter term can be estimated as follows:
	\begin{equation}
		\label{prelim:4}
		\begin{aligned}
			&\sum_{n=0}^{s-2}\left[\dfrac{1}{p' \varepsilon}\int_{nk}^{(n+1)k} \left\|\dfrac{\tilde{a}(t+k)-\tilde{a}(t)}{k}\right\|_{W^{-1,p'}(\Omega)}^{p'}\dd t+\dfrac{\varepsilon k}{p}\|u_{\ell,k}^{n+1}\|_{W_0^{1,p}(\Omega)}^p\right]\\
			& \le \dfrac{C}{p' \varepsilon} \sum_{n=0}^{s-2} \int_{nk}^{(n+1)k} \left\|\dfrac{\dd \tilde{a}}{\dd t}(t)\right\|_{W^{-1,p'}(\Omega)}^{p'} \dd t +\dfrac{\varepsilon k}{p} \sum_{n=0}^{s-2} \|u_{\ell,k}^{n+1}\|_{W_0^{1,p}(\Omega)}^p\\
			&=\dfrac{C}{p' \varepsilon} \int_{0}^{(s-1)k} \left\|\dfrac{\dd \tilde{a}}{\dd t}(t)\right\|_{W^{-1,p'}(\Omega)}^{p'} \dd t + \dfrac{\varepsilon k}{p} \sum_{n=0}^{s-2} \|u_{\ell,k}^{n+1}\|_{W_0^{1,p}(\Omega)}^p.
		\end{aligned}
	\end{equation}
	
	Combining~\eqref{prelim:2}--\eqref{prelim:4} thus gives
	\begin{equation}
		\label{prelim:5}
		\begin{aligned}
			&\left|\sum_{n=0}^{s-1} \int_{\Omega} \tilde{a}_k^n (u_{\ell,k}^{n+1} - u_{\ell,k}^{n}) \dd x\right|
			\le \left|\sum_{n=0}^{s-2} \int_{nk}^{(n+1)k} \int_{\Omega} \left(\dfrac{\tilde{a}(t+k)-\tilde{a}(t)}{k}\right) u_{\ell,k}^{n+1} \dd x \dd t\right|\\
			&\quad+\left|\int_{\Omega}\tilde{a}_k^{s-1} u_{\ell,k}^{s} \dd x\right| +\left|\int_{\Omega} \tilde{a}_k^0 u_0 \dd x\right|\\
			& \le \dfrac{C}{p' \varepsilon} \int_{0}^{(s-1)k} \left\|\dfrac{\dd \tilde{a}}{\dd t}(t)\right\|_{W^{-1,p'}(\Omega)}^{p'} \dd t +\dfrac{\varepsilon k}{p} \sum_{n=0}^{s-2} \|u_{\ell,k}^{n+1}\|_{W_0^{1,p}(\Omega)}^p\\
			&\quad +\dfrac{1}{p' \varepsilon} \|\tilde{a}\|_{L^\infty(0,T;W^{-1,p'}(\Omega))}^{p'} +\dfrac{\varepsilon}{p}\|u_{\ell,k}^{s}\|_{W_0^{1,p}(\Omega)}^p\\
			&\quad +\dfrac{1}{p' \varepsilon} \|\tilde{a}\|_{L^\infty(0,T;W^{-1,p'}(\Omega))}^{p'} +\dfrac{\varepsilon}{p}\|u_{0}\|_{W_0^{1,p}(\Omega)}^p\\
			& \le \left[\dfrac{C}{p' \varepsilon}\int_{0}^{sk} \left\|\dfrac{\dd \tilde{a}}{\dd t}(t)\right\|_{W^{-1,p'}(\Omega)}^{p'} \dd t +\dfrac{2}{p' \varepsilon}\|\tilde{a}\|_{L^\infty(0,T;W^{-1,p'}(\Omega))}^{p'}\right]\\
			&\quad\cdot \left[\underbrace{\dfrac{\varepsilon k}{p}\left(\sum_{n=0}^{s-2} \|u_{\ell,k}^{n}\|_{W_0^{1,p}(\Omega)}^p\right)}_{\textup{bounded by~\eqref{est:2}}}+\dfrac{\varepsilon}{p}\|u_{\ell,k}^{s}\|_{W_0^{1,p}(\Omega)}^p+\dfrac{\varepsilon}{p}\|u_0\|_{W_0^{1,p}(\Omega)}^p\right]\\
			&\le C(1+\varepsilon\|u_{\ell,k}^{s}\|_{W_0^{1,p}(\Omega)}^p).
		\end{aligned}
	\end{equation}

For any $0\le n \le N-1$, we have that an application of the Young's inequality~\cite{Young1912} gives:
\begin{equation*}
	\begin{aligned}
		&\dfrac{1}{\ell}\int_{\Omega} \{u_{\ell,\kappa}^{n+1}\}^{-}(u_{\ell,\kappa}^{n+1}-u_{\ell,\kappa}^{n}) \dd x
		=\dfrac{1}{\ell}\int_{\Omega} \left(-\left|\{u_{\ell,\kappa}^{n+1}\}^{-}\right|^2-\{u_{\ell,\kappa}^{n+1}\}^{-}\left(\{u_{\ell,\kappa}^{n}\}^{+}-\{u_{\ell,\kappa}^{n}\}^{-}\right)\right) \dd x\\
		&\le -\dfrac{1}{\ell}\int_{\Omega} \left|\{u_{\ell,\kappa}^{n+1}\}^{-}\right|^2 \dd x
		+\dfrac{1}{\ell}\int_{\Omega} \{u_{\ell,\kappa}^{n+1}\}^{-} \{u_{\ell,\kappa}^{n}\}^{-} \dd x\\
		& \le -\dfrac{1}{\ell}\int_{\Omega} \left|\{u_{\ell,\kappa}^{n+1}\}^{-}\right|^2 \dd x
		+\dfrac{1}{2\ell}\int_{\Omega} \left|\{u_{\ell,\kappa}^{n+1}\}^{-}\right|^2+\left|\{u_{\ell,\kappa}^{n}\}^{-}\right|^2 \dd x\\
		&= -\dfrac{1}{2\ell}\int_{\Omega} \left|\{u_{\ell,\kappa}^{n+1}\}^{-}\right|^2 \dd x
		+\dfrac{1}{2\ell}\int_{\Omega} \left|\{u_{\ell,\kappa}^{n}\}^{-}\right|^2 \dd x.
	\end{aligned}
\end{equation*}

For any given $1 \le s \le N$, we thus have that the previous estimate gives:
\begin{equation*}
	\begin{aligned}
		&\dfrac{1}{\ell}\sum_{n=0}^{s-1}\int_{\Omega} \{u_{\ell,\kappa}^{n+1}\}^{-}(u_{\ell,\kappa}^{n+1}-u_{\ell,\kappa}^{n}) \dd x
		\le -\dfrac{1}{2\ell} \sum_{n=0}^{s-1} \int_{\Omega} \left|\{u_{\ell,\kappa}^{n+1}\}^{-}\right|^2 - \left|\{u_{\ell,\kappa}^{n}\}^{-}\right|^2 \dd x\\
		&=-\dfrac{1}{2\ell}\int_{\Omega} \left|\{u_{\ell,\kappa}^{s}\}^{-}\right|^2 - \left|\{u_{\ell,\kappa}^{0}\}^{-}\right|^2 \dd x
		=-\dfrac{1}{2\ell}\int_{\Omega} \left|\{u_{\ell,\kappa}^{s}\}^{-}\right|^2 - \left|\{u_0\}^{-}\right|^2 \dd x\\
		&=-\dfrac{1}{2\ell}\int_{\Omega} \left|\{u_{\ell,\kappa}^{s}\}^{-}\right|^2 \dd x \le 0,
	\end{aligned}
\end{equation*}
where the last equality directly follows from the fact that $u_0 \in K$.

To summarise, we have shown that
\begin{equation}
\label{SC1}
\dfrac{1}{\ell}\sum_{n=0}^{s-1}\int_{\Omega} \{u_{\ell,\kappa}^{n+1}\}^{-}(u_{\ell,\kappa}^{n+1}-u_{\ell,\kappa}^{n}) \dd x \le 0.
\end{equation}
	
	Summing~\eqref{prelim:1} over $0 \le n \le s-1$, with $1\le s \le N$, applying Lemma~\ref{lem:7}, exploiting~\eqref{prelim:5},~\eqref{SC1}, Young's inequality~\cite{Young1912} and the Poincar\'e-Friedrichs inequality gives
	\begin{equation*}
		\begin{aligned}
			&k \sum_{n=0}^{s-1} \left\|\dfrac{|u_{\ell,k}^{n+1}|^\frac{\alpha-2}{2}u_{\ell,k}^{n+1}-|u_{\ell,k}^{n}|^\frac{\alpha-2}{2}u_{\ell,k}^{n}}{k}\right\|_{L^2(\Omega)}^2
			+\dfrac{\mu_1 c_0}{p} \|u_{\ell,k}^{s}\|_{W_0^{1,p}(\Omega)}^p\\
			&\le k \sum_{n=0}^{s-1} \left\|\dfrac{|u_{\ell,k}^{n+1}|^\frac{\alpha-2}{2}u_{\ell,k}^{n+1}-|u_{\ell,k}^{n}|^\frac{\alpha-2}{2}u_{\ell,k}^{n}}{k}\right\|_{L^2(\Omega)}^2
			+\dfrac{\mu_1}{p} \sum_{n=0}^{s-1} \left\{\|\nabla u_{\ell,k}^{n+1}\|_{L^p(\Omega)}^p - \|\nabla u_{\ell,k}^{n}\|_{L^p(\Omega)}^p\right\}\\
			&\le C(1+\varepsilon\|u_{\ell,k}^{s}\|_{W_0^{1,p}(\Omega)}^p),
		\end{aligned}
	\end{equation*}
	so that, in the end, we obtain the following estimate
	\begin{equation*}
		k \sum_{n=0}^{s-1} \left\|\dfrac{|u_{\ell,k}^{n+1}|^\frac{\alpha-2}{2}u_{\ell,k}^{n+1}-|u_{\ell,k}^{n}|^\frac{\alpha-2}{2}u_{\ell,k}^{n}}{k}\right\|_{L^2(\Omega)}^2
		+\left(\dfrac{\mu_1 c_0}{p} - C \varepsilon\right) \|u_{\ell,k}^{s}\|_{W_0^{1,p}(\Omega)}^p \le C,
	\end{equation*}
	and the estimates~\eqref{est:3} and ~\eqref{est:4} straightforwardly follow.
	
	In order to establish the estimate~\eqref{est:5}, we exploit the boundedness of the $p$-Laplace operator (cf., e.g., Chapter~9 in~\cite{{PGCLNFAA}}), the fact that $\mu(x) \le \mu_2$, and~\eqref{est:4} so as to be in a position to evaluate
	\begin{equation*}
		\label{prelim:6}
		\begin{aligned}
			&\left\|\dfrac{|u_{\ell,k}^{n+1}|^{\alpha-2}u_{\ell,k}^{n+1}-|u_{\ell,k}^{n}|^{\alpha-2}u_{\ell,k}^{n}}{k}\right\|_{W^{-1,p'}(\Omega)}\\
			&\le \left\|\nabla \cdot \left(\mu |\nabla u_{\ell,k}^{n+1}|^{p-2} \nabla u_{\ell,k}^{n+1}\right)\right\|_{W^{-1,p'}(\Omega)}
			+\dfrac{1}{\ell}\|\{u_{\ell,k}^{n+1}\}^{-}\|_{W^{-1,p'}(\Omega)}+\|\tilde{a}\|_{W^{1,p}(0,T;\mathcal{C}^0(\overline{\Omega}))}\\
			&\le C\left(1+\dfrac{1}{\ell}\right), \quad\textup{ for some } C>0 \textup{ independent of } N,
		\end{aligned}
	\end{equation*}
	for all $0 \le n \le N-1$, thus establishing the estimate~\eqref{est:5}.
	
	By~\eqref{est:1}, for each $0\le n \le N$, we have that
	\begin{equation*}
		\begin{aligned}
			&\|u_{\ell,k}^{n}\|_{L^\alpha(\Omega)} =\left(\int_{\Omega} |u_{\ell,k}^{n}|^\alpha \dd x\right)^{1/\alpha}
			=\left(\int_{\Omega} \left||u_{\ell,k}^{n}|^{\alpha-2}u_{\ell,k}^{n}\right|^\frac{\alpha}{\alpha-1} \dd x\right)^{1/\alpha}\\
			&=\left(\int_{\Omega} \left||u_{\ell,k}^{n}|^{\alpha-2}u_{\ell,k}^{n}\right|^{\alpha'} \dd x\right)^\frac{1}{\alpha' (\alpha-1)}
			=\left\||u_{\ell,k}^{n}|^{\alpha-2} u_{\ell,k}^{n}\right\|_{L^{\alpha'}(\Omega)}^{\alpha'/\alpha}.
		\end{aligned}
	\end{equation*}
	
	Since the first term is bounded independently of $N$ and $\ell$, we immediately infer that
	\begin{equation}
		\label{chi-kappa}
		\max_{0\le n\le N} \left\||u_{\ell,k}^{n}|^{\alpha-2} u_{\ell,k}^{n}\right\|_{L^{\alpha'}(\Omega)} \le C,
	\end{equation}
	for some $C>0$ independent of $N$ and $\ell$, thus establishing~\eqref{est:5:1}, and completing the proof.
\end{proof}

Given $\bm{v}_{\ell,k}=\{v_{\ell,k}^{n}\}_{n=0}^{N}$, the function $D_k(\Pi_k \bm{v}_{\ell,k}):(0,T) \to W_0^{1,p}(\Omega)$ is defined by
\begin{equation}
	\label{finite-difference}
	D_k(\Pi_k \bm{v}_{\ell,k})(t):=\dfrac{v_{\ell,k}^{n+1}-v_{\ell,k}^{n}}{k},\quad\textup{ for all } n k < t \le (n+1)k, \quad 0 \le n \le N-1.
\end{equation}

As a result of the estimates~\eqref{est:1}-\eqref{est:5:1}, we have that
\begin{equation}
	\label{estimates}
	\begin{aligned}
		\{\Pi_k \bm{u}_{\ell,k}\}_{k >0} &\textup{ is bounded in } L^\infty(0,T;W_0^{1,p}(\Omega)) \textup{ independently of }k,\\
		\{B_\ell(\Pi_k \bm{u}_{\ell,k})\}_{k >0} &\textup{ is bounded in } L^\infty(0,T;W^{-1,p'}(\Omega))\textup{ independently of }k,\\
		\{|\Pi_k \bm{u}_{\ell,k}|^\frac{\alpha-2}{2}\Pi_k \bm{u}_{\ell,k}\}_{k >0} &\textup{ is bounded in } L^\infty(0,T;L^2(\Omega))\textup{ independently of }k,\\
		\{D_k(|\Pi_k \bm{u}_{\ell,k}|^\frac{\alpha-2}{2}\Pi_k \bm{u}_{\ell,k})\}_{k >0} &\textup{ is bounded in } L^\infty(0,T;L^2(\Omega))\textup{ independently of }k,\\
		\{|\Pi_k \bm{u}_{\ell,k}|^{\alpha-2} \Pi_k \bm{u}_{\ell,k}\}_{k >0} &\textup{ is bounded in } L^\infty(0,T;L^{\alpha'}(\Omega))\textup{ independently of }k,\\
		\{D_k(|\Pi_k \bm{u}_{\ell,k}|^{\alpha-2}\Pi_k \bm{u}_{\ell,k})\}_{k >0} &\textup{ is bounded in } L^\infty(0,T;W^{-1,p'}(\Omega))\textup{ independently of }k.
	\end{aligned}
\end{equation}

Thanks to~\eqref{estimates}, we can establish the existence of solutions for Problem~\eqref{Pkappa}.

\begin{theorem}
	\label{thm:3}
	Let $T>0$, $\Omega \subset \mathbb{R}^2$ and $p$ be as in Section~\ref{Sec:2} and let $\alpha$ be as in~\eqref{alpha}. Let $\ell>0$ be given, let $N \ge 1$ be an integer, and define $k:=T/N$.
	Assume that ($H$\ref{H1})--($H$\ref{H4}) hold.
	The \emph{a priori} estimates~\eqref{estimates} imply that the following convergence process takes place (recall that $B_\ell$ has been defined in~\eqref{Bkappa}):
	\begin{equation}
		\label{conv-proc}
		\begin{gathered}
			\Pi_k \bm{u}_{\ell,k} \wsc u_\ell \textup{ in } L^\infty(0,T;W_0^{1,p}(\Omega)),\\
			B_\ell(\Pi_k \bm{u}_{\ell,k}) \wsc g_\ell\textup{  in } L^\infty(0,T;W^{-1,p'}(\Omega)),\\
			|\Pi_k \bm{u}_{\ell,k}|^\frac{\alpha-2}{2}\Pi_k \bm{u}_{\ell,k} \wsc v_\ell \textup{ in } L^\infty(0,T;L^2(\Omega)),\\
			D_k(|\Pi_k \bm{u}_{\ell,k}|^\frac{\alpha-2}{2}\Pi_k \bm{u}_{\ell,k}) \rightharpoonup \dfrac{\dd v_\ell}{\dd t}\textup{ in } L^2(0,T;L^2(\Omega)),\\
			|\Pi_k \bm{u}_{\ell,k}|^{\alpha-2} \Pi_k \bm{u}_{\ell,k} \wsc w_\ell \textup{ in } L^\infty(0,T;L^{\alpha'}(\Omega)),\\
			D_k(|\Pi_k \bm{u}_{\ell,k}|^{\alpha-2}\Pi_k \bm{u}_{\ell,k}) \wsc \dfrac{\dd w_\ell}{\dd t} \textup{ in } L^\infty(0,T;W^{-1,p'}(\Omega)),\\
			|u_{\ell,k}^N|^{\alpha-2} u_{\ell,k}^N \rightharpoonup \chi_\ell \textup{ in } L^{\alpha'}(\Omega).
		\end{gathered}
	\end{equation}
	
	Besides, the weak-star limit $u_\ell$ recovered in the first convergence of~\eqref{conv-proc} is a solution for Problem~\ref{Pkappa}, and the weak-star limits $v_\ell$ and $w_\ell$ satisfy
	\begin{equation*}
		\begin{aligned}
			v_\ell&=|u_\ell|^\frac{\alpha-2}{2} u_\ell,\\
			w_\ell&=|u_\ell|^{\alpha-2} u_\ell.
		\end{aligned}
	\end{equation*}
\end{theorem}
\begin{proof}
	The convergence process~\eqref{conv-proc} holds by virtue of an application of the Banach-Alaoglu-Bourbaki theorem (cf., e.g., Theorem~3.6 of~\cite{Brez11}) to the estimates~\eqref{estimates} and~\eqref{chi-kappa}.
	The nontivial part of the proof amounts to identifying the weak-star limits $v_\ell$ and $w_\ell$ and to showing that the weak-star limit $u_\ell$ solves Problem~\ref{Pkappa}.
	
	To begin with, we show that $w_\ell=|u_\ell|^{\alpha-2} u_\ell$. Given any $u \in W_0^{1,p}(\Omega)$, we look for a number $\beta \in \mathbb{R}$ for which
	\begin{equation*}
		\left||u|^{\alpha-2} u\right|^{\beta-2} |u|^{\alpha-2} u =u.
	\end{equation*}
	
	Simple algebraic manipulations transform the latter into
	\begin{equation*}
		|u|^{(\alpha-2)(\beta-2) +(\beta-2)+(\alpha-2)} u = u,
	\end{equation*} and we observe that a sufficient condition insuring this is that $\beta$ satisfies
	\begin{equation*}
		(\alpha-1)(\beta-2)+(\alpha-2)=0,
	\end{equation*}
	which is equivalent to writing
	\begin{equation*}
		\beta= 2+\dfrac{2-\alpha}{\alpha-1}=\dfrac{2\alpha-2+2-\alpha}{\alpha-1}=\dfrac{\alpha}{\alpha-1}=\alpha'.
	\end{equation*}
	
	Let $v:=|u|^{\alpha-2} u$ and observe that if $\beta=\alpha'$ then $|v|^{\beta-2} v \in W_0^{1,p}(\Omega)$ so that the set
	\begin{equation*}
		S:=\{v;(|v|^{\alpha'-2}v) \in W_0^{1,p}(\Omega)\}
	\end{equation*}
	is non-empty. Define the semi-norm
	\begin{equation*}
		M(v):=\left\|\nabla(|v|^{\alpha'-2}v)\right\|_{\bm{L}^p(\Omega)}^\frac{1}{\alpha'-1},\quad\textup{ for all }v \in S,
	\end{equation*}
	and define the set 
	\begin{equation*}
		\mathscr{M}:=\{v\in S; M(v)\le 1\}.
	\end{equation*}
	
	An application of the Poincar\'e-Friedrichs inequality gives that there exists a constant $c_0=c_0(\Omega)>0$ such that
	\begin{equation*}
		\begin{aligned}
			1 &\ge M(v) \ge c_0^\frac{1}{\alpha'-1} \left\||v|^{\alpha'-2} v\right\|_{W_0^{1,p}(\Omega)}^\frac{1}{\alpha'-1} \ge c_0^\frac{1}{\alpha'-1} \left\||v|^{\alpha'-2} v\right\|_{L^p(\Omega)}^\frac{1}{\alpha'-1}\\
			&=c_0^\frac{1}{\alpha'-1} \left(\int_{\Omega} \left||v|^{\alpha'-2} v\right|^p\dd x\right)^{1/(p(\alpha'-1))}
			=c_0^\frac{1}{\alpha'-1} \left(\int_{\Omega} |v|^{(\alpha'-1)p}\dd x\right)^{1/(p(\alpha'-1))}=c_0^\frac{1}{\alpha'-1} \|v\|_{L^{(\alpha'-1)p}(\Omega)}.
		\end{aligned}
	\end{equation*}
	
	Let $\{v_m\}_{m=1}^\infty$ be a sequence in $\mathscr{M}$. Since, by the Rellich-Kondra\v{s}ov theorem, we have that $W_0^{1,p}(\Omega) \hookrightarrow\hookrightarrow L^p(\Omega)$ we obtain that, up to passing to a subsequence, there exists an element $w \in L^p(\Omega)$ such that
	\begin{equation}
		\label{conv-1}
		(|v_m|^{\alpha'-2} v_m) \to w,\quad\textup{ in } L^p(\Omega)\textup{ as }m\to\infty.
	\end{equation}
	
	Since $1<\alpha<2$ and $2.8\le p \le 5$, then $\alpha'>2$ and it thus results that $1<p'<p<(\alpha'-1)p<\infty$ (so that $L^{(\alpha'-1)p}(\Omega)$ is uniformly convex; cf., e.g., \cite{Brez11}) and that $\{v_m\}_{m=1}^\infty$ is bounded in $L^{p'}(\Omega)$. The reflexivity of $L^{p'}(\Omega)$ puts us in a position to apply the Banach-Eberlein-Smulian theorem (cf., e.g., Theorem~5.14-4 of~\cite{PGCLNFAA}) and extract a subsequence, still denoted $\{v_m\}_{m=1}^\infty$, that weakly converges to an element $v \in L^{(\alpha'-1)p}(\Omega) \hookrightarrow L^{p'}(\Omega)$.
	Consider the mapping
	\begin{equation*}
		v \in L^{p'}(\Omega) \mapsto (|v|^{\alpha'-2} v) \in L^p(\Omega),
	\end{equation*}
	and observe that this mapping is hemi-continuous and monotone, being associated with the mapping $\xi\in \mathbb{R} \to (|\xi|^{\alpha'-2}\xi) \in \mathbb{R}$, with $\alpha'>2$ thanks to~\eqref{alpha}, which is continuous and monotone.
	Therefore, an application of Theorem~9.13-2 of~\cite{PGCLNFAA} gives that $w=|v|^{\alpha'-2}v \in L^p(\Omega)$.
	Therefore, the convergence~\eqref{conv-1} reads:
	\begin{equation}
		\label{conv-1-new}
		(|v_m|^{\alpha'-2} v_m) \to w=(|v|^{\alpha'-2} v),\quad\textup{ in } L^p(\Omega)\textup{ as }m\to\infty.
	\end{equation}
	
	In order to show that $\mathscr{M}$ is relatively compact in $L^{(\alpha'-1)p}(\Omega)$, we have to show that every sequence $\{v_m\}_{m=1}^\infty \subset \mathscr{M}$ admits a convergent subsequence in $L^{(\alpha'-1)p}(\Omega)$. Observe that we can extract a subsequence, still denoted $\{v_m\}_{m=1}^\infty$ that weakly converges to an element $v$ in $L^{(\alpha'-1)p}(\Omega)$. Since $(\alpha'-1)p>1$, an application of Lemma~\ref{lem:2-0} gives:
	\begin{equation*}
		\begin{aligned}
			&\left|\|v_m\|_{L^{(\alpha'-1)p}(\Omega)}-\|v\|_{L^{(\alpha'-1)p}(\Omega)}\right|
			=\left|\left(\int_{\Omega} |v_m|^{(\alpha'-1)p} \dd x\right)^{\frac{1}{(\alpha'-1)p}}-\left(\int_{\Omega} |v|^{(\alpha'-1)p} \dd x\right)^{\frac{1}{(\alpha'-1)p}}\right|\\
			&\le\left|\int_{\Omega} |v_m|^{(\alpha'-1)p} \dd x-\int_{\Omega} |v|^{(\alpha'-1)p} \dd x\right|^{\frac{1}{(\alpha'-1)p}}
			=\left|\int_{\Omega} \left||v_m|^{\alpha'-2}v_m\right|^p \dd x-\int_{\Omega} \left||v|^{\alpha'-2}v\right|^p \dd x\right|^{\frac{1}{(\alpha'-1)p}}\\
			&=\left|\left\||v_m|^{\alpha'-2}v_m\right\|_{L^p(\Omega)}^p-\left\||v|^{\alpha'-2}v\right\|_{L^p(\Omega)}^p\right|^{\frac{1}{(\alpha'-1)p}}.
		\end{aligned}
	\end{equation*}
	
	An application of~\eqref{conv-1-new} gives that the right-hand side of the latter term tends to zero as $m \to \infty$, thus establishing that
	\begin{equation*}
		\|v_m\|_{L^{(\alpha'-1)p}(\Omega)}\to\|v\|_{L^{(\alpha'-1)p}(\Omega)},\quad\textup{ as }m\to\infty.
	\end{equation*}
	
	Since the space $L^{(\alpha'-1)p}(\Omega)$ is uniformly convex, an application of Theorem~5.12-3 of~\cite{PGCLNFAA} gives that 
	\begin{equation*}
		v_m \to v,\quad\textup{ in }L^{(\alpha'-1)p}(\Omega)\textup{ as }m\to\infty,
	\end{equation*}
	thus establishing the sought relative compactness.
	
	The established relative compactness of the set $\mathscr{M}$ in $L^{(\alpha'-1)p}(\Omega)$ and the sixth convergence in the process~\eqref{conv-proc} (which in turn implies that the time-derivatives in the sense of distributions are uniformly bounded) allow us to apply the Dubinskii compactness theorem (Lemma~\ref{Dub:dis}) with $A_0=L^{(\alpha'-1)p}(\Omega)$, $A_1=W^{-1,p'}(\Omega)$, $q_0=q_1=2$, so that
	\begin{equation}
		\label{conv-2}
		|\Pi_k \bm{u}_{\ell,k}|^{\alpha-2} \Pi_k \bm{u}_{\ell,k} \to w_\ell,\quad\textup{ in } L^2(0,T;L^{(\alpha'-1)p}(\Omega)) \textup{ as } k \to 0,
	\end{equation}
	where, once again, the monotonicity of the mapping $\xi\in\mathbb{R} \mapsto |\xi|^{\alpha-2} \xi$, the first convergence in the process~\eqref{conv-proc} and Theorem~9.13-2 of~\cite{PGCLNFAA} imply that
	\begin{equation*}
		w_\ell = |u_\ell|^{\alpha-2} u_\ell.
	\end{equation*}
	
	Second, we show that $v_\ell=|u_\ell|^\frac{\alpha-2}{2} u_\ell$. Given any $u \in W_0^{1,p}(\Omega)$, we look for a number $\beta \in \mathbb{R}$ for which
	\begin{equation*}
		\left||u|^\frac{\alpha-2}{2} u\right|^\frac{\beta-2}{2} |u|^\frac{\alpha-2}{2} u =u.
	\end{equation*}
	
	We observe that a sufficient condition insuring this is that $\beta$ satisfies
	\begin{equation*}
		\left(\dfrac{\alpha-2}{2}+1\right)\dfrac{\beta-2}{2}+\dfrac{\alpha-2}{2}=0,
	\end{equation*}
	which is equivalent to writing
	\begin{equation*}
		\beta= 2+2\left(\dfrac{2-\alpha}{2}\dfrac{2}{\alpha}\right)=\dfrac{4}{\alpha}.
	\end{equation*}
	
	Let $v:=|u|^\frac{\alpha-2}{2} u$ and observe that if $\beta=4/\alpha$ then $|v|^\frac{\beta-2}{2} v \in W_0^{1,p}(\Omega)$ so that the set
	\begin{equation*}
		\tilde{S}:=\{v;(|v|^\frac{(4/\alpha)-2}{2}v) \in W_0^{1,p}(\Omega)\}
	\end{equation*}
	is non-empty. Define the semi-norm
	\begin{equation*}
		\tilde{M}(v):=\left\|\nabla(|v|^\frac{(4/\alpha)-2}{2}v)\right\|_{\bm{L}^p(\Omega)}^\frac{\alpha}{2}=\left\|\nabla(|v|^\frac{2-\alpha}{\alpha} v)\right\|_{\bm{L}^p(\Omega)}^\frac{\alpha}{2},\quad\textup{ for all }v \in \tilde{S},
	\end{equation*}
	and define the set 
	\begin{equation*}
		\tilde{\mathscr{M}}:=\{v\in \tilde{S}; \tilde{M}(v)\le 1\}.
	\end{equation*}
	
	An application of the Poincar\'e-Friedrichs inequality gives that there exists a constant $c_0=c_0(\Omega)>0$ such that
	\begin{equation*}
		\begin{aligned}
			1 &\ge \tilde{M}(v) \ge c_0^\frac{\alpha}{2} \left\||v|^\frac{2-\alpha}{\alpha} v\right\|_{W_0^{1,p}(\Omega)}^\frac{\alpha}{2} \ge c_0^\frac{\alpha}{2} \left\||v|^\frac{2-\alpha}{\alpha} v\right\|_{L^p(\Omega)}^\frac{\alpha}{2}
			=c_0^\frac{\alpha}{2} \left(\int_{\Omega} \left||v|^\frac{2-\alpha}{\alpha} v\right|^p\dd x\right)^{\alpha/(2p)}\\
			&=c_0^\frac{\alpha}{2} \left(\int_{\Omega} |v|^\frac{2p}{\alpha}\dd x\right)^{\alpha/(2p)}
			=c_0^\frac{\alpha}{2} \|v\|_{L^\frac{2p}{\alpha}(\Omega)}.
		\end{aligned}
	\end{equation*}
	
	Let $\{v_m\}_{m=1}^\infty$ be a sequence in $\tilde{\mathscr{M}}$. Since, by the Rellich-Kondra\v{s}ov theorem, we have that $W_0^{1,p}(\Omega) \hookrightarrow\hookrightarrow L^p(\Omega)$ we obtain that, up to passing to a subsequence, there exists an element $w \in L^p(\Omega)$ such that
	\begin{equation}
		\label{conv-3}
		(|v_m|^\frac{2-\alpha}{2} v_m) \to w,\quad\textup{ in } L^p(\Omega)\textup{ as }m\to\infty.
	\end{equation}
	
	Since $1<\alpha<2$ and $2.8\le p \le 5$, it thus results that $1 \le p' <2<p<\frac{2p}{\alpha}<2p<\infty$ and that $\{v_m\}_{m=1}^\infty$ is bounded in $L^\frac{2p}{\alpha}(\Omega)$. The reflexivity of $L^\frac{2p}{\alpha}(\Omega)$ puts us in a position to apply the Banach-Eberlein-Smulian theorem (cf., e.g., Theorem~5.14-4 of~\cite{PGCLNFAA}) and extract a subsequence, still denoted $\{v_m\}_{m=1}^\infty$, that weakly converges to an element $v \in L^{p'}(\Omega)$.
	Consider the mapping
	\begin{equation*}
		v \in L^{p'}(\Omega) \mapsto (|v|^\frac{2-\alpha}{2} v) \in L^p(\Omega),
	\end{equation*}
	and observe that this mapping is hemi-continuous and monotone, since the mapping $\xi\in \mathbb{R} \to (|\xi|^\frac{2-\alpha}{2}\xi) \in \mathbb{R}$ is continuous and monotone.
	Therefore, an application of Theorem~9.13-2 of~\cite{PGCLNFAA} gives that $w=|v|^\frac{2-\alpha}{\alpha}v \in L^p(\Omega)$.
	Hence, the convergence~\eqref{conv-1} reads:
	\begin{equation}
		\label{conv-3-new}
		(|v_m|^\frac{2-\alpha}{\alpha} v_m) \to w=(|v|^\frac{2-\alpha}{\alpha} v),\quad\textup{ in } L^p(\Omega)\textup{ as }m\to\infty.
	\end{equation}
	
	In order to show that $\tilde{\mathscr{M}}$ is relatively compact in $L^\frac{2p}{\alpha}(\Omega)$, we have to show that every sequence $\{v_m\}_{m=1}^\infty \subset \tilde{\mathscr{M}}$ admits a convergent subsequence in $L^\frac{2p}{\alpha}(\Omega)$.
	
	Since $2p/\alpha>1$, an application of Lemma~\ref{lem:2-0} gives:
	\begin{equation*}
		\begin{aligned}
			&\left|\|v_m\|_{L^\frac{2p}{\alpha}(\Omega)}-\|v\|_{L^\frac{2p}{\alpha}(\Omega)}\right|
			=\left|\left(\int_{\Omega} |v_m|^{\frac{2p}{\alpha}} \dd x\right)^{\frac{\alpha}{2p}}-\left(\int_{\Omega} |v|^{\frac{2p}{\alpha}} \dd x\right)^{\frac{\alpha}{2p}}\right|\\
			&\le\left|\int_{\Omega} |v_m|^{\frac{2p}{\alpha}} \dd x-\int_{\Omega} |v|^{\frac{2p}{\alpha}} \dd x\right|^{\frac{\alpha}{2p}}
			=\left|\int_{\Omega} \left||v_m|^{\frac{2-\alpha}{\alpha}}v_m\right|^p \dd x-\int_{\Omega} \left||v|^{\frac{2-\alpha}{\alpha}}v\right|^p \dd x\right|^{\frac{\alpha}{2p}}\\
			&=\left|\left\||v_m|^{\frac{2-\alpha}{\alpha}}v_m\right\|_{L^p(\Omega)}^p-\left\||v|^{\frac{2-\alpha}{\alpha}}v\right\|_{L^p(\Omega)}^p\right|^{\frac{\alpha}{2p}}.
		\end{aligned}
	\end{equation*}
	
	An application of~\eqref{conv-3-new} gives that the right-hand side of the latter term tends to zero as $m \to \infty$, thus establishing that
	\begin{equation*}
		\|v_m\|_{L^\frac{2p}{\alpha}(\Omega)}\to\|v\|_{L^\frac{2p}{\alpha}(\Omega)},\quad\textup{ as }m\to\infty.
	\end{equation*}
	
	Since the space $L^\frac{2p}{\alpha}(\Omega)$ is uniformly convex, an application of Theorem~5.12-3 of~\cite{PGCLNFAA} gives that 
	\begin{equation*}
		v_m \to v,\quad\textup{ in } L^\frac{2p}{\alpha}(\Omega) \textup{ as }m\to\infty,
	\end{equation*}
	thus establishing the sought relative compactness.
	
	The latter shows that
	\begin{equation*}
		v_m \to v,\quad\textup{ in } L^\frac{2p}{\alpha}(\Omega) \textup{ as } m\to\infty,
	\end{equation*}
	in turn implying that the set $\tilde{\mathscr{M}}$ is relatively compact in $L^\frac{2p}{\alpha}(\Omega)$, as it was to be proved. The established relative compactness of the set $\tilde{\mathscr{M}}$ in $L^\frac{2p}{\alpha}(\Omega)$ and the fourth convergence in the process~\eqref{conv-proc} (which in turn implies that the time-derivatives in the sense of distributions are bounded independently of $m$) allow us to apply the Dubinskii compactness theorem (Lemma~\ref{Dub:dis}) with $A_0=L^\frac{2p}{\alpha}(\Omega)$, $A_1=L^2(\Omega)$, $q_0=q_1=2$, so that
	\begin{equation}
		\label{conv-4}
		|\Pi_k \bm{u}_{\ell,k}|^\frac{\alpha-2}{2} \Pi_k \bm{u}_{\ell,k} \to v_\ell,\quad\textup{ in } L^2(0,T;L^\frac{2p}{\alpha}(\Omega)) \textup{ as } k \to 0,
	\end{equation}
	where, once again, the monotonicity of the mapping $\xi\in\mathbb{R} \mapsto |\xi|^\frac{\alpha-2}{2} \xi$, the first convergence in the process~\eqref{conv-proc} and Theorem~9.13-2 of~\cite{PGCLNFAA} imply that
	\begin{equation*}
		v_\ell = |u_\ell|^\frac{\alpha-2}{2} u_\ell,
	\end{equation*}
	showing that the entire convergence process~\eqref{conv-proc} holds.
	
	We are left to show that the weak-star limit $u_\ell$ is a solution for Problem~\ref{Pkappa}. Let $v \in \mathcal{D}(\Omega)$ and let $\psi \in \mathcal{C}^1([0,T])$. For each $0 \le n \le N-1$, multiply~\eqref{penalty:seq} by $\{v \psi(nk)\}$, getting
	\begin{equation}
		\label{step:1}
		\begin{aligned}
			&\dfrac{\psi(nk)}{k}\int_{\Omega}\{|u_{\ell,k}^{n+1}|^{\alpha-2} u_{\ell,k}^{n+1} - |u_{\ell,k}^{n}|^{\alpha-2}u_{\ell,k}^{n}\} v \dd x\\
			&\quad+\psi(nk)\int_{\Omega}\mu |\nabla u_{\ell,k}^{n+1}|^{p-2} \nabla u_{\ell,k}^{n+1} \cdot \nabla v\dd x
			-\psi(nk)\int_{\Omega}\dfrac{\{u_{\ell,k}^{n+1}\}^{-}}{\ell} v \dd x\\
			&=\int_{\Omega}\left(\dfrac{1}{k} \int_{nk}^{(n+1)k} \tilde{a}(t) \dd t\right) v \psi(nk) \dd x.
		\end{aligned}
	\end{equation}
	
	Multiplying~\eqref{step:1} by $k$ and summing over $0 \le n \le N-1$, we obtain
	\begin{equation}
		\label{step:2}
		\begin{aligned}
			&\sum_{n=0}^{N-1} k \int_{\Omega}\dfrac{|u_{\ell,k}^{n+1}|^{\alpha-2} u_{\ell,k}^{n+1} - |u_{\ell,k}^{n}|^{\alpha-2}u_{\ell,k}^{n}}{k} v \psi(nk) \dd x\\
			&\quad+\sum_{n=0}^{N-1} k \int_{\Omega}\mu |\nabla u_{\ell,k}^{n+1}|^{p-2} \nabla u_{\ell,k}^{n+1} \cdot \nabla(\psi(nk) v)\dd x\\
			&\quad-\dfrac{1}{\ell} \sum_{n=0}^{N-1} k \int_{\Omega} \{u_{\ell,k}^{n+1}\}^{-} v \psi(nk) \dd x
			=\sum_{n=0}^{N-1} k \int_{\Omega}\left(\dfrac{1}{k} \int_{nk}^{(n+1)k} \tilde{a}(t) \dd t\right) v \psi(nk) \dd x.
		\end{aligned}
	\end{equation}
	
	We define $\psi_k(t):=\psi(nk)$, $nk \le t \le (n+1)k$ and $0 \le n \le N-1$. Equation~\eqref{step:2} can be thus re-arranged as follows:
	\begin{equation}
		\label{step:3}
		\begin{aligned}
			&\int_{0}^{T} \int_{\Omega} D_k(|\Pi_k \bm{u}_{\ell,k}|^{\alpha-2}\Pi_k \bm{u}_{\ell,k}) v \dd x \psi_k(t) \dd t\\
			&\quad-\int_{0}^{T} \int_{\Omega} \nabla \cdot \left(\mu |\nabla (\Pi_k \bm{u}_{\ell,k})|^{p-2} \nabla (\Pi_k \bm{u}_{\ell,k})\right) v\dd x \psi_k(t) \dd t\\
			&\quad-\dfrac{1}{\ell} \int_{0}^{T} \int_{\Omega} \{\Pi_k \bm{u}_{\ell,k}\}^{-} v\dd x \psi_k(t) \dd t
			=\int_{0}^{T} \left(\int_{\Omega} \tilde{a}(t) v \dd x\right) \psi_k(t) \dd t.
		\end{aligned}
	\end{equation}
	
	Letting $k \to 0$ and exploiting the convergence process~\eqref{conv-proc} and the Riemann integrability of $\psi$, we obtain:
	\begin{equation}
		\label{step:4}
		\begin{aligned}
			&\int_{0}^{T} \left\langle \dfrac{\dd}{\dd t}\left(|u_\ell|^{\alpha-2} u_\ell\right), v \right\rangle_{W^{-1,p'}(\Omega), W_0^{1,p}(\Omega)} \psi(t) \dd t
			+\int_{0}^{T} \langle g_\ell(t), v \rangle_{W^{-1,p'}(\Omega), W_0^{1,p}(\Omega)} \psi(t) \dd t\\
			&=\int_{0}^{T} \int_{\Omega} \tilde{a}(t) v \dd x \psi(t) \dd t.
		\end{aligned}
	\end{equation}
	
	Let us rearrange the first term on the left-hand side of equation~\eqref{step:2} as follows:
	\begin{equation*}
		\begin{aligned}
			&\dfrac{1}{k} \sum_{n=0}^{N-1} k \int_{\Omega} \{|u_{\ell,k}^{n+1}|^{\alpha-2} u_{\ell,k}^{n+1} - |u_{\ell,k}^{n}|^{\alpha-2} u_{\ell,k}^{n}\} v \psi(nk) \dd x\\
			&=\int_{\Omega} \Bigg\{ \left[|u_{\ell,k}^{1}|^{\alpha-2} u_{\ell,k}^{1} -|u_0|^{\alpha-2}u_0\right] v \psi(0)\\
			&\qquad+ \left[|u_{\ell,k}^{2}|^{\alpha-2} u_{\ell,k}^{2} -|u_{\ell,k}^{1}|^{\alpha-2}u_{\ell,k}^{1}\right] v \psi(k)\\
			&\qquad + \dots \\
			&\qquad +\left[|u_{\ell,k}^{N-1}|^{\alpha-2} u_{\ell,k}^{N-1} -|u_{\ell,k}^{N-2}|^{\alpha-2}u_{\ell,k}^{N-2}\right] v \psi((N-1)k)\\
			&\qquad +\left[|u_{\ell,k}^{N}|^{\alpha-2} u_{\ell,k}^{N} -|u_{\ell,k}^{N-1}|^{\alpha-2}u_{\ell,k}^{N-1}\right] v \psi(T)
			\Bigg\} \dd x\\
			&=\int_{\Omega} -|u_0|^{\alpha-2} u_0 v \psi(0) \dd x\\
			&\quad+\int_{\Omega} \Bigg\{\left[-|u_{\ell,k}^1|^{\alpha-2} u_{\ell,k}^1 v (\psi(k)-\psi(0))\right]
			+\left[-|u_{\ell,k}^2|^{\alpha-2} u_{\ell,k}^2 v (\psi(2k)-\psi(k))\right]\\
			&\qquad+ \dots\\
			&\qquad +\left[-|u_{\ell,k}^{N-2}|^{\alpha-2} u_{\ell,k}^{N-2} v (\psi((N-1)k)-\psi((N-2)k))\right]\\
			&\qquad + \left[-|u_{\ell,k}^{N-1}|^{\alpha-2} u_{\ell,k}^{N-1} v (\psi(T)-\psi((N-1)k))\right]
			\Bigg\} \dd x\\
			&\quad+\int_{\Omega} |u_{\ell,k}^N|^{\alpha-2} u_{\ell,k}^N v \psi(T) \dd x\\
			&=-\sum_{n=0}^{N-1} k \int_{\Omega} |u_{\ell,k}^n|^{\alpha-2} u_{\ell,k}^n v \left[\dfrac{\psi(nk)-\psi((n-1)k)}{k}\right] \dd x\\
			&\quad+\int_{\Omega} |u_{\ell,k}^N|^{\alpha-2} u_{\ell,k}^N v \psi(T) \dd x-\int_{\Omega} |u_0|^{\alpha-2} u_0 v \psi(0) \dd x.
		\end{aligned}
	\end{equation*}
	
	Therefore, equation~\eqref{step:2} can be thoroughly re-arranged as follows:
	\begin{equation}
		\label{step:5}
		\begin{aligned}
			&\int_{\Omega} |u_{\ell,k}^N|^{\alpha-2} u_{\ell,k}^N v \psi(T) \dd x-\int_{\Omega} |u_0|^{\alpha-2} u_0 v \psi(0) \dd x\\
			&\quad-\sum_{n=0}^{N-1} k \int_{\Omega} |u_{\ell,k}^n|^{\alpha-2} u_{\ell,k}^n v \left[\dfrac{\psi(nk)-\psi((n-1)k)}{k}\right] \dd x\\
			&\quad +\sum_{n=0}^{N-1} k \int_{\Omega}\mu |\nabla u_{\ell,k}^{n+1}|^{p-2} \nabla u_{\ell,k}^{n+1} \cdot \nabla(\psi(nk) v)\dd x\\
			&\quad-\dfrac{1}{\ell} \sum_{n=0}^{N-1} k \int_{\Omega} \{u_{\ell,k}^{n+1}\}^{-} v \psi(nk) \dd x
			=\sum_{n=0}^{N-1} k \int_{\Omega}\left(\dfrac{1}{k} \int_{nk}^{(n+1)k} \tilde{a}(t) \dd t\right) v \psi(nk) \dd x.
		\end{aligned}
	\end{equation}
	
	Letting $k \to 0$ in~\eqref{step:5} and applying~\eqref{conv-proc}, we obtain that:
	\begin{equation}
		\label{step:6}
		\begin{aligned}
			&-\int_{0}^{T} \int_{\Omega} |u_\ell|^{\alpha-2} u_\ell v  \dd x \dfrac{\dd \psi}{\dd t}\dd t
			+\int_{0}^{T} \langle g_\ell(t), v \rangle_{W^{-1,p'}(\Omega), W_0^{1,p}(\Omega)} \psi(t) \dd t\\
			&\quad+\int_{0}^{T} \int_{\Omega} [\chi_\ell \psi(T)-|u_0|^{\alpha-2}u_0 \psi(0)] v \dd x = \int_{0}^{T} \int_{\Omega} \tilde{a}(t) v \dd x\psi(t)  \dd t.
		\end{aligned}
	\end{equation}
	
	Observe that an application of the Sobolev embedding theorem (cf., e.g., Theorem~6.6-1 of~\cite{PGCLNFAA}) gives $L^{\alpha'}(\Omega) \hookrightarrow W^{-1,p'}(\Omega)$. An integration by parts in~\eqref{step:4}, and an application of the continuity of $u_\ell$ established in Lemma~\ref{lem:3} give:
	\begin{equation}
		\label{step:7}
		\begin{aligned}
			&-\int_{0}^{T} \int_{\Omega} |u_\ell|^{\alpha-2} u_\ell v \dd x \dfrac{\dd \psi}{\dd t} \dd t
			+\langle|u_\ell(T)|^{\alpha-2} u_\ell(T), \psi(T) v\rangle_{W^{-1,p'}(\Omega), W_0^{1,p}(\Omega)}\\
			&\quad -\int_{\Omega} |u_\ell(0)|^{\alpha-2} u_\ell(0) \psi(0) v \dd x + \int_{0}^{T} \langle g_\ell(t),  v \rangle_{W^{-1,p'}(\Omega), W_0^{1,p}(\Omega)} \psi(t) \dd t\\
			&=\int_{0}^{T}\int_{\Omega} \tilde{a}(t) v \dd x \psi(t) \dd t.
		\end{aligned}
	\end{equation}
	
	Comparing equations~\eqref{step:6} and~\eqref{step:7} gives
	\begin{equation}
		\label{step:8}
		\begin{aligned}
			&\langle|u_\ell(T)|^{\alpha-2} u_\ell(T), \psi(T) v\rangle_{W^{-1,p'}(\Omega), W_0^{1,p}(\Omega)} -\int_{\Omega} |u_\ell(0)|^{\alpha-2} u_\ell(0) \psi(0) v \dd x\\
			&=\int_{0}^{T} \int_{\Omega} [\chi_\ell \psi(T)-|u_0|^{\alpha-2}u_0 \psi(0)] v \dd x
		\end{aligned}
	\end{equation}
	
	Since $\psi \in \mathcal{C}^1([0,T])$ is arbitrarily chosen, let us specialize $\psi$ in~\eqref{step:8} in a way such that $\psi(0)=0$. We obtain
	\begin{equation}
		\label{step:9}
		\langle|u_\ell(T)|^{\alpha-2} u_\ell(T) - \chi_\ell, \psi(T) v\rangle_{W^{-1,p'}(\Omega), W_0^{1,p}(\Omega)}=0,\quad\textup{ for all }v \in \mathcal{D}(\Omega).
	\end{equation}
	
	Since the duality in~\eqref{step:9} is continuous with respect to $v$, and since $\mathcal{D}(\Omega)$ is, by definition, dense in $W_0^{1,p}(\Omega)$, we immediately infer that:
	\begin{equation}
		\label{step:10}
		|u_\ell(T)|^{\alpha-2} u_\ell(T) = \chi_\ell \in L^{\alpha'}(\Omega).
	\end{equation}
	
	It is easy to observe that:
	\begin{equation*}
		|\chi_\ell|^{\alpha'-2} \chi_\ell = \left||u_\ell(T)|^{\alpha-2} u_\ell(T)\right|^{\alpha'-2} \chi_\ell =|u_\ell(T)|^{2-\alpha} \left[|u_\ell(T)|^{\alpha-2} u_\ell(T)\right]=u_\ell(T) \in L^\alpha(\Omega).
	\end{equation*}
	
	Let us now specialize $\psi$ in~\eqref{step:8} in a way such that $\psi(T)=0$. We obtain
	\begin{equation}
		\label{step:11}
		\int_{\Omega} \left(|u_\ell(0)|^{\alpha-2} u_\ell(0) - |u_0|^{\alpha-2}u_0 \right)\psi(0) v \dd x=0,\quad\textup{ for all }v \in \mathcal{D}(\Omega).
	\end{equation}
	
	Since the integration in~\eqref{step:11} is continuous with respect to $v$, and since $\mathcal{D}(\Omega)$ is, by definition, dense in $W_0^{1,p}(\Omega)$, we immediately infer that:
	\begin{equation*}
		|u_\ell(0)|^{\alpha-2} u_\ell(0) = |u_0|^{\alpha-2}u_0,
	\end{equation*}
	so that the injectivity of the monotone and hemi-continuous operator $\xi \mapsto |\xi|^{\alpha-2} \xi$ in turn implies that:
	\begin{equation}
		\label{step:12}
		u_\ell(0)=u_0 \in K.
	\end{equation}
	
	The last thing to check is that $g_\ell=B_\ell(u_\ell)$. For each $0 \le n \le N-1$, we multiply~\eqref{penalty:seq} by $k u_{\ell,k}^{n+1}$ and apply Lemma~\ref{lem:7}, thus getting
	\begin{equation*}
		\begin{aligned}
			&\dfrac{1}{\alpha'} \sum_{n=0}^{N-1} \left\{\left\||u_{\ell,k}^{n+1}|^{\alpha-2}u_{\ell,k}^{n+1}\right\|_{L^{\alpha'}(\Omega)}^{\alpha'} - \left\||u_{\ell,k}^{n}|^{\alpha-2}u_{\ell,k}^{n}\right\|_{L^{\alpha'}(\Omega)}^{\alpha'}\right\}\\
			&\quad+\sum_{n=0}^{N-1} k \left\langle B_\ell(u_{\ell,k}^{n+1}), u_{\ell,k}^{n+1} \right\rangle_{W^{-1,p'}(\Omega), W_0^{1,p}(\Omega)}\\
			&\le \sum_{n=0}^{N-1} k \int_{\Omega} \left(\dfrac{1}{k}\int_{nk}^{(n+1)k} \tilde{a}(t)\dd t\right) u_{\ell,k}^{n+1} \dd x,
		\end{aligned}
	\end{equation*}
	which in turn implies:
	\begin{equation}
		\label{step:13}
		\begin{aligned}
			&\dfrac{1}{\alpha'} \left\||u_{\ell,k}^{N}|^{\alpha-2}u_{\ell,k}^{N}\right\|_{L^{\alpha'}(\Omega)}^{\alpha'}\\
			&\quad+\int_{0}^{T} \left\langle B_\ell(\Pi_k \bm{u}_{\ell,k}), \Pi_k \bm{u}_{\ell,k} \right\rangle_{W^{-1,p'}(\Omega), W_0^{1,p}(\Omega)}\\
			&\le \int_{0}^{T} \int_{\Omega} \tilde{a}(t) \Pi_k \bm{u}_{\ell,k} \dd x \dd t
			+\dfrac{1}{\alpha'} \left\||u_0|^{\alpha-2}u_0\right\|_{L^{\alpha'}(\Omega)}^{\alpha'}.
		\end{aligned}
	\end{equation}
	
	Passing to the $\liminf$ as $k \to 0$ in~\eqref{step:13} and keeping in mind the convergence process~\eqref{conv-proc} as well as the identities~\eqref{step:10}--\eqref{step:12} gives, on the one hand:
	\begin{equation}
		\label{step:14}
		\begin{aligned}
			&\dfrac{1}{\alpha'} \|u_{\ell}(T)\|_{L^{\alpha}(\Omega)}^{\alpha}
			+\liminf_{k \to 0}\int_{0}^{T} \left\langle B_\ell(\Pi_k \bm{u}_{\ell,k}), \Pi_k \bm{u}_{\ell,k} \right\rangle_{W^{-1,p'}(\Omega), W_0^{1,p}(\Omega)} \dd t\\
			&=\dfrac{1}{\alpha'} \left\||u_{\ell}(T)|^{\alpha-2}u_{\ell}(T)\right\|_{L^{\alpha'}(\Omega)}^{\alpha'}
			+\liminf_{k \to 0}\int_{0}^{T} \left\langle B_\ell(\Pi_k \bm{u}_{\ell,k}), \Pi_k \bm{u}_{\ell,k} \right\rangle_{W^{-1,p'}(\Omega), W_0^{1,p}(\Omega)} \dd t\\
			&\le \int_{0}^{T} \int_{\Omega} \tilde{a}(t) u_\ell \dd x \dd t+
			\dfrac{1}{\alpha'} \left\||u_0|^{\alpha-2}u_0\right\|_{L^{\alpha'}(\Omega)}^{\alpha'}=\int_{0}^{T} \int_{\Omega} \tilde{a}(t) v \dd x \psi(t) \dd t+\dfrac{1}{\alpha'}\|u_0\|_{L^\alpha(\Omega)}^\alpha.
		\end{aligned}
	\end{equation}
	
	On the other hand, the specializations $v=u_\ell$ and $\psi\equiv 1$ in~\eqref{step:4}, and an application of Lemma~\ref{lem:8} give:
	\begin{equation}
		\label{step:15}
		\begin{aligned}
			&\dfrac{1}{\alpha'}\|u_\ell(T)\|_{L^\alpha(\Omega)}^\alpha +\int_{0}^{T} \langle g_\ell(t), u_\ell(t) \rangle_{W^{-1,p'}(\Omega), W_0^{1,p}(\Omega)} \psi(t) \dd t\\
			&=\int_{0}^{T} \int_{\Omega} \tilde{a}(t) u_\ell(t) \dd x \psi(t) \dd t+\dfrac{1}{\alpha'}\|u_0\|_{L^\alpha(\Omega)}^\alpha.
		\end{aligned}
	\end{equation}
	
	Combining~\eqref{step:14} and~\eqref{step:15} gives:
	\begin{equation}
		\label{step:16}
		\liminf_{k \to 0}\int_{0}^{T} \left\langle B_\ell(\Pi_k \bm{u}_{\ell,k}), \Pi_k \bm{u}_{\ell,k} \right\rangle_{W^{-1,p'}(\Omega), W_0^{1,p}(\Omega)}
		\le \int_{0}^{T} \langle g_\ell(t), u_\ell(t) \rangle_{W^{-1,p'}(\Omega), W_0^{1,p}(\Omega)} \psi(t) \dd t,
	\end{equation}
	
	We now exploit the method of Minty, known as Minty's trick~\cite{Minty1962}. Let $w\in L^p(0,T;W_0^{1,p}(\Omega))$. An application of the strict monotonicity of the operator $B_\ell$ defined in~\eqref{Bkappa} and~\eqref{step:16} gives:
	\begin{equation}
		\label{step:17}
		\begin{aligned}
			&\int_{0}^{T} \langle g_\ell-B_\ell w, u_\ell-w\rangle_{W^{-1,p'}(\Omega), W_0^{1,p}(\Omega)} \dd t\\
			&\ge \liminf_{k \to 0} \int_{0}^{T} \langle B_\ell (\Pi_k \bm{u}_{\ell,k})-B_\ell w, \Pi_k \bm{u}_{\ell,k}-w \rangle_{W^{-1,p'}(\Omega), W_0^{1,p}(\Omega)} \dd t \ge 0.
		\end{aligned}
	\end{equation}
	
	Let $\lambda>0$ and specialize $w=u_\ell -\lambda v$ in~\eqref{step:17}, where $v$ is arbitrarily chosen in $L^p(0,T;W_0^{1,p}(\Omega))$. We obtain that:
	\begin{equation}
		\label{step:18}
		\int_{0}^{T} \langle g_\ell-B_\ell(u_\ell -\lambda v)), \lambda v \rangle_{W^{-1,p'}(\Omega), W_0^{1,p}(\Omega)} \dd t \ge 0.
	\end{equation}
	
	Dividing~\eqref{step:18} by $\lambda >0$ and letting $\lambda \to 0^+$ gives:
	\begin{equation}
		\label{step:19}
		\int_{0}^{T} \langle g_\ell-B_\ell (u_\ell), v \rangle_{W^{-1,p'}(\Omega), W_0^{1,p}(\Omega)} \dd t \ge 0.
	\end{equation}
	
	Since $v\in L^p(0,T;W_0^{1,p}(\Omega))$ is arbitrary, we obtain that:
	\begin{equation}
		\label{step:20}
		g_\ell = B_\ell (u_\ell) \in L^\infty(0,T;W^{-1,p'}(\Omega)),
	\end{equation}
	which finally implies that $u_\ell$ is a solution of Problem~\ref{Pkappa}. This completes the proof.
\end{proof}

The next step, which constitutes the final step of the existence result we want to prove, is the passage to the limit as $\ell \to 0^+$ as well as the recovery of the actual model governing the variation of shallow ice sheets in time.

In all what follows, we exploit the $p$-norm in $\mathbb{R}^2$, namely,
\begin{equation*}
|x|:=\left(|x_1|^p+|x_2|^p\right)^{1/p},\quad\textup{ for all }x=(x_1,x_2) \in \mathbb{R}^2.
\end{equation*}

The estimates~\eqref{est:1}, \eqref{est:3:1}, \eqref{est:5}, \eqref{est:5:1} and Theorem~\ref{thm:3} imply that there exists a constant $C>0$ independent of $\ell$ such that
\begin{equation}
	\label{boundukappa}
	\begin{aligned}
		\|u_\ell\|_{L^\infty(0,T;W_0^{1,p}(\Omega))} &\le C,\\
		\left\||u_\ell|^\frac{\alpha-2}{2}u_\ell\right\|_{L^\infty(0,T;L^2(\Omega))} &\le C,\\
		\left\|\dfrac{\dd}{\dd t}(|u_\ell|^\frac{\alpha-2}{2} u_\ell)\right\|_{L^2(0,T;L^2(\Omega))} &\le C,\\
		\left\||u_\ell|^{\alpha-2} u_\ell\right\|_{L^\infty(0,T;L^{\alpha'}(\Omega))} &\le C,\\
		\left\|\dfrac{\dd}{\dd t}(|u_\ell|^{\alpha-2}u_\ell)\right\|_{L^\infty(0,T;W^{-1,p'}(\Omega))} &\le C\left(1+\dfrac{1}{\ell}\right).
	\end{aligned}
\end{equation}

By the Banach-Alaoglu-Bourbaki theorem (cf., e.g., Theorem~3.6 of~\cite{Brez11}) we infer that, up to passing to a subsequence still denoted by $\{u_\ell\}_{\ell>0}$, the following convergences hold:
\begin{equation}
	\label{conv-proc-kappa}
	\begin{aligned}
		u_\ell \wsc u, &\textup{ in } L^\infty(0,T;W_0^{1,p}(\Omega)),\\
		|u_\ell|^\frac{\alpha-2}{2} u_\ell \wsc v, &\textup{ in } L^\infty(0,T;L^2(\Omega)),\\
		\dfrac{\dd}{\dd t}\left(|u_\ell|^\frac{\alpha-2}{2} u_\ell\right) \rightharpoonup \dfrac{\dd v}{\dd t}, &\textup{ in } L^2(0,T;L^2(\Omega)),\\
		|u_\ell|^{\alpha-2} u_\ell \wsc w, &\textup{ in } L^\infty(0,T;L^{\alpha'}(\Omega)).
	\end{aligned}
\end{equation}

Observe that the first convergence of~\eqref{conv-proc-kappa}, namely $u_\ell \wsc u$ in $L^\infty(0,T;W_0^{1,p}(\Omega))$, implies that
\begin{equation*}
	u_\ell \rightharpoonup u,\quad\textup{ in }L^2(0,T;L^2(\Omega)) \simeq L^2((0,T)\times \Omega) \textup{ as }\ell \to 0^+.
\end{equation*}

The third estimate in~\eqref{est:penalty} imply that $-\{u_\ell\}^{-} \to 0$ in $L^2(0,T;L^2(\Omega))$. 
Therefore, the continuity and the monotonicity of the negative part established in Lemma~\ref{lem:1} allow us to apply Theorem 9.13-2 of~\cite{PGCLNFAA}, getting:
\begin{equation}
	\label{sign}
	\{u\}^{-} =0 \quad \textup{ in } L^2(0,T;L^2(\Omega)),
\end{equation}
which means that $u(t) \ge 0$ a.e. in $\Omega$, for a.e. $t \in (0,T)$.

Using the same compactness argument as in Theorem~\ref{thm:3}, an application of Dubinskii's theorem (Lemma~\ref{Dub}) with $A_0:=L^\frac{2p}{\alpha}(\Omega)$, $A_1:=L^2(\Omega)$, $q_0=\infty$ and $q_1=2$ gives that:
\begin{equation*}
	\left\{|u_\ell|^\frac{\alpha-2}{2}u_\ell\right\}_{\ell>0} \textup{ strongly converges in } L^\infty(0,T;L^\frac{2p}{\alpha}(\Omega)).
\end{equation*}

Moreover, by Lemma~\ref{Dub:lemma}(a), it can be established that:
\begin{equation}
	\label{Dub:strong}
	\left\{|u_\ell|^\frac{\alpha-2}{2}u_\ell\right\}_{\ell>0} \textup{ strongly converges in } \mathcal{C}^0([0,T];L^2(\Omega)).
\end{equation}

The same monotonicity argument as in Theorem~\ref{thm:3} and the third convergence in~\eqref{conv-proc-kappa} in turn imply that the limit $v$ in~\eqref{Dub:strong} takes the following form:
\begin{equation*}
	v=|u|^\frac{\alpha-2}{2} u \in H^1(0,T;L^2(\Omega)).
\end{equation*}

The first convergence in~\eqref{conv-proc-kappa} gives that $u_\ell \rightharpoonup u$ in $L^2(0,T;W_0^{1,p}(\Omega))$. Let us now show that:
\begin{equation}
	\label{conv-abs}
	|u_\ell| \rightharpoonup u,\quad\textup{ in }L^2(0,T;W_0^{1,p}(\Omega)).
\end{equation}

Note that, thanks to Stampacchia's theorem~\cite{Stampacchia1965}, the announced convergence \emph{a priori} makes sense.
To prove that this convergence holds, fix any $v \in L^2(0,T;W^{-1,p'}(\Omega))$ and observe that:
\begin{equation*}
	\begin{aligned}
		&\int_{0}^{T} \left\langle v, |u_\ell|-u\right\rangle_{W^{-1,p'}(\Omega), W_0^{1,p}(\Omega)}\dd t
		=\int_{0}^{T} \left\langle v, \{u_\ell\}^{+}+\{u_\ell\}^{-}-u\right\rangle_{W^{-1,p'}(\Omega), W_0^{1,p}(\Omega)}\dd t\\
		&=\int_{0}^{T} \left\langle v, (\{u_\ell\}^{+}-\{u_\ell\}^{-})+2\{u_\ell\}^{-}-u\right\rangle_{W^{-1,p'}(\Omega), W_0^{1,p}(\Omega)}\dd t\\
		&=\int_{0}^{T} \left\langle v, u_\ell -u\right\rangle_{W^{-1,p'}(\Omega), W_0^{1,p}(\Omega)}\dd t
		+2\int_{0}^{T} \left\langle v,\{u_\ell\}^{-}\right\rangle_{W^{-1,p'}(\Omega), W_0^{1,p}(\Omega)}\dd t\\
		& \to 0, \quad\textup{ as }\ell \to 0^+,
	\end{aligned}
\end{equation*}
since the first term converges to zero by the first convergence in~\eqref{conv-proc-kappa} and the second term converges to zero thanks to~\eqref{sign}. The convergence~\eqref{conv-abs} is thus established.
An application of Lemma~\ref{lem:5} to the convergence~\eqref{Dub:strong} gives that:
\begin{equation*}
	\left\{|u_\ell|\right\}_{\ell>0} \textup{ strongly converges in } \mathcal{C}^0([0,T];L^\alpha(\Omega)),
\end{equation*}
and, so, it also strongly converges in $L^2(0,T;L^\alpha(\Omega))$. Combining this with the convergences~\eqref{sign} and~\eqref{conv-abs} gives that:
\begin{equation}
	\label{cont}
	|u_\ell| \to u,\quad\textup{ in } \mathcal{C}^0([0,T];L^\alpha(\Omega)).
\end{equation}

Define the linear and continuous operator $L_0:\mathcal{C}^0([0,T];L^\alpha(\Omega)) \to L^\alpha(\Omega)$ by:
\begin{equation*}
	L_0(v):=v(0),\quad\textup{ for all }v \in \mathcal{C}^0([0,T];L^\alpha(\Omega)).
\end{equation*}

By~\eqref{cont} and the continuity of $L_0$, we have that $|u_\ell(0)| \to u(0)$ in $L^\alpha(\Omega)$. However, since $u_\ell(0)=u_0 \ge 0$ for all $\ell>0$, we immediately deduce that $u(0)=u_0 \in K$ and so that the weak-star limit $u$ satisfies the expected initial condition.

Note that the variational equation in Problem~\ref{Pkappa} takes the following equivalent form:
\begin{equation}
	\label{eq:1}
	\dfrac{\dd}{\dd t}(|u_\ell(t)|^{\alpha-2} u_\ell(t)) -\nabla \cdot \left(\mu |\nabla u_\ell(t)|^{p-2} \nabla u_\ell(t)\right) -\dfrac{1}{\ell} \{u_\ell(t)\}^{-}=\tilde{a}(t),\quad \textup{ in }W^{-1,p'}(\Omega).
\end{equation}

Integrating~\eqref{eq:1} in $(0,T)$ gives:
\begin{equation*}
	\begin{aligned}
		&-\dfrac{1}{\ell}\int_{0}^{T} \{u_\ell(t)\}^{-} \dd t\\
		&=\int_{0}^{T} \tilde{a}(t) \dd t + \int_{0}^{T} \nabla \cdot \left(\mu |\nabla u_\ell(t)|^{p-2}\nabla u_\ell(t)\right) \dd t -\int_{0}^{T} \dfrac{\dd}{\dd t}(|u_\ell(t)|^{\alpha-2} u_\ell(t)) \dd t\\
		&=\int_{0}^{T} \tilde{a}(t) \dd t + \int_{0}^{T} \nabla \cdot \left(\mu |\nabla u_\ell(t)|^{p-2}\nabla u_\ell(t)\right) \dd t
		-\left[|u_\ell(T)|^{\alpha-2} u_\ell(T)-|u_0|^{\alpha-2} u_0\right]\\
		&=\int_{0}^{T} \tilde{a}(t) \dd t + \int_{0}^{T} \nabla \cdot \left(\mu |\nabla u_\ell(t)|^{p-2}\nabla u_\ell(t)\right) \dd t
		-\left[|u_\ell(T)|^{\alpha-2} u_\ell(T)-u_0^{\alpha-1}\right],\quad\textup{ in }W^{-1,p'}(\Omega),
	\end{aligned}
\end{equation*}
where the last equality holds since $u_0 \in K$ by assumption.

Let $v \in W_0^{1,p}(\Omega)$ be arbitrarily chosen satisfying $\|v\|_{W_0^{1,p}(\Omega)}=1$. By Stampacchia's theorem (cf. the seminal article~\cite{Stampacchia1965}) we have that $|v| \in W_0^{1,p}(\Omega)$ as well. We then have:
\begin{equation*}
	\begin{aligned}
		&\dfrac{1}{\ell}\left|\left\langle\int_{0}^{T}\{u_\ell(t)\}^{-}\dd t,v\right\rangle_{W^{-1,p'}(\Omega), W_0^{1,p}(\Omega)}\right|
		\le\left|\left\langle\int_{0}^{T} \tilde{a}(t) \dd t, v\right\rangle_{W^{-1,p'}(\Omega), W_0^{1,p}(\Omega)}\right|\\
		&\quad+\left|\left\langle\int_{0}^{T} \nabla\cdot\left(\mu |\nabla u_\ell(t)|^{p-2} \nabla u_\ell(t)\right)\dd t,v\right\rangle_{W^{-1,p'}(\Omega), W_0^{1,p}(\Omega)}\right|\\
		&\quad+\left|\int_{\Omega} \underbrace{|u_\ell(T)|^{\alpha-2} u_\ell(T)}_{\in L^{\alpha'}(\Omega) \textup{ by }\eqref{step:10}} v \dd x\right|
		+\int_{\Omega} u_0^{\alpha-1} |v| \dd x\\
		&\le \int_{0}^{T} \|\tilde{a}(t)\|_{W^{-1,p'}(\Omega)} \|v\|_{W_0^{1,p}(\Omega)}\dd t
		+T\left\|\nabla \cdot\left(\mu |\nabla u_\ell|^{p-2} \nabla u_\ell\right)\right\|_{L^\infty(0,T;W^{-1,p'}(\Omega))} \|v\|_{W_0^{1,p}(\Omega)}\\
		&\quad+\left\||u_\ell(T)|^{\alpha-2} u_\ell(T)\right\|_{L^{\alpha'}(\Omega)} \|v\|_{L^\alpha(\Omega)} +\|u_0^{\alpha-1}\|_{L^{\alpha'}(\Omega)} \|v\|_{L^\alpha(\Omega)}\\
		&\le \|\tilde{a}\|_{W^{1,p}(0,T;\mathcal{C}^0(\overline{\Omega}))}
		+T\left\|\nabla \cdot\left(\mu |\nabla u_\ell|^{p-2} \nabla u_\ell\right)\right\|_{L^\infty(0,T;W^{-1,p'}(\Omega))}
		+\|u_\ell(T)\|_{L^\alpha(\Omega)}^{\alpha-1} + \|u_0\|_{L^\alpha(\Omega)}^{\alpha-1}.
	\end{aligned}
\end{equation*}

Thanks to Lemma~\ref{lem:5}, the boundedness of $\mu$, and the boundedness of the $p$-Laplacian, the latter term is bounded independently of $\ell$. By the arbitrariness of $v\in W_0^{1,p}(\Omega)$ with $\|v\|_{W_0^{1,p}(\Omega)}$, we deduce that:
\begin{equation*}
	\sup_{\substack{v \in W_0^{1,p}(\Omega)\\ \|v\|_{W_0^{1,p}(\Omega)}=1}} \dfrac{1}{\ell}\left|\left\langle\int_{0}^{T}\{u_\ell(t)\}^{-}\dd t,v\right\rangle_{W^{-1,p'}(\Omega), W_0^{1,p}(\Omega)}\right| \le C,\quad\textup{ for all }\ell>0,
\end{equation*}
for some $C>0$ independent of $\ell$ or, equivalently,
\begin{equation}
	\label{eq:2}
	\left\|\dfrac{1}{\ell}\int_{0}^{T} \{u_\ell(t)\}^{-} \dd t\right\|_{W^{-1,p'}(\Omega)} \le C,\quad\textup{ for all }\ell>0,
\end{equation}
for some $C>0$ independent of $\ell$. By Bochner's theorem (cf., e.g., Theorem~8.9 of~\cite{Leoni2017}), we have that the following inequality is always true:
\begin{equation*}
	\left\|\dfrac{1}{\ell}\int_{0}^{T} \{u_\ell(t)\}^{-} \dd t\right\|_{W^{-1,p'}(\Omega)} \le \dfrac{1}{\ell}\int_{0}^{T} \|\{u_\ell(t)\}^{-}\|_{W^{-1,p'}(\Omega)} \dd t.
\end{equation*}

We now show that the inverse inequality holds true too. Observe that, by Stampacchia's theorem, we have that $\{u_\ell(t)\}^{-} \in W_0^{1,p}(\Omega)$, so that $\{u_\ell(t)\}^{-} \ge 0$ in $\overline{\Omega}$, for a.e. $t\in (0,T)$. Therefore, we have:
\begin{equation*}
	\langle\{u_\ell(t)\}^{-},v\rangle_{W^{-1,p'}(\Omega), W_0^{1,p}(\Omega)}=\int_{\Omega} \{u_\ell(t)\}^{-} v\dd x \le \int_{\Omega} \{u_\ell(t)\}^{-} |v|\dd x, \quad\textup{ for all }v \in W_0^{1,p}(\Omega),
\end{equation*}
for a.e. $t \in (0,T)$. Therefore, for a.e. $t\in (0,T)$, the supremum
\begin{equation*}
	\sup_{\substack{v \in W_0^{1,p}(\Omega)\\ \|v\|_{W_0^{1,p}(\Omega)}=1}} \left|\langle\{u_\ell(t)\}^{-},v\rangle_{W^{-1,p'}(\Omega), W_0^{1,p}(\Omega)}\right|
\end{equation*}
is attained for functions $v\in W_0^{1,p}(\Omega)$ with unitary norm that are either greater or equal than zero in $\overline{\Omega}$, or less or equal than zero in $\overline{\Omega}$. In view of this remark, we have that:
\begin{equation}
	\label{key-1}
	\begin{aligned}
		&\sup_{\substack{v \in W_0^{1,p}(\Omega)\\ \|v\|_{W_0^{1,p}(\Omega)}=1}} \int_{0}^{T} \left|\left\langle\dfrac{\{u_\ell(t)\}^{-}}{\ell},v \right\rangle_{W^{-1,p'}(\Omega), W_0^{1,p}(\Omega)}\right| \dd t\\
		&=\sup_{\substack{v \in W_0^{1,p}(\Omega)\\ \|v\|_{W_0^{1,p}(\Omega)}=1\\v \ge 0 \textup{ in }\overline{\Omega}}}\int_{0}^{T} \left\langle\dfrac{\{u_\ell(t)\}^{-}}{\ell},v\right\rangle_{W^{-1,p'}(\Omega), W_0^{1,p}(\Omega)} \dd t\\
		&\le \sup_{\substack{v \in W_0^{1,p}(\Omega)\\ \|v\|_{W_0^{1,p}(\Omega)}=1}}
		\left|\int_{0}^{T} \left\langle\dfrac{\{u_\ell(t)\}^{-}}{\ell},v\right\rangle_{W^{-1,p'}(\Omega), W_0^{1,p}(\Omega)} \dd t\right|\\
		&\le \sup_{\substack{v \in W_0^{1,p}(\Omega)\\ \|v\|_{W_0^{1,p}(\Omega)}=1}} \int_{0}^{T} \left|\left\langle\dfrac{\{u_\ell(t)\}^{-}}{\ell},v \right\rangle_{W^{-1,p'}(\Omega), W_0^{1,p}(\Omega)}\right| \dd t,
	\end{aligned}
\end{equation}
so that the last two inequalities in~\eqref{key-1} are, actually, equalities.

Let us now show that:
\begin{equation}
	\label{eq:3}
	\begin{aligned}
		&\int_{0}^{T} \left(\sup_{\substack{v \in W_0^{1,p}(\Omega)\\ \|v\|_{W_0^{1,p}(\Omega)}=1}}\left|\left\langle \dfrac{\{u_\ell(t)\}^{-}}{\ell}, v \right\rangle_{W^{-1,p'}(\Omega), W_0^{1,p}(\Omega)}\right|\right) \dd t\\
		&\le \sup_{\substack{v \in W_0^{1,p}(\Omega)\\ \|v\|_{W_0^{1,p}(\Omega)}=1}} \int_{0}^{T} \left|\left\langle\dfrac{\{u_\ell(t)\}^{-}}{\ell},v\right\rangle_{W^{-1,p'}(\Omega), W_0^{1,p}(\Omega)}\right| \dd t.
	\end{aligned}
\end{equation}

Let $\{v_j\}_{j=1}^\infty \subset W_0^{1,p}(\Omega)$, $\|v_k\|_{W_0^{1,p}(\Omega)}=1$ be such that the following convergence in $\mathbb{R}$ holds
\begin{equation}
	\label{eq:4}
	\lim_{j\to\infty}\left|\left\langle\dfrac{\{u_\ell(t)\}^{-}}{\ell},v_j\right\rangle_{W^{-1,p'}(\Omega), W_0^{1,p}(\Omega)}\right|
	=
	\sup_{\substack{v \in W_0^{1,p}(\Omega)\\ \|v\|_{W_0^{1,p}(\Omega)}=1}} \left|\left\langle\dfrac{\{u_\ell(t)\}^{-}}{\ell},v\right\rangle_{W^{-1,p'}(\Omega), W_0^{1,p}(\Omega)}\right|,
\end{equation}
for a.e. $t \in (0,T)$. Observe that, for each $j \ge 1$, we always have:
\begin{equation}
	\label{eq:5}
	\begin{aligned}
		&\int_{0}^{T} \left|\left\langle\dfrac{\{u_\ell(t)\}^{-}}{\ell},v_j\right\rangle_{W^{-1,p'}(\Omega), W_0^{1,p}(\Omega)}\right|\dd t\\
		&\le \sup_{\substack{v \in W_0^{1,p}(\Omega)\\ \|v\|_{W_0^{1,p}(\Omega)}=1}} \int_{0}^{T}\left|\left\langle\dfrac{\{u_\ell(t)\}^{-}}{\ell},v\right\rangle_{W^{-1,p'}(\Omega), W_0^{1,p}(\Omega)}\right| \dd t.
	\end{aligned}
\end{equation}

Since the convergence in~\eqref{eq:4} holds for a.e. $t \in (0,T)$, we are in a position to apply Fatou's lemma (cf., e.g., \cite{Royden1988}). A subsequent application of~\eqref{eq:5} gives:
\begin{equation*}
	\begin{aligned}
		&\int_{0}^{T} \sup_{\substack{v \in W_0^{1,p}(\Omega)\\ \|v\|_{W_0^{1,p}(\Omega)}=1}} \left|\left\langle\dfrac{\{u_\ell(t)\}^{-}}{\ell},v\right\rangle_{W^{-1,p'}(\Omega), W_0^{1,p}(\Omega)}\right| \dd t\\
		&\le \liminf_{j \to\infty} \int_{0}^{T} \left|\left\langle\dfrac{\{u_\ell(t)\}^{-}}{\ell},v_j\right\rangle_{W^{-1,p'}(\Omega), W_0^{1,p}(\Omega)}\right| \dd t\\
		&\le \sup_{\substack{v \in W_0^{1,p}(\Omega)\\ \|v\|_{W_0^{1,p}(\Omega)}=1}} \int_{0}^{T} \left|\left\langle\dfrac{\{u_\ell(t)\}^{-}}{\ell},v\right\rangle_{W^{-1,p'}(\Omega), W_0^{1,p}(\Omega)}\right| \dd t,
	\end{aligned}
\end{equation*}
and this proves~\eqref{eq:3}. A a result of~\eqref{key-1}, \eqref{eq:3} and the properties of the Bochner integral (cf., e.g., Theorem~8.13 of~\cite{Leoni2017}), we have:
\begin{equation*}
	\begin{aligned}
		&\int_{0}^{T} \left\|\dfrac{\{u_\ell(t)\}^{-}}{\ell}\right\|_{W^{-1,p'}(\Omega)}\dd t
		=\int_{0}^{T} \sup_{\substack{v \in W_0^{1,p}(\Omega)\\ \|v\|_{W_0^{1,p}(\Omega)}=1}} \left|\left\langle \dfrac{\{u_\ell(t)\}^{-}}{\ell},v\right\rangle_{W^{-1,p'}(\Omega), W_0^{1,p}(\Omega)}\right|\dd t\\
		&\le \sup_{\substack{v \in W_0^{1,p}(\Omega)\\ \|v\|_{W_0^{1,p}(\Omega)}=1}} \int_{0}^{T} \left|\left\langle\dfrac{\{u_\ell(t)\}^{-}}{\ell},v\right\rangle_{W^{-1,p'}(\Omega), W_0^{1,p}(\Omega)}\right| \dd t\\
		&=\sup_{\substack{v \in W_0^{1,p}(\Omega)\\ \|v\|_{W_0^{1,p}(\Omega)}=1}} \left|\left\langle\int_{0}^{T}\dfrac{\{u_\ell(t)\}^{-}}{\ell} \dd t,v\right\rangle_{W^{-1,p'}(\Omega), W_0^{1,p}(\Omega)}\right|
		=\left\|\int_{0}^{T}\dfrac{\{u_\ell(t)\}^{-}}{\ell} \dd t\right\|_{W^{-1,p'}(\Omega)}.
	\end{aligned}
\end{equation*}

In conclusion, we have shown that:
\begin{equation*}
	\left\|\dfrac{1}{\ell}\int_{0}^{T} \{u_\ell(t)\}^{-} \dd t\right\|_{W^{-1,p'}(\Omega)} \le \dfrac{1}{\ell}\int_{0}^{T} \|\{u_\ell(t)\}^{-}\|_{W^{-1,p'}(\Omega)} \dd t.
\end{equation*}

Combining the latter with Bochner's theorem (cf., e.g., Theorem 8.13 of~\cite{Leoni2017}), we obtain:
\begin{equation}
	\label{eq:6}
	\left\|\dfrac{1}{\ell}\int_{0}^{T} \{u_\ell(t)\}^{-} \dd t\right\|_{W^{-1,p'}(\Omega)} = \dfrac{1}{\ell}\int_{0}^{T} \|\{u_\ell(t)\}^{-}\|_{W^{-1,p'}(\Omega)} \dd t.
\end{equation}

Combining~\eqref{eq:2} and~\eqref{eq:6}, we obtain
\begin{equation*}
	\dfrac{1}{\ell}\int_{0}^{T} \left\|\{u_\ell(t)\}^{-}\right\|_{W^{-1,p'}(\Omega)} \dd t \le C,\quad\textup{ for all }\ell>0,
\end{equation*}
for some $C>0$ independent of $\ell$ or, equivalently, that
\begin{equation}
	\label{eq:7}
	\left\{\dfrac{\{u_\ell\}^{-}}{\ell}\right\}_{\ell>0} \textup{ is bounded in } L^1(0,T;W^{-1,p'}(\Omega)) \textup{ independently of }\ell.
\end{equation}

Therefore, we can derive the following estimate plugging~\eqref{eq:7} into~\eqref{eq:1}:
\begin{equation}
	\label{eq:8}
	\begin{aligned}
		&\sup_{\substack{v \in W_0^{1,p}(\Omega)\\ \|v\|_{W_0^{1,p}(\Omega)}=1}} \left|\left\langle \dfrac{\dd}{\dd t}\left(|u_\ell|^{\alpha-2} u_\ell\right)(t), v \right\rangle_{W^{-1,p'}(\Omega), W_0^{1,p}(\Omega)}\right|\\
		&\le \sup_{\substack{v \in W_0^{1,p}(\Omega)\\ \|v\|_{W_0^{1,p}(\Omega)}=1}} \left|\left\langle\mu |\nabla u_\ell(t)|^{p-2} \nabla u_\ell(t),v\right\rangle_{W^{-1,p'}(\Omega), W_0^{1,p}(\Omega)}\right|\\
		&\quad +\dfrac{1}{\ell}\|\{u_\ell(t)\}^{-}\|_{W^{-1,p'}(\Omega)} +\|\tilde{a}(t)\|_{W^{-1,p'}(\Omega)}, \textup{ for a.e. }t\in (0,T).
	\end{aligned}
\end{equation}

Applying the boundedness of $\mu$ assumed in~($H$\ref{H2}), the boundedness of $\{u_\ell\}_{\ell>0}$ established in~\eqref{conv-proc-kappa}, the boundedness of $\tilde{a}$ assumed in~($H$\ref{H4}), the boundedness of the $p$-Laplacian, \eqref{eq:7} and the monotonicity of the integral to~\eqref{eq:8} give:
\begin{equation}
	\label{eq:9}
	\int_{0}^{T}\sup_{\substack{v \in W_0^{1,p}(\Omega)\\ \|v\|_{W_0^{1,p}(\Omega)}=1}} \left|\left\langle \dfrac{\dd}{\dd t}\left(|u_\ell|^{\alpha-2} u_\ell\right)(t), v \right\rangle_{W^{-1,p'}(\Omega), W_0^{1,p}(\Omega)}\right|\dd t \le C,
\end{equation}
for some $C>0$ independent of $\ell$. This means that
\begin{equation*}
	\left\{\dfrac{\dd}{\dd t}\left(|u_\ell|^{\alpha-2} u_\ell\right)\right\}_{\ell>0} \textup{ is bounded in } L^1(0,T;W^{-1,p'}(\Omega)).
\end{equation*}

Therefore, by the Dinculeanu-Zinger theorem (Lemma~\ref{ds}), there exists a vector-valued measure $\tilde{w}_t \in \mathcal{M}([0,T];W^{-1,p'}(\Omega))$ such that:
\begin{equation}
	\label{measure-conv}
	\dfrac{\dd}{\dd t}\left(|u_\ell|^{\alpha-2} u_\ell\right) \wsc \tilde{w}_t,\quad\textup{ in } \mathcal{M}([0,T];W^{-1,p'}(\Omega)) \textup{ as }\ell \to 0^+.
\end{equation}

It can be observed that the vector-valued measure $\tilde{w}_t$ appearing in~\eqref{measure-conv} can be interpreted as
\begin{equation*}
	\dfrac{\dd}{\dd t}\left(|u|^{\alpha-2} u\right).
\end{equation*}

To see this, let $v \in W_0^{1,p}(\Omega)$ and let $\varphi \in \mathcal{D}(0,T)$. By the fourth convergence in~\eqref{conv-proc-kappa} and an application of Proposition~2.2.34 of~\cite{GasPap06}, we have that:
\begin{equation}
	\label{key-2}
	\begin{aligned}
		&\langle\langle \tilde{w}_t , \varphi v \rangle\rangle_{\mathcal{M}([0,T];W^{-1,p'}(\Omega)), \mathcal{C}^0([0,T];W_0^{1,p}(\Omega))}\\
		&=\lim_{\ell \to0^+}\int_{0}^{T} \left\langle \dfrac{\dd}{\dd t}\left(|u_\ell|^{\alpha-2} u_\ell\right) , v \right\rangle_{W^{-1,p'}(\Omega), W_0^{1,p}(\Omega)} \varphi(t) \dd t\\
		&=-\lim_{\ell \to0^+} \int_{0}^{T} \int_{\Omega} \left(|u_\ell|^{\alpha-2} u_\ell\right) v\dd x\varphi'(t)\dd t\\
		&=-\int_{0}^{T} \int_{\Omega} \left(|u|^{\alpha-2} u\right) v\varphi'(t)\dd x\dd t.
	\end{aligned}
\end{equation}

Keeping the convergence~\eqref{measure-conv} in mind, as well as the fourth convergence in~\eqref{conv-proc-kappa}, we are in position to carry out the same compactness argument as in Theorem~\ref{thm:3} in order to apply Dubinskii's theorem (Lemma~\ref{Dub}) with $A_0=L^{(\alpha'-1)p}(\Omega)$, $A_1=W^{-1,p'}(\Omega)$, $q_0=(\alpha'-1)p>1$ and $q_1=1$. The monotonicity of the mapping $\xi \mapsto |\xi|^{\alpha-2}\xi$ finally gives us:
\begin{equation}
	\label{eq:10}
	|u_\ell|^{\alpha-2} u_\ell \to w=|u|^{\alpha-2} u, \quad \textup{ in } L^{(\alpha'-1)p}(0,T;L^{(\alpha'-1)p}(\Omega)),
\end{equation}
and $w=|u|^{\alpha-2} u \in L^\infty(0,T;L^{\alpha'}(\Omega))$ by the fourth convergence in~\eqref{conv-proc-kappa}.

Thanks to Lemma~\ref{lem:6} and the Dinculeanu-Zinger theorem (Lemma~\ref{ds}), we have that the following chain of embeddings holds:
\begin{equation*}
	L^1(0,T;W^{-1,p'}(\Omega)) \hookrightarrow \left(\mathcal{C}^0([0,T];W_0^{1,p}(\Omega))\right)^\ast \simeq \mathcal{M}([0,T];W^{-1,p'}(\Omega)).
\end{equation*}

Thanks to Lemma~\ref{lem:5} and Lemma~\ref{lem:8}, we have:
\begin{equation}
	\label{eq:11}
	\begin{aligned}
		&\int_{0}^{T} \left\langle\dfrac{\dd}{\dd t}\left(|u_\ell(t)|^{\alpha-2} u_\ell(t)\right),u_\ell(t)\right\rangle_{W^{-1,p'}(\Omega), W_0^{1,p}(\Omega)} \dd t
		=\dfrac{\|u_\ell(T)\|_{L^\alpha(\Omega)}^\alpha}{\alpha'}-\dfrac{\|u_0\|_{L^\alpha(\Omega)}^\alpha}{\alpha'}\\
		&\to \dfrac{\|u(T)\|_{L^\alpha(\Omega)}^\alpha}{\alpha'}-\dfrac{\|u_0\|_{L^\alpha(\Omega)}^\alpha}{\alpha'},\quad\textup{ as }\ell \to 0^+.
	\end{aligned}
\end{equation}

Specializing $v=u_\ell - u$ in~\eqref{eq:1}, we obtain:
\begin{equation}
	\label{eq:12}
	\begin{aligned}
		&\int_{0}^{T} \left\langle \dfrac{\dd}{\dd t}\left(|u_\ell|^{\alpha-2} u_\ell\right), u_\ell - u \right\rangle_{W^{-1,p'}(\Omega), W_0^{1,p}(\Omega)}\dd t\\
		&\quad+\int_{0}^{T} \int_{\Omega} \mu(x) |\nabla u_\ell(t)|^{p-2} \nabla u_\ell(t) \cdot \nabla(u_\ell(t) - u(t)) \dd x \dd t\\
		&\quad-\dfrac{1}{\ell} \int_{0}^{T} \int_{\Omega} \{u_\ell(t)\}^{-} (u_\ell(t)-u(t)) \dd x \dd t
		=\int_{0}^{T} \int_{\Omega} \tilde{a}(t) (u_\ell(t)-u(t))\dd x \dd t.
	\end{aligned}
\end{equation}

Thanks to the convergence process~\eqref{conv-proc-kappa}, we have that:
\begin{equation}
\label{eq:13}
\left(\sup_{t\in (0,T)} \|u_\ell(t)\|_{W^{1,p}_0(\Omega)}\right)_{\ell>0} \textup{ is bounded independently of }\ell>0.
\end{equation}

For each $i \in \{1,2\}$ and for each $w \in L^p(0,T;W_0^{1,p}(\Omega))$, define the mapping $\widetilde{\partial_i w}:(0,T) \to L^p(\Omega)$ by $\widetilde{\partial_i w}(t):=\partial_i(w(t))$, a.e. in $(0,T)$. Note that the definition of this mapping is well-posed thanks to the assumed regularity of $w$. In particular, it is easy to observe that the boundedness of the sequence $(u_\ell)_{\ell>0}$ in the convergence process~\eqref{conv-proc-kappa} in turn implies that
$$
\int_0^T \int_{\Omega} \left|\widetilde{\partial_i u_\ell}(t)\right|^p\dd x \dd t
=\int_0^T \int_{\Omega} \left|\partial_i (u_\ell(t))\right|^p\dd x \dd t
\le \|u_\ell\|_{L^p(0,T;W_0^{1,p}(\Omega))}^p,
$$
so that the sequence $\left(\widetilde{\partial_i u_\ell}\right)_{\ell>0}$ is bounded in $L^p(0,T;L^p(\Omega))$ independently of $\ell$. Therefore, up to passing to subsequences, the sequence $\left(\widetilde{\partial_i u_\ell}\right)_{\ell>0}$ weakly converges in $L^p(0,T;L^p(\Omega))$ as $\ell \to 0^+$.
Thanks to the convergence process~\eqref{conv-proc-kappa}, is is possible to identify the weak limit of the aforementioned sequence with the function $\widetilde{\partial_i u}$.
Since $|\cdot|$ denotes the $p$-norm of $\mathbb{R}^2$, we obtain that the remarks made above lead to
\begin{equation*}
\begin{aligned}
&\int_{0}^{T}\int_{\Omega} |\nabla(u(t))|^p \dd x \dd t=\sum_{i=1}^{2}\int_{0}^{T} \|\partial_i(u(t))\|_{L^p(\Omega)}^p \dd t
=\sum_{i=1}^{2}\int_{0}^{T} \|(\widetilde{\partial_i u})(t)\|_{L^p(\Omega)}^p \dd t\\
&\le \liminf_{\ell\to0^+}\left(\sum_{i=1}^{2}\int_{0}^{T} \|(\widetilde{\partial_i u_\ell})(t)\|_{L^p(\Omega)}^p \dd t\right)
=\liminf_{\ell\to0^+}\left(\sum_{i=1}^{2}\int_{0}^{T} \|\partial_i(u_\ell(t))\|_{L^p(\Omega)}^p \dd t\right)\\
&=\liminf_{\ell\to0^+}\int_{0}^{T} \int_{\Omega} |\nabla(u_\ell(t))|^p \dd x \dd t,
\end{aligned}
\end{equation*}
where the last equality is due to the fact that the assumption $2.8 \le p \le 5$ implies that the mapping $\xi \ge 0 \mapsto \xi^p$ is monotonically increasing (viz., e.g., Lemma~\ref{lem:4}). In conclusion, the latter series of computations shows that:
\begin{equation*}
\int_{0}^{T}\int_{\Omega} |\nabla(u(t))|^p \dd x \dd t \le \liminf_{\ell\to0^+}\int_{0}^{T} \int_{\Omega} |\nabla(u_\ell(t))|^p \dd x \dd t.
\end{equation*}

Since $\mu$ is a positive and smooth weight, we have that:
\begin{equation}
\label{eq:16}
\begin{aligned}
&\limsup_{\ell\to0^+}\left(-\int_{0}^{T} \int_{\Omega} \mu(x) |\nabla (u_\ell(t))|^p \dd x \dd t\right)
=-\liminf_{\ell\to0^+}\int_{0}^{T} \int_{\Omega} \mu(x) |\nabla (u_\ell(t))|^p \dd x \dd t\\
&\le -\int_{0}^{T}\int_{\Omega} \mu(x) |\nabla(u(t))|^p \dd x \dd t.
\end{aligned}
\end{equation}

Define the operator $A:L^1(0,T;W_0^{1,p}(\Omega)) \to L^\infty(0,T;W^{-1,p'}(\Omega))$ as follows:
\begin{equation*}
\begin{aligned}
\langle Au , v \rangle_{L^\infty(0,T;W^{-1,p'}(\Omega)),L^1(0,T;W_0^{1,p}(\Omega))}:=\int_{0}^{T} \int_{\Omega} \mu(x) |\nabla(u(t))|^{p-2} \nabla(u(t)) \cdot \nabla(v(t)) \dd x \dd t,
\end{aligned}
\end{equation*}
for all $u,v \in L^1(0,T;W_0^{1,p}(\Omega))$. 
For any $v \in L^1(0,T;W^{1,p}_0(\Omega))$, an application of H\"older's inequality and~\eqref{eq:13} gives that:
\begin{equation*}
\begin{aligned}
&\left|\int_{0}^{T} \int_{\Omega} \mu(x) |\nabla (u_\ell(t))|^{p-2} \nabla(u_\ell(t)) \cdot\nabla(v(t)) \dd x \dd t\right|\\
&=\left|\int_{0}^{T} \sum_{i=1}^2\int_{\Omega} \mu(x) |\nabla (u_\ell(t))|^{p-2} \partial_i(u_\ell(t)) \partial_i(v(t)) \dd x \dd t\right|\\
& \le \mu_2\left|\sum_{i=1}^2\int_{0}^{T} \left(\int_{\Omega}\left| |\nabla(u_\ell(t))|^{p-2} \partial_i(u_\ell(t))\right|^{p/(p-1)} \dd x\right)^{(p-1)/p} \left(\int_{\Omega} |\partial_i(v(t))|^p \dd x\right)^{1/p} \dd t\right|\\
& \le \mu_2\left|\sum_{i=1}^2\int_{0}^{T} \left(\int_{\Omega}|\nabla(u_\ell(t))|^{p(p-2)/(p-1)} |\partial_i(u_\ell(t))|^{p/(p-1)} \dd x\right)^{(p-1)/p} \left(\int_{\Omega} |\partial_i(v(t))|^p \dd x\right)^{1/p} \dd t\right|\\
&\le 2 \mu_2 \left|\int_{0}^{T} \|\nabla(u_\ell(t))\|_{\bm{L}^p(\Omega)}^{p-1} \|\nabla(v(t))\|_{\bm{L}^p(\Omega)} \dd t\right| \le2 \mu_2 \left(\sup_{t\in (0,T)} \|u_\ell(t)\|_{W_0^{1,p}(\Omega)}\right)^{p-1} \|v\|_{L^1(0,T;W_0^{1,p}(\Omega))}.
\end{aligned}
\end{equation*}

Therefore, we obtain that the operator $A$ is bounded, in the sense that it maps bounded sets of $L^1(0,T;W_0^{1,p}(\Omega))$ into bounded sets of $L^\infty(0,T;W^{-1,p'}(\Omega))$ (cf., e.g., \cite{Lions1969}).

In light of~\eqref{eq:13} we have that, up to passing to a suitable subsequence, the following convergence occurs:
\begin{equation*}
Au_\ell \wsc g_\mu,\quad\textup{ in } L^\infty(0,T;W^{-1,p'}(\Omega)) \textup{ as }\ell\to 0^+.
\end{equation*}

In particular, for any test function $v \in \mathcal{C}^0([0,T];W^{1,p}_0(\Omega))$, an application of Theorem~2.2.9 of~\cite{GasPap06} gives:
\begin{equation}
\label{c:1}
\int_{0}^{T} \int_{\Omega} \mu(x) |\nabla(u_\ell(t))|^{p-2} \nabla(u_\ell(t)) \cdot\nabla(v(t))\dd x \dd t \to \int_{0}^{T} \langle g_\mu(t), v(t) \rangle_{W^{-1,p'}(\Omega),W^{1,p}_0(\Omega)} \dd t.
\end{equation}

Let $v \in \mathcal{C}^0([0,T];W^{1,p}_0(\Omega))$ be such that $v(t) \ge 0$ in $\overline{\Omega}$ (since $p>2$), and multiply equation~\eqref{eq:1} by $(v(t)-u_\ell(t))$ in the sense of the duality between $W^{-1,p'}(\Omega)$ and $W^{1,p}_0(\Omega)$. A subsequent integration in $(0,T)$ gives:
\begin{equation}
\label{preliminary}
\begin{aligned}
&\int_{0}^{T} \left\langle \dfrac{\dd}{\dd t}(|u_\ell|^{\alpha-2} u_\ell), v(t)-u_\ell(t) \right\rangle_{W^{-1,p'}(\Omega),W^{1,p}_0(\Omega)} \dd t\\
&\quad+\int_{0}^{T} \int_{\Omega} \mu(x) |\nabla(u_\ell(t))|^{p-2} \nabla(u_\ell(t)) \cdot (\nabla(v(t))-\nabla(u_\ell(t))) \dd x \dd t\\
&\ge \int_{0}^{T} \int_{\Omega} \tilde{a}(t) (v(t)-u_\ell(t)) \dd x\dd t,
\end{aligned}
\end{equation}
where the inequality holds thanks to the monotonicity of $-\{\cdot\}^{-}$ (cf., e.g., Lemma~\ref{lem:4}).

Passing to the $\limsup$ as $\ell\to0^+$ in~\eqref{preliminary}, an application of~\eqref{eq:11} and~\eqref{eq:16} and an application of~\eqref{c:1} give that the limit $u$ is a solution to the following \emph{abstract limit problem}:
\begin{equation*}
\begin{aligned}
&\left\langle\left\langle\dfrac{\dd}{\dd t}(|u|^{\alpha-2} u), v\right\rangle\right\rangle_{(\mathcal{C}^0([0,T];W^{1,p}_0(\Omega)))^\ast,\mathcal{C}^0([0,T];W^{1,p}_0(\Omega))}
-\dfrac{\|u(T)\|_{L^\alpha(\Omega)}^\alpha}{\alpha'}+\dfrac{\|u_0\|_{L^\alpha(\Omega)}^\alpha}{\alpha'}\\
&\quad+\int_{0}^{T} \langle g_\mu(t),v(t)\rangle_{W^{-1,p'}(\Omega),W^{1,p}_0(\Omega)} \dd t
-\int_{0}^{T} \int_{\Omega} \mu(x) |\nabla(u(t))|^p \dd x \dd t\\
&\ge \int_{0}^{T} \int_{\Omega} \tilde{a}(t) (v(t)-u(t)) \dd x\dd t.
\end{aligned}
\end{equation*}

In the end, we are in position to write down the model governing the evolution of the thickness of a shallow ice sheet, and to assert that this model admits at least one solution.
This is the main result of this article, the proof of which has been given above.

\begin{theorem}
	\label{thm:4}
	Let $T>0$, let $\Omega \subset \mathbb{R}^2$ be a domain, let $p$ be as in Section~\ref{Sec:2}, and let $\alpha$ be as in~\eqref{alpha}. 
	Assume that ($H$\ref{H1})--($H$\ref{H4}) hold.	
	Then, the family $\{u_\ell\}_{\ell>0}$ of solutions of Problem~\ref{Pkappa} generated by Theorem~\ref{thm:3} satisfies the following convergences as $\ell \to 0^+$:
	\begin{equation*}
		\begin{aligned}
			u_\ell \wsc u, &\textup{ in } L^\infty(0,T;W_0^{1,p}(\Omega)),\\
			|u_\ell|^\frac{\alpha-2}{2} u_\ell \wsc |u|^\frac{\alpha-2}{2} u, &\textup{ in } L^\infty(0,T;L^2(\Omega)),\\
			\dfrac{\dd}{\dd t}\left(|u_\ell|^\frac{\alpha-2}{2} u_\ell\right) \rightharpoonup \dfrac{\dd}{\dd t}\left(|u|^\frac{\alpha-2}{2} u\right), &\textup{ in } L^2(0,T;L^2(\Omega)),\\
			|u_\ell|^{\alpha-2} u_\ell \wsc |u|^{\alpha-2} u, &\textup{ in } L^\infty(0,T;L^{\alpha'}(\Omega)),\\
			\dfrac{\dd}{\dd t}\left(|u_\ell|^{\alpha-2} u_\ell\right) \wsc \dfrac{\dd}{\dd t}\left(|u|^{\alpha-2} u\right), &\textup{ in } \mathcal{M}([0,T];W^{-1,p'}(\Omega)).
		\end{aligned}
	\end{equation*}

	Moreover, the limit $u$ is a solution to the following variational problem:
	\begin{customprob}{$\mathcal{P}$}
		\label{P}
		Find $u \in \mathcal{K}:=\{v \in L^\infty(0,T;W_0^{1,p}(\Omega)); v(t) \in K \textup{ a.e. in } (0,T)\}$ such that
		\begin{align*}
			u &\in L^\infty(0,T;W_0^{1,p}(\Omega)),\\
			\dfrac{\dd}{\dd t}\left(|u|^\frac{\alpha-2}{2} u\right) &\in L^2(0,T;L^2(\Omega)),\\
			\dfrac{\dd}{\dd t}\left(|u|^{\alpha-2}u\right) &\in \mathcal{M}([0,T];W^{-1,p'}(\Omega)),
		\end{align*}
		and find $g_\mu \in L^\infty(0,T;W^{-1,p'}(\Omega))$ satisfying the following doubly-nonlinear parabolic variational inequality
		\begin{equation*}
		\begin{aligned}
		&\left\langle\left\langle\dfrac{\dd}{\dd t}(|u|^{\alpha-2} u), v\right\rangle\right\rangle_{(\mathcal{C}^0([0,T];W^{1,p}_0(\Omega)))^\ast,\mathcal{C}^0([0,T];W^{1,p}_0(\Omega))}
		-\dfrac{\|u(T)\|_{L^\alpha(\Omega)}^\alpha}{\alpha'}+\dfrac{\|u_0\|_{L^\alpha(\Omega)}^\alpha}{\alpha'}\\
		&\quad+\int_{0}^{T} \langle g_\mu(t),v(t)\rangle_{W^{-1,p'}(\Omega),W^{1,p}_0(\Omega)} \dd t
		-\int_{0}^{T} \int_{\Omega} \mu(x) |\nabla(u(t))|^p \dd x \dd t\\
		&\ge \int_{0}^{T} \int_{\Omega} \tilde{a}(t) (v(t)-u(t)) \dd x\dd t,
		\end{aligned}
		\end{equation*}
		for all $v\in \mathcal{C}^0([0,T];W_0^{1,p}(\Omega))$ such that $v(t)\ge 0$ in $\overline{\Omega}$ for all $t\in [0,T]$,
		as well as the following initial condition
		$$
		u(0)=u_0,
		$$
		for some prescribed $u_0 \in K$.
	\end{customprob}
	\qed
\end{theorem}

We conclude the article with a remark where we propose a sufficient condition to identify the limit $g_\mu$ as well as the vector-valued measure.

\begin{remark}
	In the special case where $|u|^{\alpha-2} u \in W^{1,1}(0,T;W^{-1,p'}(\Omega))$ and $u \in \mathcal{C}^0([0,T];W_0^{1,p}(\Omega))$ (these constitute another \emph{regularity assumption} on a solution to Problem~\ref{P}) we see that the following convergence holds thanks to~\eqref{key-2}:
	\begin{equation}
		\label{eq:17}
		\lim_{\ell\to0^+}\int_{0}^{T} \left\langle \dfrac{\dd}{\dd t}\left(|u_\ell|^{\alpha-2} u_\ell\right), u(t) \right\rangle_{W^{-1,p'}(\Omega), W_0^{1,p}(\Omega)}\dd t
		=\dfrac{\|u(T)\|_{L^\alpha(\Omega)}^\alpha}{\alpha'}-\dfrac{\|u_0\|_{L^\alpha(\Omega)}^\alpha}{\alpha'}.
	\end{equation}
	
	Combining~\eqref{eq:11} and~\eqref{eq:17} and the first convergence in~\eqref{conv-proc-kappa}, we obtain that~\eqref{eq:12} gives:
	\begin{equation}
		\label{eq:18}
		\limsup_{\ell\to 0^+} \int_{0}^{T} \int_{\Omega}\mu(x) |\nabla u_\ell(t)|^{p-2} \nabla u_\ell(t) \cdot \nabla(u_\ell(t)-u(t)) \dd x\dd t \le 0.
	\end{equation}
	
	Since the $p$-Laplacian in divergence form is pseudo-monotone (cf., e.g., Proposition~2.5 of~\cite{Lions1969}), as it is hemi-continuous, strictly monotone and bounded, an application of~\eqref{eq:18} gives that
	\begin{equation}
		\label{eq:19}
		\begin{aligned}
			&\int_{0}^{T} \int_{\Omega}\mu(x) |\nabla u(t)|^{p-2} \nabla u(t) \cdot \nabla(u(t)-v(t)) \dd x\dd t\\
			&\le \liminf_{\ell\to 0^+} \int_{0}^{T} \int_{\Omega}\mu(x) |\nabla u_\ell(t)|^{p-2} \nabla u_\ell(t) \cdot \nabla(u_\ell(t)-v(t)) \dd x\dd t,
		\end{aligned}
	\end{equation}
	for all $v \in L^2(0,T;W_0^{1,p}(\Omega))$. Observe that the latter space is suitable for treating pseudo-monotonicity, as it is reflexive.
	
	Let us now consider an arbitrary function $v \in \mathcal{C}^0([0,T];W_0^{1,p}(\Omega))$ such that $v(t)\ge 0$ in $\overline{\Omega}$ for all $t \in [0,T]$.
	Evaluating the variational equations~\eqref{eq:1} at $w_\ell:=v-u_\ell$, we obtain:
	\begin{equation}
		\label{final-step}
		\begin{aligned}
			&\int_{0}^{T} \left\langle \dfrac{\dd }{\dd t}\left(|u_\ell(t)|^{\alpha-2} u_\ell(t)\right) , v(t)-u_\ell(t)\right\rangle_{W^{-1,p'}(\Omega), W_0^{1,p}(\Omega)} \dd t\\
			&\quad + \int_{0}^{T} \int_{\Omega} \mu(x) |\nabla u_\ell(t)|^{p-2} \nabla u_\ell(t) \cdot \nabla(v(t)-u_\ell(t)) \dd x \dd t\\
			&\quad-\dfrac{1}{\ell}\int_{0}^{T} \int_{\Omega} \{u_\ell(t)\}^{-} (v(t)-u_\ell(t)) \dd x \dd t\\
			&=\int_{0}^{T} \int_{\Omega} \tilde{a}(t) (v(t)-u_\ell(t)) \dd x \dd t.
		\end{aligned}
	\end{equation}
	
	Let $\ell\to 0^+$ in~\eqref{final-step}, exploit the convergence process~\eqref{conv-proc-kappa}, inequality~\eqref{eq:19} and the fact that the integral associated with the negative part is $\le 0$, since $u(t) \ge 0$ a.e. in $\overline{\Omega}$ for a.e. $t \in (0,T)$, gives that $u$ is a solution to the following variational inequality
	\begin{equation*}
		\begin{aligned}
			&\int_{0}^{T} \int_{\Omega} \dfrac{\dd }{\dd t}\left(|u(t)|^{\alpha-2} u(t)\right) (v(t)-u_\ell(t)) \dd x \dd t\\
			&\quad+ \int_{0}^{T} \int_{\Omega} \mu(x) |\nabla u(t)|^{p-2} \nabla u(t) \cdot \nabla(v(t)-u(t)) \dd x \dd t\\
			&\ge \int_{0}^{T} \int_{\Omega} \tilde{a}(t) (v(t)-u(t))\dd x \dd t+\dfrac{\|u(T)\|_{L^\alpha(\Omega)}^\alpha}{\alpha'}-\dfrac{\|u_0\|_{L^\alpha(\Omega)}^\alpha}{\alpha'},
		\end{aligned}
	\end{equation*}
for all $v\in \mathcal{C}^0([0,T];W_0^{1,p}(\Omega))$ such that $v(t)\ge 0$ in $\overline{\Omega}$ for all $t\in [0,T]$,
as well as the following initial condition
$$
u(0)=u_0,
$$
for some prescribed $u_0 \in K$.	
\bqed
\end{remark}

\begin{remark}	
	In the special case where there exists a constant $C>0$ independent of $\ell$ such that
	\begin{equation}
		\label{SC2'}
		\dfrac{1}{\ell}\|\{u_\ell\}^{-}\|_{L^\infty(0,T;W^{-1,p'}(\Omega))} \le C,
	\end{equation}
	we do not need to resort to vector-valued measures. Indeed, the regularity condition~\eqref{SC2'} lets us obtain the following sharper result (cf., \eqref{eq:1}):
	\begin{equation*}
	\left\{\dfrac{\dd}{\dd t}\left(|u_\ell|^{\alpha-2} u_\ell\right)\right\}_{\ell>0} \textup{ is bounded in } L^\infty(0,T;W^{-1,p'}(\Omega)) \textup{ independently of $\ell$}.
	\end{equation*}
	
	The condition~\eqref{SC2'} means that the melting regime is at some point overcome by the accumulation regime. One can expect this type of behaviour by virtue of the cyclic accumulation and ablation phases ice sheets undergo.
	\bqed
\end{remark}


\section*{Conclusions}

In this article we formulated a time-dependent model governing the evolution of the thickness of a shallow grounded ice sheet undergoing regimes of melting and ablation. The shallow ice sheet varying height, that is the unknown the model is described in terms of, is subjected to obey the constraint of being nonnegative. For this reason, the problem under consideration can be regarded as an obstacle problem.

First, we recovered the formal model, based on Glen's power law. This passage was inspired by the modelling for the static case of Jouvet \& Bueler~\cite{JouvBuel2012}.

Second, we incorporated the constraint, in the form of a monotone term, in the model. By so doing it was possible to write down the variational formulation of the model under consideration in terms of a variational equation posed over a vector space.
The existence of solutions for the penalized model was established thanks to the Dubinskii compactness theorem and a series of preliminary results, that we stated in the form of Lemmas.

Third, we let the penalty parameter approach zero and we recovered the actual model, which takes the form of a variational inequality tested over a nonempty, closed, and convex subset of the space $\mathcal{C}^0([0,T];W_0^{1,p}(\Omega))$.

It is worth mentioning that the concept of solution we are considering in this paper is different from the one discussed in Definition~3.1 on page~688 of~\cite{Diaz2002}. In our context, the modelling did not lead us to consider the analogue of the function $j$ appearing in Definition~3.1 on page~688 of~\cite{Diaz2002}, and we did not require our weak solutions to satisfy the identity indicated in Definition~3.1 on page~688 of~\cite{Diaz2002}. Therefore, differently from~\cite{Diaz2002}, the proof of the existence result for the concept of solution we considered in this paper required the exploitation of a compactness argument based on the Dubinskii's compactness lemma as well as several other brand new preparatory inequalities. 
		

\section*{Acknowledgements}

The authors would like to express their most sincere gratitude to Professor Ed Bueler (University of Alaska, Fairbanks) for the insightful comments on the physics of the problem and for the suggested improvements.

The authors would like to express their most sincere gratitude to the Anonymous Referees for the insightful comments and the suggested improvements.

This article was partly supported by the Research Fund of Indiana University.

\bibliographystyle{abbrvnat} 
\bibliography{references.bib}	

\end{document}